\documentclass[11pt]{amsart}
\usepackage{graphicx, amssymb, amsmath, amsthm}
\usepackage{epsfig}
\usepackage{hyperref}
\usepackage{mathrsfs}

\usepackage{tikz}
\usepackage{here}

\usepackage{pifont}

\usepackage{listings}
\numberwithin{equation}{section}
\usepackage{xcolor}

\newtheorem{theorem}{Theorem}[section]
\newtheorem{prop}[theorem]{Proposition}
\newtheorem{lemma}[theorem]{Lemma}

\newtheorem{corollary}[theorem]{Corollary}

\newtheorem{proposition}[theorem]{Proposition}
\newtheorem{example}[theorem]{Example}

\theoremstyle{definition}
\newtheorem{definition}[theorem]{Definition}
\newtheorem{remark}[theorem]{Remark}

\makeatletter
\newcommand{\Extend}[5]{\ext@arrow0099{\arrowfill@#1#2#3}{#4}{#5}}
\makeatother

\newcommand{\R}{\mathbb{R}}  
\newcommand{\Z}{\mathbb{Z}}  
\newcommand{\N}{\mathbb{N}}  
  
\newcommand{\CC}{\mathbb{C}}

\newcommand{\II}{\mathrm{i}}
\newcommand{\p}{\partial} 
\newcommand{\pb}{\bar{\partial}}
\newcommand{\ud}{\mathrm{d}}

\newcommand{\IP}[2]{\langle #1, #2\rangle}
\newcommand{\RRe}[1]{\mathrm{Re}(#1)}
\newcommand{\IIm}[1]{\mathrm{Im}(#1)}
\newcommand{\bs}[1]{\boldsymbol{#1}}

\begin{document}
\title{On nondegenerate $\Z_{2}$-harmonic $1$-forms with shrinking branching sets}

\author[Dashen Yan]{Dashen Yan}
\address{Department of Mathematics, Stony Brook University, Stony Brook, NY 11794-3651 USA. E-mail: dashen.yan@stonybrook.edu}

\maketitle
\begin{abstract}
	We develop a gluing theorem for non-degenerate $\Z_{2}$-harmonic $1$-forms on compact manifolds, in which non-degenerate $\Z_{2}$-harmonic $1$-forms on $\R^{n}$ are glued to the regular zeros of a non-degenerate $\Z_{2}$-harmonic $1$-form. As a consequence, we show that for every compact oriented manifold $M^{n}, n\geq 3$, if the first Betti number $b^{1}(M)>0$, then $M$ admits a family of non-degenerate $\Z_{2}$-harmonic $1$-forms, which resolves a folklore conjecture. We will also discuss several possible applications to special holonomy, in particular, to the field of $G_{2}$-geometry.
\end{abstract}

\section{Introduction}\label{S_1}%%%%%%%%%%%%

	Let $(M^{n},g)$ be a Riemannian manifold of dimension $n$ and $\Sigma$ be a closed subset of Hausdorff dimension at most $n-2$. A $\Z_{2}$-harmonic $1$-form $\alpha$ is a $2$-valued $1$-form defined on $M^{n}\setminus \Sigma$, with $|\alpha|, |\nabla \alpha|\in L^{2}_{loc}$. The concept of $\Z_{2}$-harmonic $1$-form was introduced by Taubes in his study of the compactification of flat $PSL(2,\CC)$ connections on $3$-manifolds \cite{Taubes13A}. Later the $\Z_{2}$-harmonic $1$-forms also appear in various compactification problems in gauge theory \cite{Haydys15,Esfahani24,Taubes13A,Taubes13B,SqiHe24B} and in the problems of calibrated submanifolds \cite{SqiHe23}.\\
	
	Examples of $\Z_{2}$-harmonic $1$-forms have been studied by many people \cite{Haydiy23,SqiHe24,SqiHe25,Taubes20,Taubes24,Chen24,Donaldson25B,Yan25}. Of particular interest is the case when $\Sigma$ is an oriented codimension-$2$ submanifold in $M$. In this case, the local asymptotic behavior near the singular set and global deformation theory with respect to the metric and the cohomology class are well understood. We now give a precise definition in this case.
		\begin{definition}\label{S_1_Defn_Z2}
			A multivalued-harmonic $1$-form consisting of the following data
				\begin{enumerate}
					\item An oriented codimension-$2$ submanifold $\Sigma\subset M$,
					\item A flat real line bundle $L$ over $M\setminus \Sigma$, such that any loop linking along $\Sigma$ has monodromy $-1$.
				\end{enumerate}
			A section $\alpha\in \Gamma(M\setminus \Sigma, T^{*}M\otimes L)$ is called a multivalued-harmonic $1$-form if $|\alpha|\in L^{2}_{loc}$ satisfies the harmonic equation
				\begin{equation*}
					\ud \alpha=\ud ^{*}\alpha=0.
				\end{equation*}
			A multivalued-harmonic $1$-form is called \textbf{$\Z_{2}$-harmonic} if, in addition, $|\nabla \alpha|\in L^{2}_{loc}$. Moreover, let $r$ be the distance to the branching set $\Sigma$. The $\Z_{2}$-harmonic $1$-form is called \textbf{non-degenerate} if the quantity $r^{-1/2}|\alpha|$ is uniformly bounded from below near $\Sigma$.
		\end{definition}
		
		\begin{remark}\label{S_1_Rmk_Def}
			For simplicity, we also view a harmonic $1$-form as a non-degenerate $\Z_{2}$-harmonic $1$-form by setting $\Sigma=\emptyset$ and $L$ to be the trivial real line bundle.
		\end{remark}
	Let $N$ be the normal bundle of $\Sigma$, and use the inverse of the exponential to define the map $\zeta$ from a tubular neighborhood of $\Sigma$ to $N$ in $(M,g)$. The coordinate map $\zeta$ is referred to as \textbf{Fermi coordinate} on $\Sigma$. Using co-orientation, we can regard $N$ as a complex line bundle over $\Sigma$. If $\sigma$ is a section of the dual normal bundle $N^{-1}$, we obtain a complex-valued function $\IP{\sigma}{\zeta}$, which we will simply denote as $\sigma \zeta$. The monodromy around $\Sigma$ defines a square root $N^{1/2}$ with dual $N^{-1/2}$. Similarly, we can define a complex-valued function
		\begin{equation*}
			\sigma_{p}\zeta^{p}, \sigma_{p}\in \Gamma(\Sigma, N^{-p}),
		\end{equation*}
		for any half-integer $p$. According to the work of Donaldson \cite{Donaldson17}, the multivalued-harmonic $1$-form $\alpha$ satisfies
		\begin{equation*}
			\alpha=\ud \bigg(\RRe{A(t)\zeta^{\frac{1}{2}}+B(t)\zeta^{\frac{3}{2}}}-\frac{1}{2}\RRe{\bar{\mu}\zeta}\RRe{A(t)\zeta^{\frac{1}{2}}}+o(|\zeta|^{\frac{3}{2}})\bigg),
		\end{equation*}
		near the branching set $\Sigma$. Here, $A\in \Gamma(\Sigma, N^{-\frac{1}{2}})$, $B\in \Gamma(\Sigma, N\Sigma^{-\frac{3}{2}})$, and $\mu\in \Gamma(\Sigma, N)$ is the mean curvature of $\Sigma$ in $(M,g)$. In particular, a multivalued harmonic $1$-form is $\Z_{2}$-harmonic if and only if $A\equiv 0$ and is nondegenerate if, in addition, $B(t)\neq 0$ for all $t \in \Sigma$.\\
	
	We define the set $Z=\{x\in M^{n}\big||\alpha|=0\}=0$ to be the \textbf{zero} of $\alpha$. We also call $\Sigma$ the \textbf{branching zero} and $Z\setminus \Sigma$ the \textbf{ordinary zero}. Moreover, if the $\alpha$ intersects the zero section transversely at $p\in Z\setminus \Sigma$, then we call $p$ a \textbf{regular zero}. Pick a trivialization of $L$ near an ordinary zero $p$, we can write $\alpha$ as a differential of a harmonic function $f$. We may normalize so that $f(p)=0$. Under the geodesic coordinates $\{\bs{x}\}$ near $p$, the harmonic function admits the following Taylor series expansion
		\begin{equation}\label{S_1_Eqn_Taylor}
			f=P_{2}(\bs{x})+O(|\bs{x}|^{3}),
		\end{equation}
		where $P_{2}$ defines a homogeneous harmonic polynomial of degree $2$ on $\R^{n}$. In particular, $p$ is a regular zero if and only if $P_{2}$ is a nondegenerate quadric. We define the Morse index of $P_{2}$ to be the \textbf{Morse index} of $\alpha$ at $p$. Noted here that the index is defined up to multiplication by $-1$ in $\Z/n\Z$.\\

	The main result of this paper is a gluing construction of a one parameter family of distinct nondegenerate $\Z_{2}$-harmonic $1$-forms that converge to a nondegenerate $\Z_{2}$-harmonic $1$-form with a non-empty regular zero set. In particular, some of the branching sets shrink to the regular zeros.
		\begin{theorem}\label{S_1_Thm_main}
			Let $(M^{n},g)$ be a compact Riemannian manifold with dimension $n\geq 3$, and $\alpha$ is a nondegenerate $\Z_{2}$-harmonic $1$-form with branching set $\Sigma$ being a codimension-$2$ submanifold. Let $\mathcal{R}=\{p_{1},\cdots, p_{q},\cdots\}$ be a non-empty subset of regular zeros. Suppose further that for each $p_{q}\in \mathcal{R}$, there exists a non-degenerate $\Z_{2}$-harmonic $1$-form $\alpha_{q}$ on $(\R^{n},g_{\R^{n}})$ with the following matching conditions
				\begin{enumerate}
					\item The branching set $\Sigma_{q}\subset \R^{n}$ is compact, and therefore the associated real line bundle $L_{q}$ is trivial outside a large compact set.
					\item Pick a trivialization of $L_{q}$ outside a large compact set, the $\Z_{2}$-harmonic $1$-form admits
						\begin{equation*}
							\alpha_{q}=\ud\big( P_{2}(\bs{y})+O'(|\bs{y}|^{2-n})\big), |\bs{y}|\to \infty.
						\end{equation*}
						Here, $P_{2}$ is the degree $2$ harmonic polynomial given by the Taylor's expansion of $\alpha$ at $p$ in Equation \ref{S_1_Eqn_Taylor}, and $O'(|\bs{y}|^{m})$ denotes a class of functions such that
						\begin{equation*}
							|\nabla^{k}f|=O(|\bs{y}|^{m-k}).
						\end{equation*}
				\end{enumerate}
			Then there exists a one parameter family of non-degenerate $\Z_{2}$-harmonic $1$-forms $\alpha_{\epsilon},0<\epsilon<\epsilon_{0}$ with respect to the original metric $g$, such that $\alpha_{\epsilon}$ converges to $\alpha$ away from $\mathcal{R}$ and that there exist diffeomorphisms $\Psi_{q,\epsilon} $ near each $p_{q}\in \mathcal{R}$  
				\begin{equation*}
					\big(\Psi_{q,\epsilon}^{-1}\big)^{*}\alpha_{\epsilon} \to \alpha_{q}, \text{ in } C^{\infty}_{loc}(\R^{n}\setminus \Sigma_{q}),
				\end{equation*}
				as $\epsilon\to 0$.
		\end{theorem}
	The key feature of Theorem \ref{S_1_Thm_main} is that the Riemannian metric g remains unchanged. This is a special feature of dimensions at least 3. In contrast, when M is a compact Riemann surface, the space of $\Z_2$-harmonic 1-forms is naturally identified with the finite-dimensional space of quadratic differentials $H^0(M, K_M^{\otimes 2})$. It follows that, once the metric and therefore the complex structure of the Riemann surface is fixed, there can be no sequence of $\Z_2$-harmonic 1-forms whose branching sets is shrinking, equivalently, whose zeros as quadratic differentials collide.
    
    This feature is crucial for geometric applications in special holonomy. On the other hand, if this  condition is dropped, one can apply a modification of Calabi's intrinsic method \cite{Calabi69} to construct a family of metrics $g'_{\epsilon}$ such that the closed $1$-form obtained in the pre-gluing in Section \ref{S_3_2} is harmonic with respect to those new metrics. And moreover, in this case, the non-degenerate condition for the $\Z_{2}$-harmonic $1$-form on $\R^{n}$ can be dropped.\\

	In a recent work \cite{Chen25}, Jiahuang Chen proved a transversality result for non-degenerate $\Z_{2}$-harmonic $1$-forms, namely, under a generic perturbation of both the metric and the cohomology class, the ordinary zeros can be perturbed into regular zeros. This generalizes Ko Honda's result for harmonic $1$-forms in \cite{Honda04}. Consequently, there is a huge supply of non-degenerate $Z_{2}$-harmonic $1$-forms with regular zeros. Combined with the following existence result for non-degenerate $\Z_{2}$-harmonic $1$-forms on $\R^{n}$ for $n\geq 3$ in \cite{Donaldson25B,Yan25}. We have the following existence result of non-degenerate $\Z_{2}$-harmonic $1$-forms on compact manifolds.
		\begin{corollary}\label{S_1_Cor_existence}
			Let $M^{n}, n\geq 3$ be an oriented compact manifold with first Betti number $b^{1}(M)\geq 1$. then there exists a Riemannian metric $g$ such that there exists a family of non-degenerate $\Z_{2}$-harmonic $1$-forms $\alpha_{\epsilon}$ constructed from a harmonic $1$-form $\alpha$ admitting regular zeros of Morse index $1$. In particular, when $n=3$ and $M$ are not a surface bundle over $S^{1}$, then for a generic choice of Riemannian metric, there exists a family of non-degenerate $\Z_{2}$-harmonic $1$-forms.
		\end{corollary}
	Geometrically, the branched covers $\widetilde{M}_{\epsilon}$ can be obtained as follows. Let $\mathcal{R}$ be the subset of regular zeros of the harmonic $1$-form $\alpha$ with Morse index equal to $1$ or $n-1$. Define
		\begin{equation*}
			\widetilde{M}_{0}=M\cup_{\mathcal{R}} \overline{M}
		\end{equation*}
		be a union of two copies of $M$, obtained by identifying each point in $\mathcal{R}\subset M$ with the corresponding point in $\mathcal{R}\subset \overline{M}$. Now $\widetilde{M}_{0}$ is a compact conifold with conical singularities at $p\in \mathcal{R}$ which are modeled on
		\begin{equation*}
			\mathcal{C}_{p}:=(S^{n-1}\sqcup S^{n-1})\times \R^{+}\cup \{p\}.
		\end{equation*}
	The double branched cover relating to the $\Z_{2}$-harmonic $1$-forms $\alpha_{\epsilon}$ in Corollary \ref{S_1_Cor_existence} is given by replacing $\mathcal{C}_{p}$ with a cylinder $S^{n-1}\times \R$. Roughly speaking, the branched covering map from the cylinder to $\R^{n}$ is locally induced by a family of Lawlor's necks, which will be given explicitly in Equation \ref{S_3_Eqn_Graph}.\\
		
	We now provide several items and background for studying this problem.
		\subsubsection*{Resolving singular special Lagrangian.}
			According to the work of He \cite{SqiHe23} , the $\Z_{2}$-harmonic $1$-forms can be identified as the infinitesimal deformations of branched-covers of smooth special Lagrangians in a compact Calabi-Yau manifold. Our result provides a linear version of resolving singular branched-cover special Lagrangians, with singularities modeled on the intersection of two real planes in $\CC^{n}$.
		\subsubsection*{Resolving singularities in $G_{2}$-manifolds.}
			Suppose $(X,\varphi)$ is a torsion-free $G_{2}$-manifold admitting a $G_{2}$-involution $\tau$, in the sense that $\varphi$ is invariant under $\tau$. Suppose further that the fixed point set $M$ of $\tau$ is an associative submanifold. A result by Joyce and Karigiannis \cite{Joyce21} states that if there is a nowhere vanishing harmonic $1$-form $\alpha$ on the $M$, then there exists a resolution of the $G_{2}$ orbifold $X/\langle \tau \rangle$ by gluing in a family of $T^{*}S^{2}$ to resolve $\CC^{2}/\Z_{2}$ singularities in the normal direction of $M$ in $X/\langle \tau\rangle$. However, the gluing ansatz in \cite{Joyce21} does not work when $\alpha$ admits zeros. Instead, when the harmonic $1$-form $\alpha$ only admits regular zeros, Theorem \ref{S_1_Thm_main} implies that there exists a family of $\Z_{2}$-harmonic $1$-forms $\alpha_{\epsilon}$ that converge to $\alpha$ and admit only branching zeros. For each $\epsilon$, near the branching set on the associative in the orbifold, one should instead glue in a nonstandard Calabi-Yau metric on $\CC^{3}$ constructed by Li \cite{Li17} to resolve the singularities. It should be noted here that $g$ remaining invariant is crucial for this potential application. For more details, see Section \ref{S_5_2}.
		\subsubsection*{Branched maximal sections}
				The author's main motivation for studying this problem comes from a non-linear version developed by Donaldson in \cite{Donaldson17A}, involving the branched maximal surface equation and the adiabatic limit of Kovalev-Lefchetz coassociative fibration.
			
				Let $\R^{3,19}=H^{2}(K3,\R)$ and let $\Lambda =H^{2}(K3,\Z)$ be the integral lattice.  Also let $\Sigma$ be a smooth link in a compact $3$-manifold $B$ and let $\Gamma$ be the affine extension
					\begin{equation*}
						0\to \R^{3,19}\to \Gamma\to O(3,19)\to 0
					\end{equation*}
					of the orthonormal group with respect to the indefinite metric. Suppose there is a homomorphism
					\begin{equation*}
					\chi: \pi_{1}(B\setminus \Sigma)\to \Gamma,
					\end{equation*}
					such that every small loop $\gamma$ around $L$ maps to a reflection
					\begin{equation*}
						\chi([\gamma])(v)=v+\IP{v}{\delta}\delta.
					\end{equation*}
				Here $\delta\in \Lambda$, $\IP{\delta}{\delta}=-2$ is a $(-2)$-class. Now, such a representation defines an affine bundle $E_{\chi}$ over $B\setminus \Sigma$. We consider a section $h$ of $E_{\chi}$, which is locally given by maps into $\R^{3,19}$, with an image solving the maximal submanifold equation. Near a tubular neighborhood $U$ of each component of $\Sigma$, the orbifold bundle can be decomposed into
					\begin{equation*}
						E_{\chi}\big|_{U\setminus \Sigma}\simeq \R^{3,18}\oplus L\cdot \delta,
					\end{equation*}
				where $\delta$ is the $(-2)$-class corresponding to the small loop $\gamma$, and $L$ is a flat real line with monodromy $-1$ around the $\gamma$.
					
				Similarly to the definition of non-degenerate $\Z_{2}$-harmonic, we call the maximal section $h$ a \textbf{branched maximal section} if
					\begin{equation*}
						h=(t,\zeta,w, \RRe{B(t)\zeta^{\frac{3}{2}}})+o(|\zeta|^{\frac{3}{2}}),
					\end{equation*}
				near $\Sigma$. Here, $t$ is the coordinate on $\Sigma$ and $\zeta$ is the complex coordinate in the normal direction. We call $h$ \textbf{avoids $(-2)$-classes} if away from the branching set $\Sigma$ there is no $(-2)$-class orthogonal to the image of the derivatives of $h$. In \cite{Donaldson17A}, Donaldson conjectured that if there is a branched maximal section $h$ that avoids $(-2)$-classes, then there is a compact $7$-manifold $X$ admitting a differentiable Kovalev-Lefschetz fibration $\pi: X\to B$ (Definition 1, page 16 \cite{Donaldson17A}) with generic fibers diffeomorphic to a $K3$ surface, and there is a one parameter family of torsion-free $G_{2}$-structures $\varphi_{t}, t>t_{0}\gg 1$ on $X$ such that the fibers of $\pi$ are coassociative. The singular locus of the fibration coincides with $\Sigma$. Moreover, the normalized hyper-symplectic structures on the $K3$ fibers will converge to a family of hyper-K\"ahler structures induced by the differential of the branched maximal section $h$ at each $b\in B$.

				Constructing a solution to the branched maximal surface equation over a compact base space is challenging, particularly under the additional requirement that it avoids $(-2)$-class. Our starting point is to develop a similar gluing technique in the setting of branched maximal sections when it does not avoid $(-2)$-classes. The easiest situation is when the image of the derivatives intersects transversally with the set of planes that are orthogonal to a $(-2)$-class. More precisely, let $Gr_{3}^{+}(\R^{3,19})$ be the Grassmannian for the positive $3$-plane, and $S_{\delta}=\{H\in  Gr^{+}_{3}\big| H\perp \delta\}$. Denote	
					\begin{equation*}
						S=\bigcup_{\delta\in \Lambda,\IP{\delta}{\delta}=-2} S_{\delta}
					\end{equation*}		
				as the union, then the complement $Gr_{3}^{+}\setminus S$ is the moduli space of polarized hyper-K\"ahler structures on $K3$ surfaces. The set $S$ is then the orbifold compactification of a sequence of non-collapsing hyper-K\"ahler structures, and the main strata $S^{\circ} \subset S$ correspond to those orbifolds admitting only $1$-singularities modeled on $\CC^{2}/\Z_{2}$. The set $S^{\circ}$ is a smooth submanifold of codimension $3$ in $Gr_{3}^{+}$. In the case when $(Dh)_{*}T_{p}B$ intersects $S^{\circ}$ transversally at $p$, we find a graphical coordinate for the affine bundle such that near $p$
					\begin{equation*}
						E_{\chi}\simeq \R^{3}\oplus \R^{18}\oplus \R \cdot \delta,
					\end{equation*}
				and $h$ can be decomposed into
					\begin{equation*}
						h(x)=(x, u(x),f(x)),
					\end{equation*}
				with $u(p), Du(p)=0$. The transversality condition yields that the Taylor expansion for $f$ at $p$ begins with a non-degenerate quadric $P_{2}$. The maximal submanifold equation in graphical coordinates \cite{Li22} implies that $P_{2}(x)$ is again a harmonic polynomial of degree $2$. The non-degenerate $\Z_{2}$-harmonic $1$-forms on $\R^{3}$ \cite{Yan25} can be used to turn the transverse intersection into a small branching set in a manner similar to Theorem \ref{S_1_Thm_main}.

		Our proof of Theorem \ref{S_1_Thm_main} is an application of a weighted version of Nash-Moser theory. On the geometric side, we invoke Donaldson's framework of deformation of multivalued harmonic functions, adapting it to a weighted setting.	 On the analysis side, instead of working within Hamilton's framework, we directly construct a family of smoothing operators and derive the corresponding tame estimates in weighted H\"older space.	\\
		
		Moreover, our application of Nash-Moser theory requires more than just the invertibility of the approximate inverse in a neighborhood of the approximate solution. A key step is to establish a uniform estimate for both the error terms and the inverse operator. this involves the mapping properties of the Laplacian operators on the real line bundle $L_{q}$ over $\R^{n}\setminus \Sigma_{q}$ and on the $L$ over $M\setminus \Sigma$. In our application of Nash-Moser theory, we make a slight modification to the arguments in Appendix A of \cite{Donaldson25A} to obtain the desired result.
				
	\subsection*{Ackownledgement}
		The author is grateful to his Ph.D advisor, Prof. LeBrun, as well as to Prof. Donaldson and Prof. He for their valuable discussions. He also wishes to thank Prof. Karigiannis and Jiahuang Chen for many fruitful discussions and helpful suggestions in Section \ref{S_5}. This work is partially supported by the Simon foundation.

    \subsection*{Notation}
        In this paper, we use $O(|x|^{k})$ to denote the remaining terms for Taylor's expansion, and $O'(|x|^{k})$ to denote the remaining terms such that $|\nabla^{p}f|\leq C_{p}|x|^{k-p}$.

\section{Preliminary}\label{S_2}%%%%%%%%%%%%
	In this section, we review Donaldson's framework on the deformation of multivalued harmonic $1$-forms. In particular, we will focus on the Schauder's type estimate for the singular Laplacian equation with a codimension-$2$ singularity.	
	\subsection{Analysis on flat model}%%%%%%%%
		In this subsection, we work on the product $\R^{n}=\R^{n-2}\times \CC$ and represent a point in $\R^{n}$ by $(t,z)$ and write $z=re^{\II \theta}$. Let $E$ be a vector bundle over $\R^{n-2}\times \CC^{*}$ with monodromy $-1$ around $\R^{n-2}\times\{0\}$. From the perspective of branched covers, if we take $(t,w)\mapsto (t, w^{2})$ as a branch cover, a section in $E$ can be thought of as an odd section with respect to the involution $(t,w)\mapsto (t,-w)$. Let $\rho$ be a smooth section in $E$ with compact support. The main focus of this section is to find the inverse operator of the Laplacian function
			\begin{equation}\label{S_2_Eqn_Lap}
				\Delta f=\rho.
			\end{equation}
			Here, we use the analyst's convention for the Laplacian. Let $L^{2}, L^{2}_{1}$ be the Hilbert space obtained as the completion of compactly-supported smooth sections under the $L^{2}$ and $L^{2}_{1}$ norms. Using the Reisz representation and Sobolev's inequality, for any compactly-supported $L^{2}$ section $\rho$, we can find a unique weak solution $f$ in $L^{2}_{1}$, and we define $f=\underline{G}(\rho)$.\\
			
		It is a classical result that the inverse can be represented by a singular integral of the Green function
			\begin{equation*}
				\underline{G}(\rho)(p)=\int_{\R^{n}}G(p,p')\rho(p')\ud p'.
			\end{equation*}
			More precisely, $G(p,p')$ is a singular section on $\pi_{1}^{*}E\otimes \pi_{2}^{*}E$ on $(\R^{n-2}\times \CC^{*})^{2}$; here $\pi_{j}, j=1,2$ are the projections to the $j$-th factor of $\R^{n-2}\times \CC$. The properties of $G$ can be summarized as follows.
				\begin{enumerate}
					\item $G(p,p')$ is symmetric with respect to $p,p'$ and is invariant under translation in the $t$-direction and rotation $SO(n-2)\times S^{1}$.
					\item $G$ is homogeneous of degree $2-n$
							\begin{equation*}
								G(\lambda p, \lambda p')=\lambda^{2-n}G(p,p'), \lambda >0.
							\end{equation*}
					\item $G$ is smooth on the complement of the diagonal in $(\R^{n-2}\times \CC^{*})^{2}$ and has pole singularities along the diagonal.
					\item If $p\in B_{1}$ and $p\in \R^{n}\setminus \overline{B_{1}}$, then there is a polyhomogeneous expansion for the Green's function
						\begin{equation*}
							G(p,p')=\RRe{\sum_{k,l\geq 0}a_{l,k}(t,p')e^{\II(l+1/2)\theta}r^{l+2k+1/2}}.
						\end{equation*}
						Moreover, from the first and second properties, we conclude that $a_{l,k}(t,p')=h_{l,k}(t-t', r',\theta')$ such that
							\begin{equation*}
								h_{l,k}(\lambda \tau, \lambda r')=\lambda^{\frac{3}{2}-n-l-2k}h_{l,k}(\tau, r').
							\end{equation*}
				\end{enumerate}

		From the behavior of the Green's function, one can derive Schauder's estimate for the Laplacian equation
			\begin{equation*}
				\Delta f=\rho,
			\end{equation*}
			and study the asymptotic behavior of $f$ near the branching set $\R^{n-2}\times \{0\}$. To begin, fix a constant $\alpha\in (0,\frac{1}{2})$, we define the H\"older semi-norm of sections on $E$ to be
			\begin{equation*}
				\|s\|_{,\alpha}=\sup \frac{|s(p)-s(p')|}{|p-p'|^{\alpha}},
			\end{equation*}
			where the supremum is taken over the pair $|z-z'|<\frac{1}{2}\min{(|z|,|z'|)}$, and $s(p),s(p')$ is compared using parallel transport along the line joining $p$ and $p'$. Parallel transport around the branching set yields
			\begin{equation*}
				|s(t,z)|\leq C \|s\|_{,\alpha}|z|^{\alpha}.
			\end{equation*}

		Let $\underline{\mathcal{T}_{k}}$ be the set of differential operators, which are products of $k$ vector fields of the form
			\begin{equation*}
				\frac{\p}{\p t_{i}}, r\frac{\p}{\p r}, \frac{\p}{\p \theta}.
			\end{equation*}
		In what follows, we introduce two different kinds of norms. The first are the H\"older norms, which we will ultimately use to define operators concerning the asymptotic behavior of the harmonic section at the branching set. While the second one is weighted $C^{k}$ norms, which will be used to establish tame estimates, addressing the fact that there is no interpolation formula for H\"older norms with a fixed H\"older exponent.

		\begin{definition}\label{S_2_Defn_DHolder}
			We define the H\"older $\mathcal{D}^{k,\alpha}$ norms of a section to be
			\begin{equation*}
				\|s\|_{\mathcal{D}^{k,\alpha}}:=\max_{0\leq j\leq k, D\in \underline{\mathcal{T}_{j}}}\|D s\|_{,\alpha}.
			\end{equation*}
		\end{definition}
		Let $\mathcal{T}$ be the set of vector fields on $\R^{n}$ that are tangent to the branching set $\R^{n-2}\times \{0\}$. We define $\mathcal{T}_{k}$ to be sums of products of at most $k$ elements in $\mathcal{T}$. It is easy to check that $\underline{\mathcal{T}_{k}}$ is a subset of $\mathcal{T}_{k}$ and that, over a unit ball, we have the following estimate
			\begin{equation*}
				\|Ds\|_{,\alpha}\leq C_{D}\|s\|_{,\alpha}.
			\end{equation*}
		In order to apply the tame estimate, we define another graded norm on the space $\mathcal{D}^{\infty,\alpha}$.
		\begin{definition}
			We define $C^{k}_{\alpha}$ norm if a section is to be
				\begin{equation*}
					\|s\|_{C^{k}_{\alpha}}:=\max_{0\leq j\leq k, D\in \underline{\mathcal{T}_{j}}}\||z|^{-\alpha}Ds\|_{C^{0}}.
				\end{equation*}
		\end{definition}
		Moreover, we find a sequence of constant $C_{k}$ such that
		\begin{equation*}
			\frac{1}{C_{k}}\|s\|_{C^{k}_{\alpha}}\leq \|s\|_{\mathcal{D}^{k,\alpha}}\leq C_{k}\|s\|_{C^{k+1}_{\alpha}}.
		\end{equation*}
			
		For the H\"older norms $\mathcal{D}^{k,\alpha}$, an analog of Schauder's estimate is as follows.
			\begin{prop}\label{S_2_Prop_Schauder}
				Suppose $\rho$ has compact support in $B_{1}$, then
					\begin{equation*}
						\|\underline{G}(\rho)\|_{\mathcal{D}^{2+k,\alpha}(B_{1})}\leq C_{k}\|\rho\|_{\mathcal{D}^{k,\alpha}}.
					\end{equation*}
			\end{prop}
			\begin{proof}
				This follows directly from analyzing the integral kernel $G(p,p')$. For more details, see \cite{Donaldson17}
			\end{proof}	
		
		We now define two operators that describe the asymptotic behavior of the solution to the Laplacian equation \ref{S_2_Eqn_Lap}.
			\begin{definition}\label{S_2_Defn_Asym}
				We define two integral operators that take compactly supported sections of $E$ on $\R^{n}$ to functions on $\R^{n-2}\times \{0\}$
					\begin{equation}
						\begin{split}
							&A(\rho)(t)=\int_{\R^{n}}h_{0,0}(t-t',r',\theta')\rho(p')\ud  p'\\
							&B(\rho)(t)=\int_{\R^{n}}h_{1,0}(t-t',r'.\theta')\rho(p')\ud p'.
						\end{split}
					\end{equation}
			\end{definition}
		
		We use the usual H\"older space $C^{k,\beta}$ on the branching set $\R^{n-2}\times \{0\}$; thus, we have the following estimates.
			\begin{proposition}\label{S_2_Prop_Asym}
				If $\rho$ has compact support in $B_{1}$, then
					\begin{equation*}
						\|A(\rho)\|_{C^{k+1,\alpha+\frac{1}{2}}}, \|B(\rho)\|_{C^{k,\alpha+\frac{1}{2}}}\leq C_{k}\|\rho\|_{\mathcal{D}^{k,\alpha}}.
					\end{equation*}
				Moreover, the section $\underline{G}(\rho)$ satisfies
					\begin{equation*}
						\underline{G}\rho=\RRe{A(t)z^{\frac{1}{2}}+B(t)z^{\frac{3}{2}}}+Err(t,z),
					\end{equation*}
					with $|Err(t,z)|=o(r^{\frac{3}{2}}\|\rho\|_{,\alpha})$.
			\end{proposition}
			\begin{proof}
				This is also proven by estimating the integral kernel $G(p,p')$. For more details, see \cite{Donaldson17}.
			\end{proof}
		
	\subsection{Local Schauder-type estimate}%%%%%%%%
		In this subsection, we extend the previous results to a class of second order operators with variable coefficients. We view this as a local model for a tubular neighborhood of a codimension-$2$ submanifold. We say a differential operator on $\R^{n}$ is \textbf{admissible} if:
			\begin{enumerate}\label{S_2_Defn_Adimssible}
				\item $\tilde{\Delta}=\Delta_{\R^{n}}+\mathcal{L}$, where $\mathcal{L} \in \mathcal{T}_{2}$. Recall that $\mathcal{T}_{2}$ is the set of sums of products of at most two tangential vector fields.
				\item $\tilde{\Delta}$ is elliptic and
						\begin{equation*}
							\tilde{\Delta}=W^{-1}\Delta_{g}W,
						\end{equation*}
						where $W$ is a smooth positive function, and $\Delta_{g}$ is the Laplacian operator for a smooth Riemannian metric $g$.
				\item $\widetilde{\Delta}=\Delta_{\R^{n}}$ outside a compact set.
			\end{enumerate}
			\begin{remark}
				The first properties will allow us to develop a good Schauder-type estimate, and the second will allow us to use Hilbert space theory, so that for every $\rho \in L^{2}$ that has compact support, there is a unique solution
					\begin{equation*}
						\widetilde{\Delta}f=\rho, f\in L^{2}_{1}.
					\end{equation*}
				Moreover, we will see in the next section that for any smooth Riemannian metric $g$, which equals the Euclidean metric outside a compact set, we can always find a smooth function $W$ making $W^{-1}\Delta_{g} W$ admissible.
			\end{remark}
			
			\begin{lemma}\label{S_2_Lemma_Schauder_1}
				Suppose $\widetilde{\Delta}=\Delta_{\R^{n}}+\mathcal{L}$ is admissible and that the coefficients in $\mathcal{L}$ are sufficiently small in $C^{2,\alpha}$. Then, if $\rho$ has compact support in $B_{1}$ and $\|\rho\|_{,\alpha}$ is bounded, the solution
					\begin{equation*}
						\tilde{\Delta}f=\rho
					\end{equation*}
					satisfies
					\begin{equation*}
						\|\Delta_{\R^{n}} f\|_{,\alpha}\leq C\|\rho\|_{,\alpha},
					\end{equation*}
					and $\Delta_{\R^{n}}f$ has compact support in $B_{1}$. Moreover, given that $\|\mathcal{F}\|_{C^{2,\alpha}}$ is sufficiently small, we have the Schauder-type estimate for higher derivatives
					\begin{equation}\label{S_2_Eqn_Schauder}
						\|\Delta_{\R^{n}}f\|_{\mathcal{D}^{k,\alpha}}\leq C_{k}(1+\|\mathcal{F}\|_{C^{k+2,\alpha}})^{k}\|\rho\|_{\mathcal{D}^{k,\alpha}}
					\end{equation}
			\end{lemma}
			\begin{proof}
				To see this, we consider the equation $\widetilde{\Delta}(\underline{G}\sigma)=\rho$ for $\sigma$. This is
					\begin{equation*}
						\sigma+\mathcal{L}(\underline{G}\sigma)=\rho.
					\end{equation*}
				The Schauder-type estimate in Proposition \ref{S_2_Prop_Schauder} implies $\|\mathcal{L}\underline{G}\sigma\|_{,\alpha}\leq C_{\mathcal{L}}\|\sigma\|_{,\alpha}$ for some constant $C_{\mathcal{L}}<\frac{1}{2}$. By applying the Banach contraction mapping theorem, the above equation has a solution $\sigma\in \mathcal{D}^{0,\alpha}$, and by uniqueness $\underline{G}\sigma$ agrees with the weak solution in $L^{2}_{1}$.\\
				
				We use induction to prove Schauder's estimates for higher derivatives. Suppose the estimates \ref{S_2_Eqn_Schauder} hold for $k'\leq k-1$. It follows from Poisson's equation that
					\begin{equation*}
						\|\Delta_{\R^{n}}f\|_{\mathcal{D}^{k,\alpha}}\leq \|\rho\|_{\mathcal{D}^{k,\alpha}}+C_{k}\|\mathcal{L}\|_{C^{k+2,\alpha}}\|f\|_{\mathcal{D}^{k+2,\alpha}}.
					\end{equation*}
Therefore, it remains to bound $\|f\|_{\mathcal{D}^{k+2,\alpha}}$. Choose $D_{k}\in \underline{\mathcal{T}_{k}}$, then we have
					\begin{equation*}
						\widetilde{\Delta}D_{k}f=D_{k}\widetilde{\Delta}f+[\widetilde{\Delta},D_{k}]f.
					\end{equation*}
				Here, $\|[\widetilde{\Delta},D_{k}]f\|_{,\alpha}\leq C_{k}\|\mathcal{L}\|_{C^{k+2,\alpha}}\|f\|_{\mathcal{D}^{k+1,\alpha}}$. Induction hypothesis, together with Proposition \ref{S_2_Prop_Schauder}, proves the lemma.
			\end{proof}

	\subsection{Fermi coordinate and normal structure}%%%%%%%%
		We step back for a moment and discuss the geometry near the branching set in a Riemannian manifold. Let $(M^{n},g)$ be a compact Riemannian manifold, and let $\Sigma$ be an oriented codimension-$2$ submanifold in $M^{n}$. In this subsection, we focus on a tubular neighborhood of $\Sigma$ in $M$. Let $V \subset \Sigma$ be a local chart, with coordinate $(t^{1},\cdots, t^{n-2})$, and $E_{1}(t), E_{2}(t)$ be a local orthonormal  frame on $N\Sigma$. A \textbf{Fermi coordinate} with respect to this choice of coordinates and frame is defined by
			\begin{equation*}\label{S_2_Eqn_Fermi}
				\begin{split}
					& t^{\alpha}(\exp_{g}(s^{j}E_{j}(t)))=t^{\alpha},\\
					&u^{i}(\exp_{g}(s^{j}E_{j}(t)))=s^{j}.
				\end{split}
			\end{equation*}
			Here, we use Einstein's summation convention, and the indices $\alpha$ range from $1$ to $n-2$, while $j$ takes values in $1$ and $2$.\\
			
		Under the Fermi coordinate, straightforward computation implies that the Laplacian operator admits the following expansion near $\Sigma$
			\begin{equation}\label{S_2_Eqn_Laplacian}
				\Delta_{g}=\Delta_{\Sigma}+\Delta_{\CC}-\mu_{\Sigma}+\mathcal{L}.
			\end{equation}
		Here, $\Delta_{\CC}$ is the Euclidean Laplacian for $s_{1},s_{2}$ coordinates, $\mu_{\Sigma}$ is the mean curvature of $\Sigma$, and $\mathcal{L}$ is a sum of products of two vector fields that vanish on $\Sigma$. However, this operator is not admissible; therefore, we are unable to apply the Schauder-type estimate to this case. Still, we can consider a slightly different operator
			\begin{equation*}
				\widetilde{\Delta}_{g}=W^{-1}(g)\Delta_{g}W(g).
			\end{equation*}
			The function $W(g)$ is defined by $W(g)=1$ away from $\Sigma$, while near $\Sigma$ we set $W(g)=V(g)^{-\frac{1}{2}}$, where $V(g)\ud s^{1}\wedge \ud s^{2}$ denotes the volume form in the normal direction.\\
		
		The above discussion depends on the choice of metric and on the local coordinates near $\Sigma$. To set up the problem in context when the metric is allowed to vary, we therefore introduce the notion of a \textbf{normal structure} on $\Sigma$, consisting of the following data.
			\begin{enumerate}
				\item A normal bundle $N\Sigma \subset TM\big|_{\Sigma}$, such that $TM\big|_{\Sigma}=N\Sigma\oplus T\Sigma$.
				\item A metric on $N\Sigma$.
				\item A $2$-jet along $\Sigma$ of diffeomorphism from the total space of $N\Sigma$ to $M$, extending the canonical $1$-jet. In other words, this is an equivalent class of diffeomorphisms from $N\Sigma$ to $M$ that agree up to second order derivatives along $\Sigma$.
			\end{enumerate}
		Given a Riemannian metric $g$, we define the induced normal structure on $\Sigma$ given by the Fermi coordinate. Suppose $g,g'$ induced the same normal structure. Then the maps $\zeta_{g},\zeta_{g}'$ from a tubular neighborhood of $\Sigma$ to $N\Sigma$ differ by a term of order $r^{2}$. Moreover, we have the following proposition for the operator $\widetilde{\Delta}_{g}$ when the normal structure is fixed.
			\begin{prop}\label{S_2_Prop_Metric_Admissible}
				The operator $\tilde{\Delta}_{g}=W(g)^{-1}\Delta_{g}W(g)$ is admissible in any coordinate system compatible with the normal structure defined by the metric $g$.
			\end{prop}
			\begin{proof}
				See Proposition 6 in \cite{Donaldson17}.
			\end{proof}
		Given a normal structure, we define the related $\|\cdot\|_{\mathcal{D}^{k,\alpha}}$ and $\|\cdot\|_{C^{k}_{\alpha}}$ norms in a small tubular neighborhood of $\Sigma$. Let $d>0$ be a sufficiently small constant. we can find an open covering $\{U_{i}\}$ of $\Sigma$ and a set of Fermi coordinates defined over $\{(t,\zeta)\in N\Sigma|_{U_{i}}\big||\zeta|<2d\}$ such that the Riemannian metric $g$ is sufficiently close to the Euclidean metric $g_{\R^{n}}$ and that the coefficients in $\mathcal{L}=\widetilde{\Delta}_{g}-\Delta_{\R^{n}}$ are sufficiently small in the sense of Lemma \ref{S_2_Lemma_Schauder_1}.
		
			\begin{definition}\label{S_2_Defn_DCHolder}
				In the tubular neighborhood of $\Sigma$, we define the $\|\cdot\|_{C^{k}_{\alpha}}$ and $\|\cdot\|_{\mathcal{D}^{k,\alpha}}$ as the maximum of the corresponding norms among the patches $\{(t,\zeta)\in N\Sigma|_{U_{i}}\big||\zeta|<2d\}$. We also define the space $\mathcal{C}^{k,\alpha},0<\alpha<\frac{1}{2}$ to be
					\begin{equation*}
						\mathcal{C}^{k+2,\alpha}:=\{s\in \mathcal{D}^{k+2,\alpha}|\widetilde{\Delta}_{g}s\in \mathcal{D}^{k,\alpha}\},
					\end{equation*}
				and define the $\mathcal{C}^{k+2,\alpha}$ norms to be
					\begin{equation*}
						\|s\|_{\mathcal{C}^{k+2,\alpha}}:=\|\widetilde{\Delta}_{g}s\|_{\mathcal{D}^{k,\alpha}}.
					\end{equation*}
			\end{definition}

		Once the normal structure is fixed and the norms are defined, we now have a global version of the asymptotic expansion of sections around $\Sigma$.
		
		\begin{definition}\label{S_2_Defn_AB}
			Fix a normal structure of $\Sigma$, we define the bounded linear maps
				\begin{equation*}
					\begin{split}
						&\underline{A}: \mathcal{C}^{k+2}\to C^{k+1,\alpha+\frac{1}{2}}(\Sigma,N\Sigma^{-\frac{1}{2}}),\\
						&\underline{B}: \mathcal{C}^{k+2}\to C^{k,\alpha+\frac{1}{2}}(\Sigma,N\Sigma^{-\frac{3}{2}}).
					\end{split}
				\end{equation*}
		\end{definition}
		
		Returning to our set-up, we now study the operator $\widetilde{\Delta}_{g}$ when $g, \Sigma$ varies. Let $\mathcal{M}_{\Sigma}$ denote the set of smooth metrics in a small tubular neighborhood of $\Sigma$ in $M$. Fix a normal structure on $\Sigma$, let $\mathcal{M}_{\Sigma}^{0}\subset \mathcal{M}_{}$ be the subset of metrics that are compatible with the normal structure. We divide the discussion into two steps. First, we consider the deformation of the metric that preserves the normal structure. Then we will show that for any sufficiently small perturbation $(\tilde{g},\tilde{\Sigma})$ there exists a diffeomorphism $\phi$ mapping the tubular neighborhood to itself, such that
			\begin{equation*}
				\phi^{*}\tilde{\Sigma}=\Sigma, \phi^{*}\tilde{g}\in \mathcal{M}_{\Sigma}^{0}.
			\end{equation*}
			From the perspectives of gauge theory, we consider space of metrics together with the deformations of the branching set $\Sigma$. The gauge group is the group of diffeomorphisms, and the gage fixing condition is that the normal structure remains invariant.\\
		  
		  To begin, let $g_{0}$ be a reference metric inducing a normal structure on $\Sigma$, and $g=g_{0}+\gamma\in \mathcal{M}_{\Sigma}^{0}$. Let $\|\cdot \|_{k}$ be the usual H\"older norm in $C^{k,\alpha}(g_{0})$.
		  	\begin{prop}\label{S_2_Prop_Coefficients}
		  		Both $\tilde{\Delta}_{g}$, $\tilde{\Delta}_{g_{0}}$ are admissible operators. Write
		  			\begin{equation*}
		  				\widetilde{\Delta}_{g}=\widetilde{\Delta}_{g_{0}}+\mathcal{L}, \mathcal{L}\in \mathcal{T}_{2}.
		  			\end{equation*}
		  		Denote $\mathcal{F}$ as the coefficients of $\mathcal{L}$, in other words, $\mathcal{F}$ is a section of the dual $2$-jet bundle. Moreover, we have the following estimate
		  			\begin{equation*}
		  				\|\mathcal{F}\|_{k}\leq C_{k}\|\gamma\|_{k+K}.
		  			\end{equation*}
		  	\end{prop}
		  	\begin{proof}
		  		See Proposition 6 in \cite{Donaldson17}.
		  	\end{proof}
		  	
		  	\begin{lemma}\label{S_2_Lemma_Deformation}
		  		Fix a normal structure of $\Sigma$ and a compatible Riemannian metric $g_{0}$ in a tubular neighborhood of $\Sigma$. Suppose the perturbation of $\widetilde{\Sigma}$ is given by a section $v$ in $N\Sigma$. then there exists a diffeomorphism $\phi$ such that it is equal to the identity outside the region $\{|\zeta|<d\}$ and that
		  			\begin{equation*}
						\Phi^{*}\tilde{\Sigma}=\Sigma, \Phi^{*}\tilde{g}=g+\gamma\in \mathcal{M}_{\Sigma}^{0},
					\end{equation*}
				with $\|v\|_{3K}, \|g-\tilde{g}\|_{3K}\ll 1$.
				Moreover, the following estimates hold
					\begin{equation}\label{S_2_Eqn_NormalEst}
							\|\Phi-Id\|_{k}, \; \|\gamma\|_{k}\leq C\big(\|\tilde{g}-g\|_{k+K}+\|v\|_{k+K}\big),0\leq k\leq K.	
					\end{equation}
		  	\end{lemma}
		  	\begin{proof}
		  		When $v\neq 0$ we can use the exponential map and the normal vector $t v, t \in [0,1]$ to define a family of diffeomorphisms $\tilde{\phi}_{t}$. Now let $\tilde{X}_{t}=\frac{\ud}{\ud t}\tilde{\phi}_{t}$, we then cut off $\tilde{X}_{t}$ to get a new  family of vector fields $X_{t}$ such that it is equal to $\tilde{X}_{t}$ on $\{|\zeta|<\frac{d}{2}\}$ and vanishes on $\{|\zeta|>d\}$. Integrating $X_{t}$, we obtain a diffeomorphism $\phi_{0}$ that is equal to the identity on $\{|\zeta|>d\}$ and that $\phi_{0}^{*}\widetilde{\Sigma}=\Sigma$. It reduces to the case when $\widetilde{\Sigma}=\Sigma$. Let $g_{t}=g_{0}+t(\tilde{g}-g_{0})$, for each $g_{t}$, we define
		  			\begin{equation*}
		  				\tilde{\phi}_{t}=\exp_{g_{0}}\circ \exp_{g_{t}}^{-1},
		  			\end{equation*}
		  			which maps the normal structure induced by $g_{t}$ to the normal structure induced by $g_{0}$. Similarly, integrating the vector field obtained by cutting off the vector field $\frac{\ud}{\ud t}\tilde{\phi}_{t}$ gives the desired diffeomorphism $\phi$. In particular, the estimates in Equation \ref{S_2_Eqn_NormalEst} hold. The estimate for $k\geq K$ requires a version of the interpolation formula, and we postpone this to Section \ref{S_4}.
		  	\end{proof}

\section{Ansatz for gluing model}\label{S_3}%%%%%%%%%%%%
	In this section, we construct approximate solutions starting from a $\Z_{2}$-harmonic $1$-form that has only branching and regular zeros. The strategy is to replace each regular zero with a rescaled model of a non-degenerate $\Z_{2}$-harmonic $1$-form on $\R^{n}$. We begin by collecting known results about $\Z_{2}$-harmonic $1$-forms with linear growth in $\R^{n}, n\geq 3$, with an emphasis on their asymptotic behavior at infinity. We then define a weighted H\"older space adapted to the geometry of the approximate solution.

	\subsection{Non-degenerate $\Z_2$-harmonic $1$-forms on $\R^{n}$}%%%%%%%%
		Let $\Sigma$ be a compact embedded codimension-$2$ submanifold in $\R^{n}$ and let $L$ be a flat line bundle on $\R^{n}\setminus \Sigma$ with monodromy $-1$ along any loop that links $\Sigma$. Recall that $\alpha \in \Gamma(\R^{n}\setminus \Sigma, T^{*}\R^{n}\otimes  L)$ is a $\Z_{2}$-harmonic $1$-form if it satisfies $\ud \alpha=\ud^{*}\alpha=0$. Suppose $B$ is a sufficiently large ball that contains $\Sigma$, then there are two choices of trivialization of $L$ over $\R^{n}\setminus B$, differing by multiplication by $-1$.\\
		
		Fixing a trivialization, we can write the $\Z_{2}$-harmonic $1$-form as
			\begin{equation*}
				\alpha=\ud f, \text{ on } \R^{n}\setminus B,
			\end{equation*}
			for a harmonic function $f$. We are particularly interested in the case where $|\alpha|$ grows linearly at infinity. Using separation of variables in spherical and radial directions, we conclude that
				\begin{equation*}
					f=P+O'(|x|^{2-n}),
				\end{equation*}
			where $P$ is a harmonic polynomial of degree $2$. One way to construct $\Z_{2}$-harmonic $1$-forms is to take the differential of a $\Z_{2}$-harmonic function. When $\Sigma$ is fixed, it was shown by Weifeng Sun \cite{Weifeng22} and later Haydys, Mazzeo and Takahashi \cite{Haydiy23} that there is a $1$-$1$ correspondence between $\Z_{2}$-harmonic functions and harmonic polynomials on $\R^{n}$.
			
			\begin{prop}\label{S_3_Prop_Classification}[Theorem 0.4 \cite{Weifeng22}]
				If $\Sigma$ is a compact, smoothly embedded, oriented codimension-$2$ submanifold in $\R^{n}, n\geq 3$. Then there is a $1$-$1$ correspondence (up to a sign) between a $\Z_{2}$-harmonic function $f$ that branches along $\Sigma$ with polynomial growth and a harmonic polynomial $P$. The bijection is given by
					\begin{equation*}
						f-P\to 0, |\bs{x}|\to \infty.
					\end{equation*} 
			In particular, the difference is bounded by $O'(|\bs{x}|^{2-n})$.
			\end{prop}
			
		The $\Z_{2}$-harmonic functions above are of $O(r^{\frac{1}{2}})$ near $\Sigma$, where $r$ is the distance to $\Sigma$. But for our purpose, that the differentials are $\Z_{2}$-harmonic $1$-forms, one needs the functions in a more restricted class, say, they will be $O(r^{\frac{3}{2}})$ near $\Sigma$. In the existence problem, such conditions are not satisfied a priori, and finding the branching set must be taken into consideration.\\
		
		Examples of $\Z_{2}$-harmonic functions, whose differentials are non-degenerate $\Z_{2}$-harmonic $1$-forms, were recently discovered by Donaldson in $\R^{3}$ using the twistor method \cite{Donaldson25B} and later by the author in $\R^{n}, n\geq 3$ \cite{Yan25} by solving the Laplacian equation in ellipsoidal coordinates and also taking a certain limit of Lawlor's necks.
			\begin{prop}\label{S_3_Prop_Construction}
				For any positive numbers $h_{1},\cdots, h_{n-1}$, there exists a non-degenerate $\Z_{2}$-harmonic function $f_{\bs{h}}$ on $\R^{n}$, whose branching set is a codimension-$2$ ellipsoid 
				\begin{equation*}
					E_{\bs{h}}: \sum_{i=1}^{n-1}\frac{x_{i}^{2}}{h_{i}^{2}}=1, x_{n}=0,
				\end{equation*}
				and such that $\ud f_{\bs{h}}\neq 0$ outside $E_{\bs{h}}$, and that at a large distance, we can pick a single-valued branch of $f_{\bs{h}}$, on which
				\begin{equation*}
					f_{\bs{h}}=a_{0}-\sum_{i=1}^{n}a_{i}x_{i}^{2}+O'(|\bs{x}|^{2-n}).
				\end{equation*}
				Here, let $S(y)=\prod_{i=1}^{n-1}(y+h_{i}^{2})$ and $a_{i}$ be constants given by
				\begin{equation*}
					\begin{split}
						&a_{i}=\frac{\prod_{j=1}^{n-1}h_{j}}{2}\int_{0}^{\infty}\frac{\ud u}{(u^{2}+h_{i}^{2})\sqrt{S(u^{2})}}, 1\leq i\leq n-1;\\
						&a_{n}=-\frac{\prod_{j=1}^{n-1}h_{j}}{2}\int_{0}^{\infty}\frac{S'(u^{2})\ud u}{S(u^{2})^{3/2}};\\
						&a_{0}=\frac{\prod_{j=1}^{n-1}h_{j}}{2}\int_{0}^{\infty}\frac{\ud u}{\sqrt{S(u^{2})}}.
					\end{split}
				\end{equation*}
				Moreover, the map $(h_{1},\cdots, h_{n-1})\mapsto (a_{1},\cdots,a_{n-1})$ from $(\R_{+})^{n-1}$ to $(\R_{+})^{n-1}$ is bijective.
			\end{prop}

		Although the above examples are the only known non-degenerate $\Z_{2}$-harmonic $1$-forms with linear growth and a compact branching set on $\R^{n}$, they are sufficient for the geometric applications in $3$-manifolds, which are our primary focus. In a slightly different direction, viewing $\Z_{2}$-harmonic $1$-forms as infinitesimal branched deformations of special Lagrangians over $\R^{n}$, based on the uniqueness results for exact special Lagrangians with index $1$ and $n-1$ \cite{Imagi16}, our examples appear to encompass all exact non-degenerate $\Z_{2}$-harmonic 1-forms with index $1$ and $n-1$.\\

		On the other hand, it is possible that other types of compact codimension-$2$ submanifolds may support non-degenerate $\Z_{2}$-harmonic $1$-forms with linear growth. In cases where the index is $1$ or $n-1$, such forms may not be exact. In this paper, we carry out our gluing construction starting from $\Z_{2}$-harmonic $1$-forms of arbitrary index, with the hope that more examples will be discovered in future work.
		
	\subsection{Model solutions and weighted H{\"o}lder spaces}\label{S_3_2}%%%%%%%%
		In this subsection, we will construct the approximate solution for our gluing problem and study the mapping properties of the Laplacian in some suitable weighted H\"older spaces.\\
	
		Suppose $\alpha$ is a non-degenerate $\Z_{2}$-harmonic $1$-form on $(M,g)$ with branching set $\Sigma_{0}$ and let $p$ be one of its zeros. Denote its branching set as $\Sigma_{0}$ Consider the normal coordinate on a small geodesic ball of radius $2\delta$ $B_{2\delta}$ centered at $p$. Similar to Proposition \ref{S_2_Prop_Metric_Admissible} , in this coordinate, the Riemannian metric
			\begin{equation*}
				g_{ij}=\delta_{ij}-\frac{1}{3}R_{ikjl}x^{k}x^{l}+O(|x|^{3}).
			\end{equation*}
			By direct computation, the scalar Laplacian can be written as
			\begin{equation*}
				\Delta_{g}=\Delta_{\R^{n}}+\mathcal{K},
			\end{equation*}
			where $\mathcal{K}$ is a sum of products of at most two vector fields that vanish on $p$.\\
			
		Write $\alpha=\ud f_{p}$ on $B_{\delta}$ for a harmonic function $f_{p}$ with $f_{p}(p)=0$. Take Taylor's expansion to the third order
			\begin{equation*}
				f_{p}=P_{p,2}(x)+O(|x|^{3}).
			\end{equation*}
		Here, $P_{2}, P_{3}$ are homogeneous polynomials of degree $2$ and $3$ respectively. Meanwhile, the Taylor's expansion of $\Delta_{g}f$ is
			\begin{equation*}
				\Delta_{g}f_{p}=\Delta_{\R^{n}}(P_{p,2})+O(|x|).
			\end{equation*}
		We conclude that both $P_{2}$ and $P_{3}$ are harmonic with respect to the flat metric on $\R^{n}$. Moreover, if $p$ is a regular zero of $\alpha$, then $P_{p,2}$ defines a non-degenerate quadric. We can pick a coordinate so that
			\begin{equation*}
				P_{p,2}=\sum_{k=1}^{n-m}a_{k}(x^{k})^{2}-\sum_{l=n-m+1}^{n}a_{l}(x^{l})^{2},\quad a_{k},a_{l}>0.
			\end{equation*}

		Suppose $\{p_{i}\}$ is a subset of the set of regular zeros of $\alpha$, and for each $p_{q}$, there exists a non-degenerate $\Z_{2}$-harmonic $1$-form $\alpha_{q}$ with a compact branching set $\Sigma_{q}$ on $\R^{n}$ and
			\begin{equation*}
				\begin{split}
					& \alpha_{q}(y)=\ud h_{q}(y), |y|\gg 1,\\
					& \alpha(x)=\ud f_{q}(x), |x|\ll 1.
				\end{split}
			\end{equation*}
			where
			\begin{equation*}
				\begin{split}
					& h_{q}(y)=P_{i,2}(y)+O'(|y|^{2-n}), |y|\gg 1,\\
					& f_{q}(x)=P_{i,2}(x)+P_{i,3}(x)+O(|x|^{4}), |x|\ll 1.
				\end{split}
			\end{equation*}

		Let $\epsilon$ be the gluing parameter that will go to $0$. Write $\Psi_{q}$ as  a diffeomorphism from $B_{2\delta}(p_{q})$ to a large open set in $\R^{n}$ defined by
			\begin{equation*}
				\Psi_{q}:=\epsilon^{-1}\exp_{p_{q}}^{-1}.
			\end{equation*}
		It is a composition of rescaling and the normal coordinate at $p_{q}$. We define $\Sigma_{\epsilon}$ to be the disjoint union
			\begin{equation*}
				\Sigma_{0}\sqcup\big(\bigsqcup_{q}\Psi_{q}^{-1}\Sigma_{q}\big).
			\end{equation*}
		When $\epsilon\to 0$, $\Psi^{*}_{q}\Sigma_{q}$ will shrink to $p_{q}$. Now define $L_{\epsilon}$ to be the flat line bundle on $M\setminus \Sigma_{\epsilon}$ with monodromy $-1$ around each component.\\

		We now construct an approximately harmonic section on $L_{\epsilon}\otimes T^{*}M$. Let $\gamma_{1}$ be a smooth non-decreasing function such that
			\begin{equation*}
				\gamma_{1}(s)=\begin{cases}
					0, & s<1,\\
					1, & s>2,
				\end{cases}
			\end{equation*}
		and $\gamma_{2}=1-\gamma_{1}$. Fix $0<\sigma\ll 1$ to be a sufficiently small constant independent of $\epsilon$ and write $\rho_{q}$ as the distance to the regular zero $p_{q}$. We define
			\begin{equation}\label{S_3_Eqn_Model_form}
				\tilde{\alpha}_{\epsilon}=
					\begin{cases}
						\alpha, & \rho_{q}>2\epsilon^{1-\sigma},\\
						\epsilon^{2}\Psi_{q}^{*}(\alpha_{q}), & \rho_{q}\leq \epsilon^{1-\sigma},\\
						\ud \big(\gamma_{1}(\epsilon^{\sigma-1}\rho_{q})f_{q}+\epsilon^{2}\gamma_{2}(\epsilon^{\sigma-1}\rho_{q})\Psi_{q}^{*}(h_{q})\big), & \epsilon^{1-\sigma}<\rho_{q}\leq 2\epsilon^{1-\sigma}.
					\end{cases}
			\end{equation}
		We also define a metric that is close to $g$
			\begin{equation*}
				g_{\epsilon}=
					\begin{cases}
						g, & \rho_{q}>2\epsilon^{1-\sigma},\\
						\gamma_{2}(\epsilon^{\sigma-1}\rho_{q})\epsilon^{2}\Psi_{q}^{*}(g_{\R^{n}})+\gamma_{1}(\epsilon^{\sigma-1}\rho_{q})g, & \rho_{q}\leq 2 \epsilon^{1-\sigma}.
					\end{cases}
			\end{equation*}
		It is easy to see that $g_{\epsilon}$ is a Euclidean metric on $\{\rho_{q}<\epsilon^{1-\sigma}\}$. It follows from the construction that $\tilde{\alpha}_{\epsilon}$ is closed, and $\ud \star_{g_{\epsilon}}\tilde{\alpha}_{\epsilon}$ is supported in the region $\{\epsilon^{1-\sigma}<\rho_{q}<2\epsilon^{1-\sigma}\}$.\\

		For $\delta$ sufficiently small, $N$ sufficiently large, and $\epsilon\to 0$, we define a weight function $\rho_{\epsilon}$ as follows: we set
			\begin{equation*}
				\rho_{\epsilon}=
					\begin{cases}
						\epsilon, & \rho_{q}<N\epsilon,\\
						\frac{1}{2\delta}\rho_{q}, & 2N\epsilon<\rho_{q}<\delta,\\
						1, & \rho_{q}>2\delta,
					\end{cases}
			\end{equation*}
		and let $\rho_{\epsilon}$ interpolate smoothly and monotonically between various regions.
		
		\begin{definition}\label{S_3_Defn_W_Holder_1}
			For each $\nu\in \R$, $k \in \N$, and $\alpha\in (0,\frac{1}{2})$, we define the weighted H\"older norm $C^{k,\alpha}_{\nu,\epsilon}$ by
				\begin{equation*}
					\|a\|_{C^{k,\alpha}_{\nu,\epsilon}}:=\sum_{j=0}^{k}\|\rho_{\epsilon}^{-\nu+j}\nabla^{j}a\|_{C^{0}}+\sup_{d(x_{1},x_{2})<\mathrm{inj}{g_{\epsilon}}}\min_{i=1,2}{\big(\rho_{\epsilon}(x_{i})^{-\nu+k+\alpha}\big)}\frac{|\nabla^{k}a(x_{1})-\nabla^{k}a(x_{2})|}{d(x_{1},x_{2})^{\alpha}}.
				\end{equation*}
			Here, the norms and covariant derivatives are taken with respect to the reference metric $g_{\epsilon}$ and $\nabla^{k}a(x_{1}), \nabla^{k}a(x_{2})$ are compared using the parallel transport along the shortest geodesic connecting $x_{1}$ and $x_{2}$. Similarly, we define weighted $C^{k}$ norms to be
				\begin{equation*}
					\|a\|_{C^{k}_{\nu,\epsilon}}:=\sum_{j=0}^{k}\|\rho_{\epsilon}^{-\nu+j}\nabla^{j}a\|_{C^{0}}.
				\end{equation*}
		\end{definition}
		
		We now consider the difference between $g_{\epsilon}$ and $g$. It follows from the construction that the difference is only supported in the annuli $\{\epsilon^{1-\sigma}<\rho_{i}<2\epsilon^{1-\sigma}\}$ and can be computed using geodesic coordinates at $p_{q}$
			\begin{equation*}
				\gamma_{ij}:=(g_{\epsilon})_{ij}-g_{ij}=\gamma_{2}(\epsilon^{\sigma-1}\rho_{q})\bigg(\frac{1}{3}R_{ikjl}(p_{q})x^{k}x^{l}+O(|x|^{3})\bigg).
			\end{equation*}
		If $\nu<0$, direct computation implies
			\begin{equation*}
				\|\gamma\|_{C^{k,\alpha}_{\nu,\epsilon}}\leq C_{k}\epsilon^{(2-\nu)(1-\sigma)},
			\end{equation*}
		On the other hand, suppose $\eta'\in \bigotimes^{2}T^{*}M$ and
			\begin{equation*}
				\|\gamma'\|_{C^{k,\alpha}_{\nu,\epsilon}}\leq C_{k}',
			\end{equation*}
		then
			\begin{equation*}
				\|\gamma'\|_{C^{k,\alpha}_{0,\epsilon}}\leq C_{k}' \epsilon^{\nu}.
			\end{equation*}

		In order to apply Nash-Moser theory, our next goal is to define H\"older norms on the space of sections $\Gamma(M\setminus \Sigma_{\epsilon}, L_{\epsilon})$ with respect to a small perturbation of the reference metric $g_{\epsilon}$. Recall that $\mathcal{M}$ is the set of Riemannian metrics and define $\mathcal{M}_{\epsilon}^{0}$ to be the set of metrics that are compatible with the normal structure on $\Sigma_{\epsilon}$ induced by $g_{\epsilon}$.
			
			\begin{definition}\label{S_3_Defn_W_Diff}
				We define $\mathrm{Diff}^{0}_{\epsilon}$ as the set of diffeomorphisms on $M$ generated by vector fields, with the norms on the vector fields given by $\|\cdot\|_{C^{k,\alpha}_{\nu,\epsilon}}$. 
			\end{definition}
			We also have a global version of Lemma \ref{S_2_Lemma_Deformation} that states that for each small perturbation $(\widetilde{\Sigma},\tilde{g})$ of the branching set and the metric $(\widetilde{\Sigma}_{\epsilon},g_{\epsilon})$, we can find a diffeomorphism $\Phi(\widetilde{\Sigma},\tilde{g})$ such that
			\begin{equation*}
				\Phi(\widetilde{\Sigma},\tilde{g})^{*}\tilde{g}\in \mathcal{M}^{0}_{\epsilon},\; \Phi(\widetilde{\Sigma},\tilde{g})^{-1}\widetilde{\Sigma}=\Sigma_{\epsilon},
			\end{equation*}
			and for a suitable choice of $\nu$, the $C^{k}_{\nu,\epsilon}$ norms of $\Phi-Id$ are controlled by the deformation. The proof needs a weighted version of the tame estimate, which we postpone until Section \ref{S_4}. In this section, we only consider the deformation of the metric in the space $\mathcal{M}^{0}_{\epsilon}$.\\
		
		Let $d>0$ be a sufficiently small constant. We define
			\begin{equation}\label{S_3_Eqn_U_0}
				U_{0,d}:=\{x|\mathrm{dist}_{g_{\epsilon}}(x,\Sigma_{0})<d\}
			\end{equation}		
		 to be a small tubular neighborhood of $\Sigma_{0}$. Let
		 	\begin{equation}\label{S_3_Eqn_V_q}
		 		V_{q,d}:=\{y\in \R^{n}|\mathrm{dist}_{\R^{n}}(y,\Sigma_{q})<d\}
		 	\end{equation}
		 and define
		 	\begin{equation}\label{S_3_Eqn_U_q}
		 		U_{q,d}:=\Psi_{q}^{-1}\big(V_{q,d}\big).
		 	\end{equation}
		 Moreover, for simplicity, we denote $M_{r}$ as $M\setminus \big(U_{0,r}\bigcup \sqcup_{q} U_{q,r}\big)$. We now give a weighted version of the spaces of functions defined in Definition \ref{S_2_Defn_DCHolder}
		 Suppose from now on $\tilde{g}=g_{\epsilon}+\gamma\in \mathcal{M}^{0}_{\epsilon}$, with $\|\gamma\|_{C^{3K,\alpha}_{\nu,\epsilon}}\ll 1$ for a fixed large $K$.

		\begin{definition}\label{S_3_Defn_W_Holder_2}
			For each $\nu \in \R$ and $k\in \N$, $\alpha\in (0,\frac{1}{2})$, we define the weighted H\"older norms $\|\cdot\|_{\mathcal{D}^{k,\alpha}_{\nu,\epsilon}}$ and $\|\cdot\|_{\mathcal{C}^{k+2,\alpha}_{\nu,\epsilon}}$ with respect to the metric $g_{\epsilon}$ for the section on $L_{\epsilon}$
				\begin{equation*}
					\begin{split}
					&\|s\|_{\mathcal{D}^{k,\alpha}_{\nu,\epsilon}}=\|s\|_{C^{k+2,\alpha}_{\nu,\epsilon}(M_{d})}+\|s\|_{\mathcal{D}^{k,\alpha}_{g_{\epsilon}}(U_{0,2d})}+\sum_{q}\|\epsilon^{-\nu}s\|_{\mathcal{D}^{k,\alpha}_{\epsilon^{-2}g_{\epsilon}}(U_{q,2d})}.\\
					& \|s\|_{\mathcal{C}^{k+2,\alpha}_{\nu,\epsilon}}=\|s\|_{C^{k+2,\alpha}_{\nu,\epsilon}(M_{d})}+\|s\|_{\mathcal{C}^{k+2,\alpha}_{g_{\epsilon}}(U_{0,2d})}+\sum_{q}\|\epsilon^{-\nu}s\|_{\mathcal{C}^{k+2,\alpha}_{\epsilon^{-2}g_{\epsilon}}(U_{q,2d})},\\
					& \|s\|_{C^{k}_{\alpha,\nu,\epsilon}}=\|s\|_{C^{k}_{\nu,\epsilon}(M_{d})}+\|s\|_{C^{k}_{g_{\epsilon},\alpha}(U_{0,2d})}+\sum_{q}\|\epsilon^{-\nu}s\|_{C^{k}_{\epsilon^{-2}g_{\epsilon},\alpha}(U_{q,2d})}.
					\end{split}
				\end{equation*}
			Let $E$ be a vector bundle over $\Sigma_{\epsilon}$, here $E$ can be $N\Sigma_{\epsilon}^{\frac{p}{2}}$, $p\in \Z$. We define the H\"older norm $\|\cdot\|_{C^{k,\beta}_{\nu,\epsilon}(\Sigma_{\epsilon},E)}$ to be
				\begin{equation*}
					\|v\|_{C^{k,\beta}_{\nu,\epsilon}(\Sigma_{\epsilon},E)}=\|v\|_{C^{k,\beta}(\Sigma_{0}, E, g_{\epsilon})}+\sum_{q}\|\epsilon^{-\nu}v\|_{C^{k,\beta}(\Sigma_{q}, E, \epsilon^{-2}g_{\epsilon})}.
				\end{equation*}
			Here, the norm on $E$ is induced by the metric $\rho_{\epsilon}^{-2}g_{\epsilon}$.
		\end{definition}		
		
		Similarly to Definition \ref{S_2_Defn_AB}, we define the map
		\begin{definition}\label{S_3_Defn_Sheaf_Map}
			For $\alpha\in (0,\frac{1}{2})$, we define two maps
				\begin{equation*}
					\begin{split}
						&\underline{A}^{\epsilon}:\mathcal{C}^{k+2,\alpha}_{\nu,\epsilon}\to C^{k+1,\alpha+\frac{1}{2}}_{\nu-2,\epsilon}(\Sigma_{\epsilon}, N\Sigma_{\epsilon}^{-\frac{1}{2}}),\\
						&\underline{B}^{\epsilon}:\mathcal{C}^{k+2,\alpha}_{\nu,\epsilon}\to C^{k,\alpha+\frac{1}{2}}_{\nu-2,\epsilon}(\Sigma_{\epsilon}, N\Sigma_{\epsilon}^{-\frac{3}{2}}),
					\end{split}
				\end{equation*}
			in the following ways. Let $s\in \mathcal{C}^{k+2,\alpha}_{\tau,\epsilon}$ define
				\begin{equation*}
					\underline{A}_{\epsilon}(s)=\underline{A}(s),\underline{B}_{\epsilon}(s)=\underline{B}(s),
				\end{equation*}			
			on $\Sigma_{0}$ with respect to $g_{\epsilon}$, and define
				\begin{equation*}
					\underline{A}_{\epsilon}(s)=\epsilon^{-2}\underline{A}(s),\underline{B}_{\epsilon}(s)=\epsilon^{-2}\underline{B}(s),
				\end{equation*}
			on $\Psi_{q}^{-1}\Sigma_{q}$ with respect to the metric $\epsilon^{-2}g_{\epsilon}$.
		\end{definition}

\subsection{Schauder estimates and elliptic theory}		
	In what follows, we turn to establish the Schauder estimates and elliptic theory for the Laplacian operator on the line bundle $L_{\epsilon}$ over $M\setminus\Sigma_{\epsilon}$.
		\begin{lemma}\label{S_3_Lemma_Schauder}
			Fix a large $K>0$ for $\tau\in \R$, $k\in \N$, and $\alpha\in (0,\frac{1}{2})$; for each $\tilde{g}=g_{\epsilon}+\gamma\in \mathcal{M}^{\epsilon}_{0}$ with $\|\gamma\|_{C^{3K,\alpha}_{0,\epsilon}}\ll 1$, there exists a constant $C_{k}$ independent of $\epsilon$ such that
				\begin{equation*}
					\|s\|_{\mathcal{C}^{k+2,\alpha}_{\tau,\epsilon}}<C_{k}\big(1+\|\gamma\|_{C^{k+K,\alpha}_{0,\epsilon}}\big)^{k}\big(\|\widetilde{\Delta}_{\tilde{g}}s\|_{\mathcal{D}^{k,\alpha}_{\tau-2,\epsilon}}+\|s\|_{C^{0}_{\tau,\epsilon}(M_{r})}\big).
				\end{equation*}
		\end{lemma}
		\begin{proof}
			This is proved by applying the usual Schauder's estimate in each model space.
		\end{proof}

		\begin{prop}\label{S_3_Prop_linear_est}
			Let $K,k,\alpha$ as in Lemma \ref{S_3_Lemma_Schauder}, if $\tau<0$, then
				\begin{equation*}
					\|s\|_{\mathcal{C}^{k+2,\alpha}_{\tau,\epsilon}}<C_{k}\big(1+\|\gamma\|_{C^{k+K,\alpha}_{0,\epsilon}}\big)^{k}|\widetilde{\Delta}_{\tilde{g}}s\|_{\mathcal{D}^{k,\alpha}_{\tau-2,\epsilon}}.
				\end{equation*}
		\end{prop}
		\begin{proof}
			Follow the Lemma \ref{S_3_Lemma_Schauder}, it suffices to prove the statement for a fixed $k=K$ and with respect to the reference metric
				\begin{equation*}
					g'=g_{(\delta/4)^{1/(1-\sigma)}},
				\end{equation*}
			which is flat in every geodesic ball $B_{p_{q}}(\frac{\delta}{2})$. In addition to this, $\|g'-g_{\epsilon}\|_{C^{K,\alpha}_{0,\epsilon}}$ is bounded by some constant depending only on $\delta$, $K$, and the original metric $g$. By contradiction, assume that there exists a sequence of $\epsilon_{i}\to 0$ and a section $s_{i}$  on $L_{\epsilon_{i}}$ such that $\|s_{i}\|_{\mathcal{C}^{K+2,\alpha}_{\tau,\epsilon_{i}}}=1$ and $\|\widetilde{\Delta}_{g'}s_{i}\|_{\mathcal{D}^{K,\alpha}_{\tau-2,\epsilon_{i}}}\to 0$.\\
			
			First of all, we show that for every sufficiently small ball $B_{p_{q}}(\delta')$ must have $\|s_{i}\|_{\mathcal{C}^{K+2,\alpha}_{\tau,\epsilon_{i}}(B_{p_{q}}(\delta'))}\to 0$. To achieve this, we consider a blow-up model in $\R^{n}$. Here, we take $g_{0}=\epsilon_{i}^{-2}g'$ as the Euclidean metric on $\R^{n}$. We define
				\begin{equation*}
					e_{i}(y)=\gamma_{2}\big(\epsilon_{1}|y|/(2\delta')\big)\epsilon^{\tau} s_{i}(\epsilon_{i} y).
				\end{equation*}
			By construction, $e_{i}(y)$ is a section on $L_{\Sigma_{q}}$ on $\R^{n}$, which supports $\{|y|<2\delta'\epsilon_{i}^{-1}\}$ and is equal to the blow-up model in $\{|y|<\delta\epsilon_{i}^{-1}\}$.\\
			
			Now, the limit of $\epsilon^{-1}\rho_{\epsilon}(\epsilon y)$, define a smooth weight function $\rho$ on $\R^{n}$. We define $\widetilde{\mathcal{C}}^{K+2,\alpha}_{\tau}$ as a rescaling limit of the weighted H\"older space $\mathcal{C}^{K+2,\alpha}_{\tau,\epsilon_{i}}$, and similarly define $\widetilde{\mathcal{D}}^{K,\alpha}_{\tau-2}$ from $\mathcal{D}^{K,\alpha}_{\tau-2,\epsilon_{i}}$. We claim that on $L_{\Sigma_{q}}$, there are estimates
				\begin{equation*}
					\|h\|_{\widetilde{\mathcal{C}}^{K+2,\alpha}_{\tau}}\leq C_{K}\|\widetilde{\Delta}_{g_{0}}h\|_{\widetilde{\mathcal{D}}^{K,\alpha}_{\tau-2}}.
				\end{equation*}
			Otherwise, we can find a sequence of $h_{i}$, with $\|h_{i}\|_{\widetilde{\mathcal{C}}^{K+2,\alpha}_{\tau}}=1$, but $\|\widetilde{\Delta}_{g_{0}}h_{i}\|_{\widetilde{\mathcal{D}}^{K,\alpha}_{\tau-2}}$ goes to $0$. A standard argument shows that $W_{g_{0}}h_{i}$ converges to a $\Z_{2}$-harmonic function $W_{g_{0}}h_{\infty}$ in $\widetilde{\mathcal{C}}^{K+2,\bar{\alpha}}_{\tau}$ for $\bar{\alpha}<\alpha$. Then the Proposition \ref{S_3_Prop_Classification} of the corresponding $\Z_{2}$-harmonic functions and harmonic polynomials implies $h_{\infty}(y)=0$. Combining with Lemma \ref{S_3_Lemma_Schauder}, we show that $h_{i}$ converges to $0$ in $\widetilde{\mathcal{C}}^{K+2,\alpha}_{\tau}$, which is a contradiction. Apply the claim to the sections $e_{i}(y)$ and transfer the estimates back to $M$, we conclude that
				\begin{equation*}
					\|s_{i}\|_{\mathcal{C}^{K+2,\alpha}_{\tau,\epsilon_{i}}(B_{p_{q}}(\delta'))}\leq C_{k} \|\widetilde{\Delta}_{g'}s_{i}\|_{\mathcal{D}^{K,\alpha}_{\tau,\epsilon_{i}}(B_{p_{q}}(2\delta'))}\to 0.
				\end{equation*}

			The above discussion also implies that there exists a sequence of positive numbers $\lambda_{i}\to 0$ such that
				\begin{equation*}
					|s_{i}(x)|\leq \lambda_{i} \rho(x)^{\tau},
				\end{equation*}
			when $x\in B_{p_{q}}(\frac{1}{2}\delta)$. We show that for every compact set $B$ that does not contain $\{p_{q}\}$ also has $\|s_{i}\|_{\mathcal{C}^{K+2,\alpha}_{\tau,\epsilon_{i}}(B)}\to 0$. Since $g_{\epsilon}=g'$ on $B$ when $\epsilon$ is small, we conclude that $W(g')s_{i}$ converges to a $\Z_{2}$-harmonic function $W(g')s$ in $\mathcal{C}^{K+2,\bar{\alpha}}_{\tau,\epsilon_{i}}(B)$, for every $B$ and $\bar{\alpha}<\alpha$. Passing to the branched cover $M_{\Sigma_{0}}$ of $M$ along $\Sigma_{0}$. Let $p_{q}',p_{q}''$ be the pre-images of $p_{q}$ under the branching map. Then $W(g')s$ is a harmonic function on $M_{\Sigma_{0}}\setminus \big(\Sigma_{0}\cup \{p_{q}',p_{q}''\}\big)$. Moreover, $W(g')s$ goes to $0$ as it approaches to the boundary. The maximum principle implies that the limit $s=0$. As a result, this contradicts the normalization assumption $\|s_{i}\|_{\mathcal{C}^{K+2,\alpha}_{\tau,\epsilon_{i}}}=1$.
		\end{proof}

	Recall that $\tilde{\alpha}_{\epsilon}$ is the approximate solution constructed in \ref{S_3_Eqn_Model_form}, and let $\tilde{g}\in \mathcal{M}_{0}^{\epsilon}$ be a small perturbation of the metric. It is easy to check that the error $W^{-1}(\tilde{g})\ud^{*}_{\tilde{g}}\tilde{\alpha}_{\epsilon}$ lies in $L^{2}(M\setminus \Sigma_{\epsilon},L_{\epsilon})$. As a result, there is a unique solution to Poisson's equation
		\begin{equation}\label{S_3_Eqn_Poisson}
			-\widetilde{\Delta}_{\tilde{g}}s_{\epsilon}(\tilde{g})+W^{-1}(\tilde{g})\ud^{*}_{\tilde{g}}\tilde{\alpha}_{\epsilon}=0.
		\end{equation}
	Then $\ud \big(W(\tilde{g})s_{\epsilon}(\tilde{g})\big)+\tilde{\alpha}_{\epsilon}$ is a harmonic section on $L_{\epsilon}$. It should be noted here that this section may not be $\Z_{2}$-harmonic, since its asymptotic behavior near the branching set is
		\begin{equation*}
			\ud \big(W(\tilde{g})s_{\epsilon}(\tilde{g})\big)+\tilde{\alpha}_{\epsilon}\sim \frac{1}{2}\RRe{\underline{A}(s_{\epsilon}(\tilde{g}))\zeta^{-\frac{1}{2}}\ud \zeta}.
		\end{equation*}
	
	The key problem in this paper is to find a suitable $\tilde{g}$ such that $\underline{A}(s_{\epsilon}(\tilde{g}))$ vanishes when $\epsilon$ is sufficiently small. Before delving into this analysis, we first establish a regularity result about $s_{\epsilon}(\tilde{g})$.

		\begin{lemma}\label{S_3_Lemma_Error_gluing}
			Let $\tau<0$, and suppose $s_{\epsilon}(\tilde{g})$ is the solution to the Poisson equation \ref{S_3_Eqn_Poisson}, then
				\begin{equation*}
					\|s_{\epsilon}\|_{_{\mathcal{C}^{k+2,\alpha}_{\tau,\epsilon}}}\leq C_{k}(1+\|\gamma\|_{C^{k+K,\alpha}_{0,\epsilon}})^{k}\big(\|\gamma\|_{C^{k+K,\alpha}_{-2+\tau,\epsilon}}+\epsilon^{n\sigma+(1-\sigma)(2-\tau)}\big)
				\end{equation*}
		\end{lemma}
		\begin{proof}
			Apply the Proposition \ref{S_3_Prop_linear_est}, it remains to estimate
				\begin{equation*}
					Err=W^{-1}(\tilde{g})\ud_{\tilde{g}}^{*}\tilde{\alpha}_{\epsilon}
				\end{equation*}
			in $\mathcal{D}^{k,\alpha}_{\tau-2,\epsilon}$. The error term is contributed by the gluing error and the perturbation of the metric. Observe that when $\gamma=0$, $\ud ^{*}_{g_{\epsilon}}\tilde{\alpha}_{\epsilon}$ supports only on the annular regions $\{\epsilon^{1-\sigma}<\rho_{q}<2\epsilon^{1-\sigma}\}$. Therefore, in the complement of those regions, in particular near the branching set $\Sigma_{\epsilon}$, the error is only contributed by the metric deformation.\\
			
			We begin by estimating the error on  $M_{d}$. Fix a small constant $a>0$ and $x_{i}\in M_{d}$, and in $B_{p_{q}}(\delta)$, consider a coordinate change
				\begin{equation*}
					x=x_{i}+\rho_{\epsilon}(x_{i})y, |y|<a,
				\end{equation*}
			and a rescale of metric $g_{i}=\rho_{\epsilon}(x_{i})^{-2}g_{\epsilon}$ and the difference $\eta_{i}=\rho_{\epsilon}(x_{i})^{-2}\gamma$.\\
				
			When $x_{i}$ lies in the gluing region $\{\frac{1}{2}\epsilon^{1-\sigma}<\rho_{q}<4\epsilon^{1-\sigma}\}$, then on a single-valued branch of $\rho_{\epsilon}(x_{i})^{-1}\tilde{\alpha}_{\epsilon}$ is exact, and its primitive can be written as
				\begin{equation*}
					\gamma_{1} \epsilon^{2}h_{q}\big(\epsilon^{-1}(x_{i}+\rho_{\epsilon}(x_{i})y)\big)+\gamma_{2}f_{q}(x_{i}+\rho_{\epsilon}(x_{i})y).
				\end{equation*}
			The above expression can be expanded into $err+f_{q}(x_{i}+\rho_{\epsilon}(x_{i})y)$, where
				\begin{equation*}
					err=\gamma_{1} \epsilon^{n} O'(|x_{i}+\rho_{\epsilon}(x_{i})y|^{2-n})+\gamma_{2}(P_{3}(x)+O(|x|^{4})).
				\end{equation*}
			Since $f_{q}$ is harmonic with respect to the metric $g$, we have
				\begin{equation*}
					\|\Delta_{g_{i}+\eta_{i}}f_{q}\|_{C^{k,\alpha}(g_{i})}\leq C_{k}\big(\rho_{\epsilon}(x_{i})^{2}\|\eta_{i}\|_{C^{k+1,\alpha}(g_{i})}+\|g-g_{\epsilon}\|_{C^{k+1,\alpha}(g_{i})}\big).
				\end{equation*}
			Here, the last term is bounded by $C_{k}\epsilon^{4(1-\sigma)}$. On the other hand, the error term due to $err$ is bounded by
				\begin{equation*}
					|\nabla^{k}_{g_{i}}err|\leq C_{k}(\epsilon^{2+(n-2)\sigma}+\epsilon^{3(1-\sigma)}).
				\end{equation*}
			Suppose $0<\sigma\ll 1$, then in this region
				\begin{equation*}
					\|Err\|_{C^{k,\alpha}(g_{i})}\leq C_{k}\big(\epsilon^{n\sigma}+\|\eta_{i}\|_{C^{k+1,\alpha}}\big).
				\end{equation*}

			When $x_{i}\in \{\frac{1}{2}N\epsilon<\rho_{q}<\epsilon^{1-\sigma}\}$, $\tilde{\alpha}_{\epsilon}=\epsilon\alpha_{q}$, then $Err$ is only contributed by $\eta$. The model form $\tilde{\alpha}_{\epsilon}$ can locally be written as
				\begin{equation*}
					\epsilon a_{q}\big(\epsilon^{-1}(x_{i}+\rho_{\epsilon}(x_{i})y)\big)\rho_{\epsilon}(x_{i})\ud y.
				\end{equation*}
			Due to the separation of variables of harmonic functions on $\R^{n}$, the derivatives of $a_{q}$ are uniformly bounded, and therefore the error term is bounded by
				\begin{equation*}
					\|Err\|_{C^{k,\alpha}(g_{i})}\leq C_{k}\|\eta_{i}\|_{C^{k+1,\alpha}(g_{i})}.
				\end{equation*}
			In the remaining regions in $M_{d}$, the above estimate also holds. Away from the $\Sigma_{\epsilon}$, to conclude, we have
				\begin{equation*}
					\|Err\|_{\mathcal{D}^{k,\alpha}_{\tau-2,\epsilon}(M_{d})}\leq C_{k}(\epsilon^{n\sigma+(1-\sigma)(2-\tau)}+\|\eta\|_{C^{k+1,\alpha}_{\tau-2,\epsilon}(M_{d})}).
				\end{equation*}
			
			We now estimate the error term near the branching set. In the tubular neighborhood $U_{0,2d}$ of $\Sigma_{0}$, we have the following asymptotic expansion
				\begin{equation*}
					\alpha= \ud \bigg(\RRe{B(t)\zeta^{\frac{3}{2}}+E(t,\zeta)}\bigg).
				\end{equation*}
			On the other hand, $\widetilde{\Delta}_{g+\gamma}=\widetilde{\Delta}_{g}+\mathcal{L}_{\gamma}$, where $ \mathcal{L}_{\gamma}\in \mathcal{T}_{2}$ with coefficients $\|\mathcal{F}_{\gamma}\|_{C^{k,\alpha}(U_{0,2d})}\leq C_{k}\|\gamma\|_{C^{k+K,\alpha}(U_{0,2d})}$. Then we have the estimate
				\begin{equation*}
					\|\mathcal{L}_{\gamma}\big(B(t)\zeta^{\frac{3}{2}}+E(t,\zeta)\big)\|_{\mathcal{D}^{k,\alpha}}\leq C_{k}\|B(t)\zeta^{\frac{3}{2}}+E(t,\zeta)\|_{\mathcal{D}^{k+2,\alpha}}\|\gamma\|_{C^{k+K,\alpha}(U_{0,2d})}.
				\end{equation*}
			
			In the tubular neighborhood $U_{q,2d}$ of $\Psi_{q}^{-1}\Sigma_{q}$, a similar argument applies to this case, and we can also show that the $\mathcal{D}^{k,\alpha}_{\tau-2,\epsilon}$ norms of the error are bounded by
				\begin{equation*}
					C_{k}\|\gamma\|_{C^{k+K,\alpha}_{\tau-2,\epsilon}(U_{q,2d})}.
				\end{equation*}
			Gathering the estimates in various regions, we conclude that
				\begin{equation*}
					\|Err\|_{\mathcal{D}^{k,\alpha}_{\tau-2,\epsilon}}\leq C_{k}\big(\|\gamma\|_{C^{k+K,\alpha}_{\tau-2,\epsilon}}+\epsilon^{n\sigma+(1-\sigma)(2-\tau)}\big).
				\end{equation*}
		\end{proof}

\section{Analytic smoothing operators and Nash-Moser theory}\label{S_4}%%%%%%%%%%%%
	In this section, we set up the Nash-Moser theory for the gluing problem, following Hamilton's article \cite{Hamilton82} and Donaldson's exposition. A subtlety here is that we use a weighted version of the Nash-Moser theorem, involving families of graded Fr\'echet spaces with smoothing operators parameterized by $\epsilon$. In this case, the dependence of inverting the smooth tame maps on $\epsilon$ does not follow directly from Hamilton's framework.\\
	
	To address this, the theorem will be divided into two parts. The first part involves embedding the graded Fr\'echet spaces into the product of graded Fr\'echet over $\R^{n}$. This step allows us to define a family of smoothing operators and eventually derive tame estimates. The second part is a Nash-Moser theorem with an approximate inverse established in \cite{Zehnder75,Zehnder76,Raymond89,Donaldson25A} . Crucially, the region where we can invert the smooth tame maps can be made uniform in terms of uniform tame estimates between those spaces, which are established in the first step.\\
	
	Let us recall some notions in Nash-Moser theory. A \textbf{graded Fr\'echet space} is a F\'echet space with an increasing sequence of norms
		\begin{equation*}
			\|f\|_{0}\leq \|f\|_{1}\leq \|f\|_{2}\leq \cdots.
		\end{equation*}
	Two increasing sequences of norms $\{\|\cdot \|_{k}\}$, $\{\|\cdot\|_{k}'\}$ are called \textbf{tamely equivalent} if there are uniform constants $r,r'$ independent of $k$ such that
		\begin{equation*}
			\|f\|_{k}\leq C_{k}\|f\|_{k+r}', \;\|f\|_{k}'\leq C_{k}\|f\|_{k+r'}.
		\end{equation*}
		A linear map $L$ from a graded tame Fr\'echet space to another satisfies a \textbf{linear tame estimate} if there is a fixed $r>0$ and constants $C_{k}$ such that
		\begin{equation*}
			\|L(f)\|_{k}\leq C_{k}\|f\|_{k+r},
		\end{equation*}
	for each $k>k_{0}$. We call $k_{0}$ the \textbf{base} and $r$ the \textbf{degree}.\\ 

	A key notion in Nash-Moser theory is \textbf{tame estimate}. Let $S$ be a map from an open set in the tame Fr\'echet space $X$ to a tame Fr\'echet space $Y$. It satisfies a tame estimate if there is a fixed $r>0$ and constants $C_{k}$ such that
		\begin{equation*}
			\|S(f)\|_{k}\leq C_{k}(1+\|f\|_{k+r}),
		\end{equation*}
	for all sufficiently large $k$ and every $f$ in the domain.\\
	
	\subsection{Smoothing operators and tame estimates}%%%%%%%%
		Recall that $\mathcal{M}^{\epsilon}$ is the set of smooth metrics on $M$, with a graded Fr\'echet structure induced by weighted $C^{k}$ norms $\|\cdot\|_{C^{k}_{\nu,\epsilon}}$ defined in  Definition \ref{S_3_Defn_W_Holder_1}. Let $C^{\infty}(\Sigma,M)$ be the space of smooth embeddings from a manifold $\Sigma$, which is diffeomorphic to $\Sigma_{\epsilon}$, to $M$, and $\mathrm{Diff}$ be a set of diffeomorphisms. For our purposes, we only need to work in a small neighborhood $\mathcal{S}^{\epsilon}$ of $\Sigma_{\epsilon}$ in $C^{\infty}(\Sigma,M)$ and a small neighborhood $\mathcal{V}^{\epsilon}$ of identity in $\mathrm{Diff}$. We endow  $S^{\epsilon}$ with graded Fr\'echet norms $\|\cdot\|_{C^{k}_{\epsilon}(\Sigma_{\epsilon},N\Sigma_{\epsilon})}$, where a small deformation of $\Sigma_{\epsilon}$ is taken to be a section of its normal bundle. We also equip $\mathcal{V}^{\epsilon}$ with graded Fr\'echet norms $\|\cdot\|_{C^{k}_{1,\tau}}$. Here, we identify a diffeomorphism in $\mathcal{V}^{\epsilon}$ with a vector field. In addition, to establish a tame estimate, we denote $\mathcal{E}^{\epsilon}$ to be the space of smooth sections on the real line bundle $L_{\epsilon}$ over $M\setminus \Sigma_{\epsilon}$ equipped with graded Fr\'echet norms $C^{k}_{\alpha,\nu,\epsilon}$. For simplicity, when there is no ambiguity, we use $\|\cdot\|_{k}$ to denote the graded Fr\'echet norms on those various graded Fr\'echet spaces.\\
		
		The main result in this subsection is the existence of smoothing operators.

		\begin{proposition}\label{S_4_Prop_Smoothing}
			Let $X$ be one of the graded Fr\'echet spaces $\mathcal{M}^{\epsilon},\mathcal{S}^{\epsilon}, \mathcal{V}^{\epsilon}, \mathcal{E}^{\epsilon}$ mentioned above. There exists a family of \textbf{smoothing operators}
				\begin{equation*}
					S^{\epsilon}_{\theta}:X\to X, \theta\geq 1,
				\end{equation*}
			such that for every pair $k_{1}\leq k_{2}$ and $x\in X$
				\begin{equation*}
					\begin{split}
						&\|S^{\epsilon}_{\theta}x\|_{k_{2}}\leq C_{k_{1},k_{2}}\theta^{k_{2}-k_{1}}\|x\|_{k_{1}},\\
						&\|(Id_{X}-S^{\epsilon}_{\theta})x\|_{k_{1}}\leq C_{k_{2},k_{1}}\theta^{k_{1}-k_{2}}\|x\|_{k_{2}},\\
						&\lim_{\theta\to \infty}S^{\epsilon}_{\theta}=Id_{X}.
					\end{split}
				\end{equation*}
			Here, the constants $C_{k_{1},k_{2}}$ depend only on $C_{k},m$ and $k_{1},k_{2}$, and are independent of $\epsilon$.
		\end{proposition}
		
		First, we focus on the case $X=\mathcal{M}^{\epsilon}$. Fix a small constant $a>0$, we cover $M$ by geodesic balls of radius $a\rho_{\epsilon}(x)$ with respect to the metric $g_{\epsilon}$. We pick a finite cover, which we label as $\{B_{r_{i}}(x_{i})\}$, such that there is a uniform upper bound for the number of $B_{4r_{i}}(x_{i})$ admitting a nonempty intersection. Let $u_{i}$ be a partition of unity that vanishes outside $B_{3r_{i}}(x_{i})$, and let $v_{i}$ be the bump function that equals $1$ on $B_{3r_{i}}(x_{i})$ and vanishes outside $B_{4r_{i}}(x_{i})$. We identify $\mathcal{M}^{\epsilon}$ as an open set in the space of smooth sections in $Sym^{2}(T^{*}M)$ using
			\begin{equation*}
				\tilde{g}=g_{\epsilon}+\eta, \eta \in \Gamma(Sym^{2}T^{*}M).
			\end{equation*}
			
		Using geodesic coordinates and rescaling by $\rho_{\epsilon}(x_{i})$, we identify $B_{4r_{i}}(x_{i})$ as a Euclidean ball of radius $4a$ centered at the origin, which we denote as $B_{4a,i}$. Then $\rho_{\epsilon}^{-2}(x_{i})u_{i}\eta$ defines a section supported on $B_{4a,i}$. Now let $g_{i}=\rho_{\epsilon}^{-2}(x_{i})g_{\epsilon}(x_{i})$ be the Euclidean metric on $\R^{n}_{i}$, and let $\nabla$ be the corresponding flat connection. Also, define $C^{\infty}_{c}(B_{4a,i})$ to be the space of compactly supported smooth sections on $B_{4a,i}$. We define two maps
			\begin{equation*}
				\begin{split}
					&M_{1}^{\epsilon}: \mathcal{M}^{\epsilon}\to \prod_{i}C^{\infty}_{c}(B_{4a,i}), \eta \mapsto (\rho_{\epsilon}^{-2-\nu}(x_{i})u_{i}\eta),\\
					&M_{2}^{\epsilon}: \prod_{i} C^{\infty}_{c}(B_{4a,i})\to \mathcal{M}^{\epsilon}, (\eta_{i})\mapsto \sum_{i}\rho_{\epsilon}^{2+\nu}(x_{i})v_{i}\eta_{i}.
				\end{split}
			\end{equation*}
		Here, the grade Fr\'echet norms on $\prod_{i}C^{\infty}_{c}(B_{4a,i})$ are defined to be
			\begin{equation*}
				\|(\eta_{i})\|_{k}=\sup_{i}\|\eta_{i}\|_{C^{k}_{c}(B_{4a,i})}.
			\end{equation*}
		It follows directly from the construction that $M_{2}^{\epsilon}\circ M_{1}^{\epsilon}=Id_{\mathcal{M}^{\epsilon}}$. Moreover, with a suitable choice of the finite cover, $\mathcal{M}^{\epsilon}$ and $\prod_{i}C^{\infty}_{c}(B_{4a,i})$ are tamely equivalent. To be precise, we have the following lemma.
			\begin{lemma}\label{S_4_Lemma_Equivalence}
				There is a choice of $\{B_{r_{i}}(x_{i})\}$ such that
					\begin{equation*}
						\begin{split}
							&\|M_{1}^{\epsilon}\eta\|_{k}\leq C_{k}\|\eta\|_{k},\\
							&\|M_{2}^{\epsilon}(\eta_{i})\|_{k}\leq C_{k}\|(\eta_{i})\|_{k}.
						\end{split}
					\end{equation*}
			\end{lemma}
		
		A similar argument also applies to graded Fr\'echet spaces $\mathcal{S}^{\epsilon}, \mathcal{V}^{\epsilon}$ and requires a slight modification to adapt to the case $\mathcal{E}^{\epsilon}$. From the perspective of weighted analysis, the $C^{k}_{\alpha,\tau,\epsilon}$ are double-weighted by $\rho_{\epsilon}$ and a function vanishing on $\Sigma_{\epsilon}$. To handle this subtlety, we define $r_{\epsilon}$ to be a smooth function such that
			\begin{equation*}
				r_{\epsilon}(x)=
				\begin{cases}
					1, & x \in  M_{2r},\\
					dist_{g_{\epsilon}}(x, \Sigma_{0}), & x \in U_{0,r},\\
					\epsilon^{-1}dist_{g_{\epsilon}}(x,\Psi_{q}^{-1}\Sigma_{q}), & x\in U_{q,r}.
				\end{cases}
			\end{equation*}
		We now cover $M$ with geodesic balls of radius $a r_{\epsilon} \rho_{\epsilon}$. Although we are unable to choose a finite subcover in this situation, we can still choose a nice subcover such that only finitely many sets admit a non-empty intersection and the related maps
			\begin{equation*}
				\begin{split}
					&M_{1}^{\epsilon}: \mathcal{E}^{\epsilon}\to \prod_{i}C^{\infty}_{c}(B_{4a,i}), \eta \mapsto (\rho_{\epsilon}^{-\nu}r_{\epsilon}^{-\alpha}(x_{i})u_{i}\eta),\\
					&M_{2}^{\epsilon}: \prod_{i} C^{\infty}_{c}(B_{4a,i})\to \mathcal{E}^{\epsilon}, (\eta_{i})\mapsto \sum_{i}\rho_{\epsilon}^{\nu}r_{\epsilon}^{\alpha}(x_{i})v_{i}\eta_{i},
				\end{split}
			\end{equation*}
		also, have the properties in Lemma \ref{S_4_Lemma_Equivalence}.\\

		With the above discussion, constructing smoothing operators on $X$ amounts to constructing smoothing operators on each  $C^{\infty}_{c}(B_{4a,i})$. Let $\gamma\in C^{\infty}(\R^{\geq 0})$ be a non-increasing smooth function such that
			\begin{equation*}
				\gamma(s)=
				\begin{cases}
					1, &s\leq 1,\\
					0, &s\geq 2.
				\end{cases}
			\end{equation*}
		We define $p_{\theta}(x)=c\theta^{n}\gamma(\theta|x|)$, where $c$ is a normalized constant such that
			\begin{equation*}
				\int_{\R^{n}}p_{\theta}(x)\ud x=1.
			\end{equation*}
		The smoothing operators $S_{\theta}:C^{\infty}_{c}(B_{4a,i})\to C^{\infty}_{c}(\R^{n})$ are obtained by convolution
			\begin{equation*}
				(S_{\theta}\eta)(x):=\int_{\R^{n}}\eta(y)p_{\theta}(x-y)\ud y.
			\end{equation*}
			\begin{remark}
				It should be noted here that if $\eta\in C^{\infty}_{c}(B_{4a,i})$ and $\theta>2/a$, then $S_{\theta}\eta\in C^{\infty}_{c}(B_{4a,i})$.
			\end{remark}
		
		Now it is a calculus exercise to check that $S_{\theta}$ satisfies those three properties in Proposition \ref{S_4_Prop_Smoothing}. Use Lemma \ref{S_4_Lemma_Equivalence}, we define
			\begin{equation}\label{S_4_Eqn_Smoothings}
				S^{\epsilon}_{\theta}:=M^{\epsilon}_{2}\circ S_{\theta}\circ M^{\epsilon}_{1},
			\end{equation}
		and $S^{\epsilon}_{\theta}$ satisfies the properties in Proposition \ref{S_4_Prop_Smoothing}.\\
		
		Once the smoothing operators are constructed, we are able to derive an interpolation formula for the graded Fr\'echet norms.
			\begin{corollary}\label{S_4_Cor_Interpolation}
				For any triple $k_{1}\leq k_{2}\leq k_{3}$, the following estimates hold
				\begin{equation*}
					\|f\|_{k_{2}}^{k_{3}-k_{1}}\leq C_{k_{1},k_{2},k_{3}}\|f\|_{k_{1}}^{k_{3}-k_{2}}\|f\|_{k_{3}}^{k_{2}-k_{1}}.
				\end{equation*}
			\end{corollary}
			\begin{proof}
				For all $\theta\geq 1$
				\begin{equation*}
					\begin{split}
						\|f\|_{k_{2}}&\leq \|S_{\theta}f\|_{k_{2}}+\|(Id-S_{\theta})f\|_{k_{2}}\\
						&\leq C_{k_{3},k_{2}}\theta^{k_{2}-k_{3}}\|f\|_{k_{3}}+C_{k_{1},k_{2}}\theta^{k_{2}-k_{1}}\|f\|_{k_{1}}.
					\end{split}
				\end{equation*}
				If $f\neq 0$, we can choose $\theta$ so that $\theta^{k_{2}-k_{3}}\|f\|_{k_{3}}$ and $\theta^{k_{2}-k_{1}}\|f\|_{k_{1}}$ are equal. This happens if we pick
					\begin{equation*}
						\theta^{k_{3}-k_{1}}=\|f\|_{k_{3}}/\|f\|_{k_{1}}.
					\end{equation*}
				Then, the corollary follows. 
			\end{proof}
			
		An immediate consequence of the interpolation formula is the tame estimates for the maps $\underline{A}$ and $\underline{B}$. Let $K$ be a large positive integer.
		
		\begin{prop}\label{S_4_Prop_AB_map}[Tame estimates]
		 Suppose $\tilde{g}=g_{\epsilon}+\epsilon^{2-\tau}\eta$, with $\|\eta\|_{C^{3K,\alpha}_{\tau-2,\epsilon}}\leq \delta'$ for a small $\delta'>0$. Then, for every $k\geq K$, the maps $\epsilon^{\tau-2}\underline{A}^{\epsilon}(s_{\epsilon}): \mathcal{M}_{0}^{\epsilon}\to C^{\infty}_{\tau-2,\epsilon}(\Sigma_{\epsilon},N\Sigma_{\epsilon}^{-\frac{1}{2}})$ and $\epsilon^{\tau-2}\underline{B}^{\epsilon}(s_{\epsilon}): \mathcal{M}_{0}^{\epsilon}\to C^{\infty}_{\tau-2,\epsilon}(\Sigma_{\epsilon},N\Sigma_{\epsilon}^{-\frac{3}{2}})$ satisfy tame estimates
				\begin{equation*}
					\begin{split}
						&\|\epsilon^{\tau-2}\underline{A}^{\epsilon}(s_{\epsilon}(\tilde{g}))\|_{C^{k+1}_{\tau-2,\epsilon}}\leq C_{k}\big(\epsilon^{(n-2+\tau)\sigma}+\|\eta\|_{C^{k+K}_{\tau-2,\epsilon}}\big),\\
						&\|\epsilon^{\tau-2}\underline{B}^{\epsilon}(s_{\epsilon}(\tilde{g}))\|_{C^{k}_{\tau-2,\epsilon}}\leq C_{k}\big(\epsilon^{(n-2+\tau)\sigma}+\|\eta\|_{C^{k+K}_{\tau-2,\epsilon}}\big).
					\end{split}
				\end{equation*}
		\end{prop}
		\begin{remark}\label{S_4_Rmk_ABmap}
			Suppose $s_{\epsilon}(g_{\epsilon}+\epsilon^{2-\tau}\eta)$ is the unique solution to the Poisson equation \ref{S_3_Eqn_Poisson}, then we have the estimate
				\begin{equation*}
					\|\underline{A}_{\tau}(s_{\epsilon})\|_{C^{k+1,\alpha+\frac{1}{2}}_{0,\epsilon}}, \|\underline{B}_{\tau}(s_{\epsilon})\|_{C^{k,\alpha+\frac{1}{2}}_{0,\epsilon}}\leq C_{k}\big(\|\eta\|_{C^{k+K,\alpha}_{\tau-2,\epsilon}}+\epsilon^{(n-2+\tau)\sigma}\big).
				\end{equation*}

			In the gluing setup, the difference between $g_{\epsilon}$ and $g$ is bounded by
				\begin{equation*}
					\|g-g_{\epsilon}\|_{C^{k,\alpha}_{-2+\tau,\epsilon}}\leq C_{k}\epsilon^{(4-\tau)(1-\sigma)}.
				\end{equation*}
			Since $(4-\tau)(1-\sigma)>(2-\tau)$, it is justified to rescale the difference by $\epsilon^{2-\tau}$. Moreover, we emphasize that the dimension $n\geq 3$ is crucial in our setup, as it ensures that the error introduced by the approximate solution remains small.
		\end{remark}
		\begin{proof}
			Since $\underline{A}^{\epsilon}, \underline{B}^{\epsilon}$ is a bounded linear map and
				\begin{equation*}
					\|\cdot\|_{C^{k}_{\tau,\epsilon}(\Sigma_{\epsilon}, N\Sigma_{\epsilon}^{-p})}\leq \|\cdot\|_{C^{k,\alpha+\frac{1}{2}}_{\tau,\epsilon}(\Sigma_{\epsilon}, N\Sigma_{\epsilon}^{-p})},
				\end{equation*}
			it suffices to prove the tame estimates
				\begin{equation}\label{S_3_Eqn_Tame}
					\|\epsilon^{\tau-2}s_{\epsilon}(\tilde{g})\|_{\mathcal{C}^{k+2,\alpha}_{\tau,\epsilon}}\leq C_{k}(\epsilon^{(n-2+\tau)\sigma}+\|\eta\|_{C^{k+K,\alpha}_{\tau-2,\epsilon}}), \forall k\geq K.
				\end{equation}
				Obviously, we have
				\begin{equation*}
					\|\epsilon^{2-\tau}\eta\|_{C^{k+K,\alpha}_{0,\epsilon}}\leq \|\eta\|_{C^{k+K.\alpha}_{\tau-2,\epsilon}},
				\end{equation*}
				then when $k\leq 2K$, the estimates follow from Lemma \ref{S_3_Lemma_Error_gluing}.\\

			On the other hand, when $k\geq 2K$, we prove by induction. Suppose the statement is true for every $k'\leq k-1$. We estimate the $\mathcal{D}^{k,\alpha}_{\tau,\epsilon}$ norms for
				\begin{equation*}
					\widetilde{\Delta}_{g_{\epsilon}}\epsilon^{\tau-2} s_{\epsilon}(\tilde{g}),
				\end{equation*}
			then Proposition \ref{S_3_Prop_linear_est} yields the estimate \ref{S_3_Eqn_Tame}. Notice that $\widetilde{\Delta}_{g_{\epsilon}}=\widetilde{\Delta}_{\tilde{g}}+\mathcal{L}$ the estimates can be broken into two parts. The first part follows directly from Poisson's equation \ref{S_3_Eqn_Poisson} and Lemma \ref{S_3_Lemma_Error_gluing}. On the other hand, the second part amounts to estimating $\|\epsilon^{\tau-2}s_{\epsilon}\|_{\mathcal{D}^{k+2,\alpha}_{\tau,\epsilon}}$ and the coefficients of $\mathcal{L}$.\\

				We use the same trick in Lemma \ref{S_3_Lemma_Error_gluing} to estimate the $\mathcal{D}^{k+2,\alpha}_{\tau,\epsilon}$ norms. First, we consider the region $M_{d}$. Let $x_{i}\in M_{d}$ and let $X_{1},\cdots, X_{k}$ be vector fields such that they are orthonormal with respect to $\rho_{\epsilon}^{-2}(x_{i})g_{\epsilon}$ on $B_{c\rho(x_{i})}(x_{i})$ and supported on $B_{2c\rho(x_{i})}(x_{i})$. Also let $\nabla_{j}=\nabla_{X_{j}}$ be the covariant derivatives with respect to the metric $g_{\epsilon}$. The commutator $T_{j}=[\widetilde{\Delta}_{\tilde{g}},\nabla_{j}]$ defines a second order differential operator which is supported on $B_{2c\rho(x_{i})}(x_{i})$. We have
					\begin{equation*}
						[\widetilde{\Delta}_{g_{i}+\eta_{i}},\nabla_{1}\cdots \nabla_{k}]=\nabla_{1}\cdots\nabla_{k-1}T_{k}+\cdots+T_{1}\cdots\nabla_{k}.
					\end{equation*}
				For simplicity, denote $u$ as $\epsilon^{\tau-2}s_{\epsilon}$ and $v$ as $\epsilon^{2-\tau}\eta$. Direct computation implies
					\begin{equation*}
						\begin{split}
							&\|\widetilde{\Delta}_{\tilde{g}}\nabla_{1}\cdots \nabla_{k}u\|_{\mathcal{D}^{0,\alpha}_{\tau-2,\epsilon}}\leq \|\nabla_{1}\cdots \nabla_{k}\widetilde{\Delta}_{\tilde{g}}u\|_{\mathcal{D}^{0,\alpha}_{\tau-2,\epsilon}}\\
							&+C_{k}\big(\|v\|_{C^{2,\alpha}_{0,\epsilon}}\|u\|_{\mathcal{D}^{k+1,\alpha}_{\tau-2,\epsilon}}+\cdots+ \|v\|_{C^{k+1,\alpha}_{0,\epsilon}}\|u\|_{\mathcal{D}^{2,\alpha}_{\tau-2,\epsilon}}\big).
						\end{split}
					\end{equation*}
				On $U_{0,2d}$ and on $U_{q,2d}$ with the weighted metric $\rho_{\epsilon}^{-2}g_{\epsilon}$, we instead take $X_{k}\in \mathcal{T}_{1}$. Then, we have a similar bound but with a loss of regularity on $\eta$
					\begin{equation*}
						\begin{split}
							&\|\widetilde{\Delta}_{\tilde{g}}\nabla_{1}\cdots \nabla_{k}u\|_{\mathcal{D}^{0,\alpha}_{\tau-2,\epsilon}}\leq \|\nabla_{1}\cdots \nabla_{k}\widetilde{\Delta}_{\tilde{g}}u\|_{\mathcal{D}^{0,\alpha}_{\tau-2,\epsilon}}\\
							&+C_{k}\big(\|v\|_{C^{2+m',\alpha}_{0,\epsilon}}\|u\|_{\mathcal{D}^{k+1,\alpha}_{\tau,\epsilon}}+\cdots+ \|v\|_{C^{k+1+m',\alpha}_{0,\epsilon}}\|u\|_{\mathcal{D}^{2,\alpha}_{\tau,\epsilon}}\big).
						\end{split}
					\end{equation*}
				Here, $K'$ is a constant representing the loss of regularity in the coefficients of the commuters $T_{j}$, which can equal $3$. Choose $K\geq K'+3$ and use Proposition \ref{S_3_Prop_linear_est}. Using the induction hypothesis, we conclude
					\begin{equation*}
						\begin{split}
							\|u\|_{\mathcal{D}^{k+2,\alpha}_{\tau,\epsilon}}&\leq C_{k}\big(\epsilon^{(n-2+\tau)\sigma}+\|\eta\|_{C^{k+K-1,\alpha}_{\tau-2,\epsilon}}+\sum_{p=0}^{k-2}\|\eta\|_{C^{K+p,\alpha}_{\tau-2,\epsilon}}\|\eta\|_{C^{k+K-2-p,\alpha}_{\tau-2,\epsilon}}\big)\\
							&\leq C_{k}\big(\epsilon^{(n-2+\tau)\sigma}+\|\eta\|_{C^{k+K,\alpha}_{\tau-2,\epsilon}}+\sum_{p=0}^{k-2}\|\eta\|_{C^{K+1+p}_{\tau-2,\epsilon}}\|\eta\|_{C^{k+K-1-p}_{\tau-2,\epsilon}}\big).
						\end{split}
					\end{equation*}
				Interpolation formula for $K\leq j\leq k+K$
					\begin{equation*}
						\|\eta\|_{C^{j}_{\tau-2,\epsilon}}\leq C_{K,j,k+K}\|\eta\|_{C^{K}_{\tau-2,\epsilon}}^{(k+K-j)/k}\|\eta\|_{C^{k+K}_{\tau-2,\epsilon}}^{(j-K)/k},
					\end{equation*}
				and the bound for $\|\eta\|_{C^{2K,\alpha}_{\tau-2,\epsilon}}$ yields the estimate
					\begin{equation*}
						\|u\|_{\mathcal{D}^{k+2,\alpha}_{\tau,\epsilon}}\leq C_{k}\big(\epsilon^{(n-2+\tau)\sigma}+\|\eta\|_{C^{k+K,\alpha}_{\tau-2,\epsilon}}\big).
					\end{equation*}
				
				It remains to bound $\|\mathcal{L}u\|_{\mathcal{D}^{k,\alpha}_{\tau-2,\epsilon}}$. Straightforward computations and similar arguments show that
					\begin{equation*}
						\|\mathcal{L}u\|_{\mathcal{D}^{k,\alpha}_{\tau-2,\epsilon}}\leq C_{k}(\epsilon^{(n-2+\tau)\sigma}+\|\eta\|_{C^{k+K,\alpha}_{\tau-2,\epsilon}}),				
					\end{equation*}
				which completes the proof of the tame estimate.
		\end{proof}

	\subsection{Setting up the gluing problem}%%%%%%%%
		Recall that in a small neighborhood of $\mathcal{U}$ of $(g_{\epsilon},\Sigma_{\epsilon})$ in  $\mathcal{M}^{\epsilon}\times \mathcal{S}^{\epsilon}$, we find a smooth map $\Phi: \mathcal{U}\to \mathcal{V}^{\epsilon}$ such that
			\begin{equation*}
				\Phi(\tilde{g},\widetilde{\Sigma})^{*}(\widetilde{\Sigma})=\Sigma_{\epsilon},
			\end{equation*}
			and that
			\begin{equation*}
				\Phi(\tilde{g},\widetilde{\Sigma})^{*}(\tilde{g})\in \mathcal{M}_{0}^{\epsilon}.
			\end{equation*}
		For $-1\ll \tau<0$, we endow $\mathcal{M}^{\epsilon}, \mathcal{M}^{\epsilon}_{0}$ with graded Fr\'echet norms $\|\cdot\|_{C^{k}_{\tau-2,\epsilon}}$, the space of vector fields $\mathcal{V}^{\epsilon}$ with $\|\cdot\|_{C^{k}_{\tau-1,\epsilon}}$ norms and $\mathcal{S}^{\epsilon}$, the space of deformation of $\Sigma_{\epsilon}$ with $\|\cdot\|_{C^{k}_{\tau-2,\epsilon}(\Sigma_{\epsilon},N\Sigma_{\epsilon})}$ norms. We now have the global version of Lemma \ref{S_2_Lemma_Deformation} together with tame estimates. Let $m\geq K$ be an arbitrary integer.
			\begin{lemma}\label{S_4_Lemma_Deformation}
				Suppose that $\tilde{g}=g_{\epsilon}+\epsilon^{2-\tau}\eta$ and that the deformation of $\Sigma_{\epsilon}$ is given by $\epsilon^{2-\tau}v \in C^{\infty}_{\tau-2,\epsilon}(\Sigma_{\epsilon}, N\Sigma_{\epsilon})$. If $\|\eta\|_{C^{3m}_{\tau-2,\epsilon}}, \|v\|_{C^{3m}_{\tau-2,\epsilon}}\ll 1$, then the diffeomorphism $\Phi_{\epsilon}(\widetilde{\Sigma},\tilde{g})$ satisfies
					\begin{equation*}
						\|\Phi_{\epsilon}-Id\|_{C^{k}_{\tau-1,\epsilon}}\leq C_{k}\epsilon^{2-\tau}\big(\|\eta\|_{C^{k}_{\tau-2,\epsilon}}+\|v\|_{C^{k}_{\tau-2,\epsilon}}\big).
					\end{equation*}
				Moreover, the map
					\begin{equation*}
						\mathcal{G}^{\epsilon}:(v,\tilde{g})\mapsto \epsilon^{\tau-2}\big(\Phi_{\epsilon}^{*}\tilde{g}-g_{\epsilon}\big)
					\end{equation*}
				is a smooth tame map with the following tame estimate
					\begin{equation*}
							\|\mathcal{G}^{\epsilon}(v,\tilde{g})\|_{C^{k}_{\tau-2,\epsilon}}\leq C_{k}\big(\|\eta\|_{C^{k+1}_{\tau-2,\epsilon}}+\|v\|_{C^{k+1}_{\tau-2,\epsilon}}\big)
					\end{equation*}
				and denote $X=(U,V)$ as the vector for $k\leq 2m$ 
					\begin{equation*}
						\begin{split}
							&\|D\mathcal{G}^{\epsilon}(X)\|_{C^{k}_{\tau-2,\epsilon}}\leq C_{k}\|X\|_{C^{k+1}_{\tau-2,\epsilon}},\\
							&\|D^{2}\mathcal{G}^{\epsilon}(X_{1},X_{2})\|_{C^{k}_{\tau-2,\epsilon}}\leq C_{k}\|X_{1}\|_{C^{k+1}_{\tau-2,\epsilon}}\|X_{2}\|_{C^{k+1}_{\tau-2,\epsilon}}.
						\end{split}
					\end{equation*}
			\end{lemma}
			\begin{proof}
				Applying Lemma \ref{S_2_Lemma_Deformation} on each model space with respect to the rescaled metric $\rho_{\epsilon}(x_{i})^{-2}g_{\epsilon}$, we construct the desired diffeomorphism $\Phi_{\epsilon}$. In particular, $\Phi_{\epsilon}=Id$ on $M_{2d}$, it suffices to work within $U_{0,2d}$ with metric $g_{\epsilon}$ and $U_{q,2d}$ with metric $\epsilon^{-2}g_{\epsilon}$. In estimating the $C^{k}_{\tau-1,\epsilon}$ norms for $\Phi_{\epsilon}-Id$, we will also have
					\begin{equation*}
						\|\Phi_{\epsilon}-Id\|_{C^{k}_{\tau-1,\epsilon}}\leq \epsilon^{2-\tau}C_{k}\big(\|v\|_{C^{k}_{\tau-1,\epsilon}}+\sum_{p=0}^{k}\|v\|_{C^{k-p}_{\tau-1,\epsilon}}\|\eta\|_{C^{p}_{\tau-2,\epsilon}}\big).
					\end{equation*}
				Interpolation formula together with Young's inequality yields the bound. Similar computation also applies to $\mathcal{G}^{\epsilon}$, and the tame estimates hold. On the other hand, the estimate for the differential and the Hessian can also be obtained through direct computation.
			\end{proof}
			
		We define a map
			\begin{equation*}
				\bs{A}^{\epsilon}: \mathcal{M}^{\epsilon}_{0}\to C^{\infty}_{\tau-2,\epsilon}(\Sigma_{\epsilon}, N\Sigma_{\epsilon}^{-\frac{1}{2}}), \eta \mapsto \epsilon^{\tau-2}\underline{A}^{\epsilon}(s_{\epsilon}(g_{\epsilon}+\epsilon^{2-\tau}\eta)),
			\end{equation*}
		and a map
			\begin{equation*}
				\bs{F}^{\epsilon}: C^{\infty}_{\tau-2,\epsilon}(\Sigma_{\epsilon},N\Sigma_{\epsilon})\to C^{\infty}_{\tau-2,\epsilon}(\Sigma_{\epsilon},N\Sigma_{\epsilon}^{-\frac{1}{2}}), v\mapsto \bs{A}^{\epsilon}(\mathcal{G}^{\epsilon}(v,g)).
			\end{equation*}
		We will apply the Nash-Moser-Hamilton-Zehnder implicit function in Theorem \ref{S_App_Thm_HM} 
		to solve the following equation
			\begin{equation*}
				\bs{F}^{\epsilon}(v)=0.
			\end{equation*}
		When $\epsilon$ is sufficiently small, we obtain a non-degenerate $\Z_{2}$-harmonic $1$-form with respect to the original metric $g$
			\begin{equation*}
				\alpha_{\epsilon}=(\Phi_{\epsilon}^{-1})^{*}\bigg(\widetilde{\alpha}_{\epsilon}+\ud^{\star}_{\Phi^{*}g}\big(W(\Phi^{*}g)s_{\epsilon}\big)\bigg),
			\end{equation*}
			whose branching set is $\Phi_{\epsilon}^{-1}(\Sigma_{\epsilon})$.\\
					
		For simplicity, in the following sections, we will use $\|\cdot\|_{k}$ to denote the weighted Fr\'echet norms when there is no ambiguity. In order to apply the Hamilton-Nash-Moser-Zehnder implicit function theorem, we need to consider the following properties of the $\bs{F}^{\epsilon}$ map.
		
		\begin{theorem}\label{S_4_Thm_tamemap}
			Let $m\geq K$ be a large constant. In the region $\|v\|_{3m}\leq \delta$, let $\epsilon\ll 1$, then the map $\bs{F}^{\epsilon}$ satisfies
				\begin{equation}\label{S_4_Eqn_Tame_1}
					\begin{split}
						&\|\bs{F}^{\epsilon}(v)\|_{k}\leq C_{k}(1+\|v\|_{k+m}),\;\forall k\geq 0\\
						& \|D\bs{F}^{\epsilon}(v)(v_{1})\|_{k}\leq C_{k}\|v\|_{k+m/2}, \;\forall 0\leq  k\leq 2m,\\
						&\|D^{2}\bs{F}^{\epsilon}(v)(v_{1},v_{2})\|_{k}\leq C_{k}\|v_{1}\|_{k+m/2}\|v_{2}\|_{k+m/2}, \forall 0\leq k\leq 2m.
					\end{split}
				\end{equation}
			The differential $D\bs{F}^{\epsilon}$ admits an approximate inverse
				\begin{equation*}
					V(v,\bs{F}^{\epsilon}): C^{\infty}_{\tau-2,\epsilon}(\Sigma_{\epsilon},N\Sigma_{\epsilon}^{-\frac{1}{2}})\to C^{\infty}_{\tau-2,\epsilon}(\Sigma_{\epsilon},N\Sigma_{\epsilon}),
				\end{equation*}
			such that
				\begin{equation}\label{S_4_Eqn_Tame_2}
					\begin{split}
						&V\circ D \bs{F}^{\epsilon}=Id+Q_{1}(v, \bs{F}^{\epsilon}),\\
						& D \bs{F}^{\epsilon}\circ V=Id+Q_{2}(v, \bs{F}^{\epsilon}),\\
						&\|Q_{1}(v,\bs{F}^{\epsilon})(v_{1})\|_{2m}\leq C \|\bs{F}^{\epsilon}\|_{3m}\|v_{1}\|_{3m},\\
						&\|Q_{2}(v,\bs{F}^{\epsilon})(a_{1})\|_{2m}\leq C \|\bs{F}^{\epsilon}\|_{3m}\|a_{1}\|_{3m},\\
						&\|V(v,\bs{F}^{\epsilon})(a_{1})\|_{k}\leq C \|a_{1}\|_{k+m/2},\forall 0 \leq k\leq m,\\
						&\|V(v,\bs{F}^{\epsilon})(a_{1})\|_{k}\leq C_{k}\big( \|a_{1}\|_{k+m}+\|v\|_{k+m}\|a_{1}\|_{2m}\big),\forall  k\geq m.
					\end{split}
				\end{equation}
			Here, $C, C_{k}$ are constants independent of $\epsilon$ and $m$.
		\end{theorem}
		
		The map $\bs{F}^{\epsilon}=\bs{A}^{\epsilon}\circ \mathcal{G}^{\epsilon}$ factors through $\mathcal{M}^{\epsilon}_{0}$. We have already proved the tame estimates for the $\mathcal{G}^{\epsilon}$ in Lemma \ref{S_4_Lemma_Deformation} and established tame estimates for $\bs{A}^{\epsilon}$ in Lemma \ref{S_4_Lemma_Tame}. Therefore, to prove the Equation \ref{S_4_Eqn_Tame_1}, it suffices to prove the estimates for the $D\bs{A}^{\epsilon}$ and $D^{2}\bs{A}^{\epsilon}$.
		\begin{lemma}\label{S_4_Lemma_Tame}
			In the region $\|\eta\|_{C^{2m}_{\tau-2,\epsilon}}\leq \delta$, the map $\bs{A}^{\epsilon}$ satisfies
				\begin{equation*}
					\begin{split}
						& \|D\bs{A}^{\epsilon}(v)(v_{1})\|_{k}\leq C_{k}\|v\|_{k+3}, \;\forall 0\leq  k\leq 2m,\\
						&\|D^{2}\bs{A}^{\epsilon}(v)(v_{1},v_{2})\|_{k}\leq C_{k}\|v_{1}\|_{k+3}\|v_{2}\|_{k+3}, \forall 0\leq k\leq 2m.
					\end{split}
				\end{equation*}
		\end{lemma}
		\begin{proof}
			Let $g_{t}=\tilde{g}+t\epsilon^{2-\tau}\eta_{1}\in \mathcal{M}^{\epsilon}_{0}$ and let $\p_{t}\widetilde{\Delta}_{t}$ be the derivatives of $\widetilde{\Delta}_{g_{t}}$ with respect to $t$ at $t=0$.  We also let $\p_{t}s_{\epsilon}$ be the derivatives of $s_{\epsilon}(g_{t})$. For simplicity, we write $w_{\epsilon}$ to denote a multivalued section in the affine $\R$-bundle on $M\setminus \Sigma_{\epsilon}$, such that
				\begin{equation*}
					\ud w_{\epsilon}=-\tilde{\alpha}_{\epsilon}.
				\end{equation*}
			Then $\p_{t} s_{\epsilon}$ satisfies
				\begin{equation*}
					\widetilde{\Delta}_{g_{t}}\p_{t}s_{\epsilon}+\big(\p_{t}\widetilde{\Delta}_{g_{t}}\big)(s_{\epsilon}+w_{\epsilon})=0.
				\end{equation*}
			Once we have a $\mathcal{D}^{k,\alpha}_{\tau-2,\epsilon}$ bound on $(\p_{t}\widetilde{\Delta}_{g_{t}})(s_{\epsilon}+w_{\epsilon})$, we will then have a $C^{k+1}_{\tau-2,\epsilon}$ bound on $\bs{A}^{\epsilon}$. On the region $U_{0,2r}$ and on $U_{q,2r}$, the Laplacian $\widetilde{\Delta}_{g_{t}}=\widetilde{\Delta}_{\tilde{g}}+\mathcal{L}_{t\epsilon^{2-\tau}\eta_{1}}$, then the derivatives
				\begin{equation*}
					\p_{t}\widetilde{\Delta}_{t}=\p_{t}\mathcal{L}_{t\epsilon^{2-\tau}\eta_{1}}\in \mathcal{T}_{2}.
				\end{equation*}
			With respect to the conformal metric $\rho_{\epsilon}^{-2}g_{\epsilon}$, the coefficients $\mathcal{F}$ of $\p_{t}\mathcal{L}$ satisfy
				\begin{equation*}
						\|\mathcal{F}\|_{C^{k,\alpha}(U_{0,2r})}\leq \|\epsilon^{2-\tau}\eta_{1}\|_{C^{k+m,\alpha}}.
				\end{equation*}
			Similar to the proof of Proposition \ref{S_3_Prop_linear_est}, we conclude that
				\begin{equation*}
					\|\epsilon^{\tau-2} s_{\epsilon}\|_{\mathcal{C}^{k+2,\alpha}_{\tau,\epsilon}}\leq C_{k}\|\eta_{1}\|_{C^{k+m}_{\tau-2,\epsilon}}, \forall 0\leq k\leq 2m. 
				\end{equation*}
				
			We now derive uniform estimates for the second order derivatives. Let $g_{t_{1},t_{2}}=\tilde{g}+t_{1}\epsilon^{2-\tau}\eta_{1}+t_{2}\epsilon^{2-\tau}\eta_{2}$, then $\p_{t_{1}}\p_{t_{1}}s_{\epsilon}$ satisfies
				\begin{equation*}
					\widetilde{\Delta}_{\tilde{g}}(\p_{t_{1}}\p_{t_{1}}s_{\epsilon})+\big(\p_{t_{1}}\widetilde{\Delta}\big)(\p_{t_{2}}s_{\epsilon})+\big(\p_{t_{2}}\widetilde{\Delta}\big)(\p_{t_{1}}s_{\epsilon})+\big(\p_{t_{1}}\p_{t_{2}}\widetilde{\Delta}\big)(s_{\epsilon}+w_{\epsilon})=0.
				\end{equation*}
				Similar arguments will lead to the estimates.
		\end{proof}
		
		A key observation by Donaldson is that the derivatives of $\bs{F}^{\epsilon}$ are invertible. To adapt this to our setting, we will need a weighted version of Donaldson's derivative formula.
			\begin{prop}\label{S_4_Prop_L2Sol}
				For any section $\bs{\sigma}=(\sigma_{0},\cdots,\sigma_{q}, \cdots)$ in $C^{\infty}_{\tau,\epsilon}(\Sigma_{\epsilon}, N\Sigma_{\epsilon}^{\frac{1}{2}})$, there exists a linear tame map
					\begin{equation*}
						P:C^{\infty}_{\tau,\epsilon}(\Sigma_{\epsilon},N\Sigma_{\epsilon}^{\frac{1}{2}})\to C^{\infty}_{\tau-2,\epsilon}(\Sigma_{\epsilon},N\Sigma_{\epsilon}^{-\frac{1}{2}}), \bs{\sigma}\mapsto (P_{0},\cdots,P_{q},\cdots),
					\end{equation*}
				and a unique section $\widetilde{Q}_{\epsilon}(\bs{\sigma})$ of $L^{\CC}_{\epsilon}$ over $M\setminus \Sigma_{\epsilon}$ such that $\widetilde{\Delta}_{\tilde{g}}\widetilde{Q}_{\epsilon}=0$ and $\widetilde{Q}_{\epsilon}$ admit the following asymptotic behaviors with respect to the conformal metric $\rho_{\epsilon}^{-2}g_{\epsilon}$
					\begin{equation*}
						\begin{split}
							&\widetilde{Q}_{\epsilon}=\sigma_{0}\zeta^{-\frac{1}{2}}+P_{0}\zeta^{\frac{1}{2}}+O(|\zeta|^{\frac{3}{2}}), \text{ on } U_{0,2r},\\
							&\widetilde{Q}_{\epsilon}=\sigma_{q}\zeta^{-\frac{1}{2}}+\epsilon^{2}P_{q}\zeta^{\frac{1}{2}}+O(|\zeta|^{\frac{3}{2}}), \text{ on } V_{q,2r}.
						\end{split}
					\end{equation*}
				Moreover, the map $P$ satisfies
					\begin{equation*}
						\|P(\bs{\sigma})\|_{k}\leq C_{k}(\|\bs{\sigma}\|_{k+4}+\|\eta\|_{k+m}\|\bs{\sigma}\|_{m}).
					\end{equation*}
			\end{prop}
			\begin{proof}
				We carry out the computation on $U_{0,2r}$ using $\tilde{g}$ and on the rescaled coordinate $U_{q,2r}$ using $\epsilon^{-2}\tilde{g}$. Let $r=dist_{\rho_{\epsilon}^{-2}g_{\epsilon}}(-,\Sigma_{\epsilon})$ be the distance function to the branching set in each model. Applying Proposition 10 in \cite{Donaldson17} to sections $\sigma_{0}\zeta^{-\frac{1}{2}}$ and $\sigma_{q}\zeta^{-\frac{1}{2}}$, we obtain
					\begin{equation*}
						\begin{split}
							&\widetilde{\Delta}_{\tilde{g}}\sigma_{0}\zeta^{-\frac{1}{2}}=\beta_{0}\zeta^{-\frac{1}{2}}+w_{0},\\
							&\widetilde{\Delta}_{\epsilon^{-2}\tilde{g}}\sigma_{q}\zeta^{-\frac{1}{2}}=\beta_{q}\zeta^{-\frac{1}{2}}+w_{q}.
						\end{split}
					\end{equation*}
				Here, $\beta_{j}=\big(\nabla^{*}\nabla +\kappa\big)\sigma_{j}, j=0,q$ and $w_{j}\in \mathcal{D}^{\infty,\alpha}$ in each geometry. Moreover, using the smoothing operator for $\mathcal{E}^{\epsilon}$ and $\mathcal{M}^{\epsilon}$, we have the following tame estimate
					\begin{equation*}
						\|w_{j}\|_{\mathcal{D}^{k,\alpha}}\leq C_{k}(\|\sigma_{j}\|_{C^{k+2,\alpha}}+\|\eta\|_{k+K}\|\sigma\|_{C^{K,\alpha}}), j=0,q.
					\end{equation*}
				On the other hand, we also have
					\begin{equation*}
						\widetilde{\Delta} (\frac{1}{4}r^{2}\beta_{j}\zeta^{-\frac{1}{2}})=\beta_{j}\zeta^{-\frac{1}{2}}+w'_{j},
					\end{equation*}
				where
					\begin{equation*}
						\|w'_{j}\|_{\mathcal{D}^{k,\alpha}}\leq C_{k}(\|\beta_{j}\|_{C^{k+2,\alpha}}+\|\eta\|_{k+K}\|\sigma\|_{C^{K,\alpha}}), j=0,q.
					\end{equation*}
				Now we can solve the equation
					\begin{equation*}
						\widetilde{\Delta}_{\tilde{g}}u_{\epsilon}=\sum_{j}	w_{j}-w'_{j}
					\end{equation*}
				with $u_{\epsilon}\in \mathcal{C}^{\infty,\alpha}_{\tau,\epsilon}$ satisfying the following estimate
					\begin{equation*}
						\|u_{\epsilon}\|_{\mathcal{C}^{k+2,\alpha}_{\tau,\epsilon}}\leq C_{k}\|\bs{\sigma}\|_{C^{k+4,\alpha}_{\tau,\epsilon}}.
					\end{equation*}
				Computation shows
					\begin{equation*}
						\widetilde{\Delta}_{\tilde{g}}\big(-u_{\epsilon}+\sum_{j}(\sigma_{j}\zeta^{-\frac{1}{2}}-\frac{1}{4}r^{2}\beta_{j}\zeta^{-\frac{1}{2}})\big)=0,
					\end{equation*}
				and we have the unique solution
					\begin{equation*}
						\widetilde{Q}_{\epsilon}=-u_{\epsilon}+\sum_{j}(\sigma_{j}\zeta^{-\frac{1}{2}}-\frac{1}{4}r^{2}\beta_{j}\zeta^{-\frac{1}{2}}).
					\end{equation*}
				We define the operator $P$ to be
					\begin{equation*}
						P:C^{\infty}_{\tau,\epsilon}(\Sigma_{\epsilon},N\Sigma_{\epsilon}^{\frac{1}{2}})\to C^{\infty}_{\tau-2,\epsilon}(\Sigma_{\epsilon},N\Sigma_{\epsilon}^{-\frac{1}{2}}), \bs{\sigma}\mapsto -\underline{A}^{\epsilon}(u_{\epsilon}),
					\end{equation*}
				and this proves the proposition.
			\end{proof}

		Suppose $\lambda_{t}$ is a family of diffeomorphisms with derivative at $t=0$ equal to $\tilde{v}$ such that $\tilde{v}|_{\Sigma_{\epsilon}}=\epsilon^{2-\tau}\bs{v}$, where 
			\begin{equation*}
				\bs{v}=(v_{0},v_{1},\cdots, v_{q},\cdots), \bs{v}\in C^{\infty}_{\tau-2,\epsilon}(\Sigma_{\epsilon},N\Sigma_{\epsilon}).
			\end{equation*}
		We can suppose that $\lambda_{t}$ preserves the normal structure induced by $\tilde{g}$. This condition is equivalent to the covariant derivatives of $\tilde{v}$ in the normal direction vanishing along $\Sigma_{\epsilon}$
			\begin{equation*}
				\nabla_{N}\tilde{v}=0 \text{ on } \Sigma_{\epsilon}.
			\end{equation*}
		
		For each $t $, we have the solution to the Poisson equation with a singular set $\Sigma_{\epsilon,t}=\lambda_{t}(\Sigma_{\epsilon})$ and metric $g_{t}$. Pulling back the solution, we obtain a solution $\Phi_{t}=W^{-1}(g_{t})s_{\epsilon}(g_{t})$ to the Poisson equation
			\begin{equation}\label{S_4_Eqn_Poisson_t}
				-\Delta_{g_{t}}\Phi_{t}+\ud^{*}_{g_{t}}\tilde{\alpha}_{\epsilon}=0.
			\end{equation}

		To simplify the problem, we again treat $\tilde{\alpha}_{\epsilon}$ as an exterior derivative of a section $w_{\epsilon}$ on a certain affine $\R$-bundle. In particular, $w_{\epsilon}$ becomes a $\Z_{2}$-function near each connected component of $\Sigma_{\epsilon}$. It follows from the construction that $w_{\epsilon}$ admits the asymptotic expansion on $U_{0,2d}$
			\begin{equation*}
				w_{\epsilon}= B(\alpha)\zeta^{\frac{3}{2}}+O(|\zeta|^{\frac{5}{2}}),
			\end{equation*}
		with respect to the metric $g_{\epsilon}$ and on each connected component of $V_{q,2d}$
			\begin{equation*}
				w_{\epsilon}= \epsilon^{2}\big(B(\alpha_{q})\zeta^{\frac{3}{2}}+O(|\zeta|^{\frac{5}{2}})\big),
			\end{equation*}
		with respect to the metric $\epsilon^{-2}g_{\epsilon}$. Again, let $\p_{t}\Delta$ be the derivative of $\Delta_{g_{t}}$ at $t=0$, and let $\p_{t}\Phi$ be the derivatives of $\Phi_{t}$. Differentiating the Poisson equation \ref{S_4_Eqn_Poisson_t}, we have
			\begin{equation*}
				(\p_{t}\Delta)(\Phi_{0}+w_{\epsilon})=\nabla_{\tilde{v}}\Delta(\Phi_{0}+w_{\epsilon})-\Delta\big(\nabla_{\tilde{v}}(\Phi_{0}+w_{\epsilon})\big),
			\end{equation*}
		on $U_{0,2d}\setminus \Sigma_{0}$ and on $U_{q,2d}\setminus \Psi_{q}^{-1}\Sigma_{q}$. On the other hand, we have $(\p_{t}\Delta)(\Phi_{0}+w_{\epsilon})+\Delta_{\tilde{g}}(\p_{t}\Phi)$=0 and henceforth	
			\begin{equation*}
				\Delta_{\tilde{g}}\big(\p_{t}\Phi-\nabla_{\tilde{v}}(\Phi_{0}+w_{\epsilon})\big)=0.
			\end{equation*}
		We know that $\Phi_{t}$ defines a smooth path in $\mathcal{D}^{\infty,\alpha}$, therefore $\p_{t}\Phi$ also lies in $\mathcal{D}^{\infty,\alpha}$. We also know that $\Phi_{0}+w_{\epsilon}$ admits an asymptotic expansion
			\begin{equation*}
				\Phi_{0}+w_{\epsilon}\sim\RRe{\underline{A}_{0}(s_{\epsilon})\zeta^{\frac{1}{2}}+(\underline{B}_{0}(s_{\epsilon})+B(\alpha))\zeta^{\frac{3}{2}}}-\frac{1}{2}\RRe{\bar{\mu}\zeta}\RRe{\underline{A}_{0}(s_{\epsilon})\zeta^{\frac{1}{2}}},
			\end{equation*}
		up to a term bounded by $O(|\zeta|^{\frac{5}{2}})$. And similarly on $U_{q,2d}$,
			\begin{equation*}
				\Phi_{0}+w_{\epsilon}\sim \RRe{\underline{A}_{q}(s_{\epsilon})\zeta^{\frac{1}{2}}+(\underline{B}_{q}(s_{\epsilon})+\epsilon^{2}B(\alpha_{q}))\zeta^{\frac{3}{2}}}-\frac{1}{2}\RRe{\bar{\mu}\zeta}\RRe{\underline{A}_{q}(s_{\epsilon})\zeta^{\frac{1}{2}}}.
			\end{equation*}

		Following the computation in Donaldson's paper, one finds that the condition $\nabla_{N}\tilde{v}=0$ implies
			\begin{equation*}
				\begin{split}
					&\nabla_{\tilde{v}}\big(A\zeta^{\frac{1}{2}}\big)=\frac{1}{2}(A\tilde{v})\zeta^{-\frac{1}{2}}+O(|\zeta|^{\frac{3}{2}})\\
					&\nabla_{\tilde{v}}\big(B\zeta^{\frac{3}{2}}\big)=\frac{3}{2}(B\tilde{v})\zeta^{\frac{1}{2}}+O(|\zeta|^{\frac{5}{2}})\\
					&\nabla_{\tilde{v}}(\RRe{\bar{\mu}}\zeta)=\IP{\tilde{v}}{\mu},
				\end{split}
			\end{equation*}
		on $U_{0,2d}\setminus\Sigma_{0}$ and $U_{0,2d}\setminus\Psi_{q}^{-1}\Sigma_{q}$, with respect to the metric $\rho_{\epsilon}^{-2}g_{\epsilon}$. Therefore, Proposition \ref{S_4_Prop_L2Sol} implies that $\p_{t}\Phi-\nabla_{\tilde{v}}(\Phi_{0}+w_{\epsilon})$ is the unique solution that asymptotes to
			\begin{equation*}
				\begin{split}
					&-\frac{\epsilon^{2-\tau}}{2}\underline{A}_{0}(s_{\epsilon})v_{0}\zeta^{-\frac{1}{2}}-\frac{1}{2}P_{0}(\underline{A}(s_{\epsilon})\cdot \epsilon^{2-\tau}\bs{v})\zeta^{\frac{1}{2}}, \text{ on }U_{0,2r},\\
					&-\frac{\epsilon^{2-\tau}}{2}\underline{A}_{q}(s_{\epsilon})v_{q}\zeta^{-\frac{1}{2}}-\frac{\epsilon^{2}}{2}P_{q}(\underline{A}(s_{\epsilon})\cdot \epsilon^{2-\tau}\bs{v})\zeta^{\frac{1}{2}}, \text{ on }U_{q,2r}.
				\end{split}
			\end{equation*}
		Here, $\underline{A}(s_{\epsilon})\cdot \bs{v}$ denotes the multiple $(\underline{A}_{0}(s_{\epsilon})v_{0},\cdots, \underline{A}_{q}(s_{\epsilon})v_{q}, \cdots)\in \Gamma(\Sigma_{\epsilon},N\Sigma_{\epsilon}^{\frac{1}{2}})$. Notice that multiplying $\p_{t}\Phi$ by $W^{-1}(\tilde{g})$ will not change the $\zeta^{\frac{1}{2}}$ coefficients,  we conclude
			\begin{equation*}
				\begin{split}
					&\epsilon^{\tau-2}\big(\underline{A}^{\epsilon}(\p_{t}s_{\epsilon})\big)_{0}=\mathrm{Re}\bigg(\frac{3}{2}(B(\alpha)+\underline{B}_{0}(s_{\epsilon}))v_{0}-\frac{1}{2}P_{0}(\underline{A}(s_{\epsilon})\cdot \bs{v})-\frac{1}{2}\IP{\mu}{v_{0}}\underline{A}_{0}(s_{\epsilon})\bigg),\\
					&\epsilon^{\tau-2}\big(\underline{A}^{\epsilon}(\p_{t}s_{\epsilon})\big)_{q}=\mathrm{Re}\bigg(\frac{3}{2}(B(\alpha_{q})+\epsilon^{-2}\underline{B}_{q}(s_{\epsilon}))v_{q}-\frac{1}{2}P_{q}(\underline{A}(s_{\epsilon})\cdot \bs{v})-\frac{1}{2\epsilon^{2}}\IP{\mu}{v_{q}}\underline{A}_{q}(s_{\epsilon})\bigg).
				\end{split}
			\end{equation*}
		We can now write the derivative of $\bs{F}^{\epsilon}$ as
			\begin{equation}\label{S_4_Eqn_SigmaD}
				D\bs{F}^{\epsilon}(\bs{v})=\frac{3}{2}\big(\bs{B}+\underline{B}^{\epsilon}(s_{\epsilon})\big)\cdot \bs{v}-\frac{1}{2}P(\underline{A}(s_{\epsilon})\cdot \bs{v})-\frac{\epsilon^{2-\tau}}{2}\bs{F}^{\epsilon}\cdot\IP{\mu}{\bs{v}}.
			\end{equation}
		Here, let $\bs{a}, \bs{b}$ be multiples, then $\bs{a}\cdot\bs{b}$ denotes the multiple $(a_{0}b_{0},a_{1}b_{1},\cdots)$, and $\bs{B}=(B(\alpha),\cdots, B(\alpha_{q}),\cdots )$ is the multiple defined by the asymptotic behavior of the approximate solution. Suppose now, $\tilde{g}=g_{\epsilon}+\epsilon^{2-\tau}\eta$, the linear map $D\bs{F}^{\epsilon}(\bs{v})$ is approximately $\frac{3}{2}(\bs{B}+\underline{B}^{\epsilon})\cdot\bs{v}$ with $\bs{F}^{\epsilon}$-quadratic errors. We define the approximate inverse to be
			\begin{equation}\label{S_4_Eqn_Approx}
				V(\bs{F}^{\epsilon},\bs{v})(\bs{a}):= \frac{2}{3}(\bs{B}+\underline{B}^{\epsilon})^{-1}\cdot \bs{a} 
			\end{equation}
			and the $\bs{F}^{\epsilon}$-quadratic errors to be
			\begin{equation*}
				\begin{split}
					&Q_{1}(\bs{F}^{\epsilon},\bs{v}):=-\big(\bs{B}+\underline{B}_{\epsilon}\big)^{-1}\cdot\frac{1}{3}P(\underline{A}(s_{\epsilon})\cdot \bs{v})-\frac{\epsilon^{2-\tau}}{3}\big(\bs{B}+\underline{B}_{\epsilon}\big)^{-1}\cdot\bs{F}^{\epsilon}\cdot\IP{\mu}{\bs{v}},\\
					&Q_{2}(\bs{F}^{\epsilon},\bs{a}):=-\frac{1}{3}P(\underline{A}(s_{\epsilon})\cdot \big(\bs{B}+\underline{B}_{\epsilon}\big)^{-1}\cdot\bs{a})-\frac{\epsilon^{2-\tau}}{3}\bs{F}^{\epsilon}\cdot\IP{\mu}{\big(\bs{B}+\underline{B}_{\epsilon}\big)^{-1}\cdot\bs{a}}.
				\end{split}
			\end{equation*}
		\begin{lemma}\label{S_4_Lemma_Tame_V}
			In the region $\|v\|_{3m}\leq \delta$, the following estimates hold
				\begin{equation*}
					\begin{split}
						&\|Q_{1}(v,\bs{F}^{\epsilon})(v_{1})\|_{2m}\leq C \|\bs{F}^{\epsilon}\|_{3m}\|v_{1}\|_{3m},\\
						&\|Q_{2}(v,\bs{F}^{\epsilon})(a_{1})\|_{2m}\leq C \|\bs{F}^{\epsilon}\|_{3m}\|a_{1}\|_{3m},\\
						&\|V(v,\bs{F}^{\epsilon})(a_{1})\|_{k}\leq C \|a_{1}\|_{k+m/2},\forall 0 \leq k\leq m,\\
						&\|V(v,\bs{F}^{\epsilon})(a_{1})\|_{k}\leq C_{k}\big( \|a_{1}\|_{k+m}+\|v\|_{k+m}\|a_{1}\|_{2m}\big),\forall  k\geq m.
					\end{split}
				\end{equation*}
		\end{lemma}
		\begin{proof}
			Tame estimates in Proposition \ref{S_4_Prop_AB_map} together with the interpolation formula show that
				\begin{equation*}
					\begin{split}
						&\|\underline{A}(s_{\epsilon})\cdot \bs{v}\|_{C^{3m}_{\tau,\epsilon}}\leq C\|\bs{F}^{\epsilon}\|_{C^{3m}_{\tau-2,\epsilon}}\|\bs{v}\|_{C^{3m}_{\tau-2,\epsilon}},\\
						&\|\underline{B}^{\epsilon}(s_{\epsilon})\|_{C^{k}_{0,\epsilon}}\leq C_{k}\big((\epsilon^{(n-2+\tau)\sigma}+\|v\|_{C^{k+m}_{\tau-2,\epsilon}}\big).\\
					\end{split}
				\end{equation*}
			When $m$ is large and $\epsilon$ is sufficiently small, straightforward computation yields the quadratic estimates for $Q_{1},Q_{2}$ and the linear estimates for $\|V(a)\|_{k}$ when $0\leq k\leq m$. Moreover, the tame estimates for $\|V(a)\|_{k},k\geq m$ are a direct consequence of the interpolation formula.	
		\end{proof}
	\subsection{Gluing}%%%%%%%%
		With the above setting, the Nash-Moser-Hamilton-Zehnder implicit function Theorem \ref{S_App_Thm_HM} proves the main Theorem \ref{S_1_Thm_main}. When $\epsilon$ is sufficiently small, there exists a section $v\in C^{\infty}_{\tau-2,\epsilon}(\Sigma_{\epsilon}, N\Sigma_{\epsilon})$ that solves the equation
			\begin{equation*}
				\bs{F}^{\epsilon}(v)=0,
			\end{equation*}
		and satisfies $\|v\|_{m}\leq C\epsilon^{(n-2+\tau)\sigma}$. This is equivalent to finding a family of diffeomorphisms $\Phi_{\epsilon}$, which is equal to the identity on $M_{2d}$ and is generated by a vector field $\tilde{v}$, whose restriction to $\Sigma_{\epsilon}$ is $\epsilon^{2-\tau}v\in \Gamma(\Sigma_{\epsilon},N\Sigma_{\epsilon})$. We find a solution $s_{\epsilon}(\tilde{g})$ to Poisson's Equation \ref{S_3_Eqn_Poisson} with respect to the pullback metric
			\begin{equation*}
				\tilde{g}=(\Phi_{\epsilon})^{*}g.
			\end{equation*}
		In what follows is that $\alpha_{\epsilon}=(\Phi_{\epsilon}^{-1})^{*}\big(\tilde{\alpha}_{\epsilon}+\ud (W(\tilde{g})s_{\epsilon})\big)$ define a $\Z_{2}$-harmonic $1$-form. The difference between the approximate solution and $\alpha_{\epsilon}$ is measured by $s_{\epsilon}$. Moreover, due to the uniqueness of the solution to
			\begin{equation*}
				\bs{F}^{\epsilon}(v)=0,
			\end{equation*}
			the diffeomorphism $\Phi_{\epsilon}$ and the section $s_{\epsilon}$ are independent of the large integer $m$. We can pick a sequence of $m_{i}\to \infty$, and the algorithm in Theorem \ref{S_App_Thm_HM} implies as $\epsilon\to 0$
			\begin{equation*}
				\|\Phi_{\epsilon}-Id\|_{C^{m_{i}}_{\tau-1,\epsilon}}, \|s_{\epsilon}\|_{\mathcal{C}^{m_{i}+2}_{\tau,\epsilon}}\leq C_{i}\epsilon^{(n-2+\tau)\sigma},
			\end{equation*}
			which proves the smooth convergence. Moreover, with respect to the conformal metric $\rho_{\epsilon}^{-2}g_{\epsilon}$, $\alpha_{\epsilon}$ is asymptotic to
			\begin{equation*}
				\begin{split}
					&\alpha_{\epsilon}=\ud \big((B(\alpha)+\underline{B}^{\epsilon}_{0}(s_{\epsilon}))\zeta^{\frac{3}{2}}+Err\big) \text{ on } \Phi_{\epsilon}^{-1}U_{0,2r},\\
					&\epsilon^{-2}\alpha_{\epsilon}=\ud \big((B(\alpha_{q})+\underline{B}^{\epsilon}_{q}(s_{\epsilon}))\zeta^{\frac{3}{2}}+Err\big) \text{ on } \Phi_{\epsilon}^{-1}U_{q,2r}.
				\end{split}
			\end{equation*}
		Here, by choosing $m$ large enough, we can show $\|\underline{B}^{\epsilon}\|_{C^{1}_{0,\epsilon}}\leq C\epsilon^{(n-2+\tau)\sigma}$ and therefore when $\epsilon$ is sufficiently small, $\alpha_{\epsilon}$ is non-degenerate.
		\begin{corollary}\label{S_4_Cor_Nonzero}
			When $M$ is a compact $3$-manifold and $\alpha$ is a non-degenerate $\Z_{2}$-harmonic $1$-form such that every ordinary zero is regular, then there exists a one parameter family of non-degenerate $\Z_{2}$-harmonic $1$-forms $\alpha_{\epsilon}$ that converges to $\alpha$ and only admits branching zeros. 
		\end{corollary}
		\begin{proof}
			At each $p\in \mathcal{R}$, a regular zero, we glue in the non-degenerate $\Z_{2}$-harmonic $1$-forms $\alpha_{p}$ constructed in Proposition \ref{S_3_Prop_Construction}. In particular, each $\alpha_{p}$ only admits branching zeros. Therefore, when $\epsilon$ is sufficiently small, $\alpha_{\epsilon}$ has no ordinary zero.
		\end{proof}

\section{Application and discussion}\label{S_5}%%%%%%%%%%%%

	\subsection{Non-degenerate $\Z_{2}$-harmonic $1$-form on compact manifolds with $b^{1}>0$}\label{S_5_1}%%%%%%%%
		In this subsection, we will prove an existing theorem for non-degenerate $\Z_{2}$-harmonic $1$-forms using our gluing result.
			\begin{prop}\label{S_5_Prop_Existence}
				Let $M^{n}, n\geq 3$ be a compact oriented manifold, with $b^{1}>0$. There exists a Riemannian metric $g$ such that a sequence of non-degenerate $\Z_{2}$-harmonic $1$-forms $\alpha_{\epsilon}$ with respect to $g$ converges to a harmonic $1$-form that contains regular zeros of index $1$ or $n-1$.
			\end{prop}
			\begin{remark}\label{S_5_RMK_Existence}
				The strategy is to construct a harmonic $1$-form with regular zeros of index $1$ and $n-1$ on $M^n$. In particular, when $n=3$ and $M^3$ are not a surface bundle over $S^{1}$, for a generic choice of metric $g$ and class $[\beta]\in H^{1}(M)$, there exists a harmonic representative $\beta\in [\beta]$ with respect to the metric $g$, such that $\beta$ admits regular zeros. Then we can apply the gluing result to obtain a sequence of non-degenerate $\Z_{2}$-harmonic $1$-forms.
			\end{remark}
		
		To prove Proposition \ref{S_5_Prop_Existence}, we will use the following proposition due to Calabi \cite{Calabi69}. The idea is to adapt Milnor's birth lemma of nondegenerate critical points to the context of Calabi's intrinsic characterization of harmonic $1$-forms on Riemannian manifolds.
			
			\begin{prop}\label{S_5_Lemma_Calabi}[c.f. Section IV, \cite{Calabi69}]
				Let $(M^{n},g)$ be a Riemannian manifold, and let $U$ be an open set in $M^{n}$ on which $\ud x$ is a harmonic $1$-form on $U$. There exists a small ball $U'\subset U$ such that, in $U'$, we can perturb the $1$-form into $\ud h$, such that $\ud h$ is harmonic with respect to a new Riemannian metric $g'$ and has only two regular zeros of index $k, k+1$ for $ 0<k<n-1$.
			\end{prop}
			\begin{proof}
				We first construct a model solution in Euclidean space. Let $M>0$ be a sufficiently large constant, we construct a function $f$ on $\R^{3}=(x,y,z)$ with the following properties
					\begin{enumerate}
						\item for $r^2=x^{2}+y^{2}+z^{2}>M^{2}$, $f(x,y,z)=x$ identically;
						\item $f$ is an even function with respect to $y$ and $z$,
						\item $f$ has exactly $2$ nondegenerate critical points $(x,y,z)=(\pm c,0,0), c\ll M$ of index respectively $1$ and $2$.
					\end{enumerate}
				An example of such a function can be constructed as follows. Let $\gamma(t)\in C^{\infty}([0,\infty))$ be a non-increasing function, such that $\gamma(t)=1$ for $t\leq 1$, $\gamma(t)=0$ for $t\geq M$ and $|\gamma'(t)|\leq \frac{2}{M}$. We define
					\begin{equation*}
						f(x,y,z)=\frac{x(r^{2}-1)+(y^{2}-z^{2})\gamma(r)}{\gamma(r)+(r^{2}-1)(1-\gamma(r))}.
					\end{equation*}
				The resulting function satisfies the above properties with $c=3^{-1/2}$. Moreover, it is easy to see that $f^{-1}(\{a\})\cap \{r\geq 1\}$ is connected; therefore, for every $p\in \{r>1+\epsilon\}$, there exist smooth curves $\ell_{1},\ell_{2}$ such that
					\begin{equation*}
						\begin{split}
							&\ell_{i}: [0,1]\to B(0,2M), \frac{\ud}{\ud t}f(l_{i}(t))>0,\\
							&\ell_{1}(0)\in S(0,2M), l_{1}(1)=p,\\
							&\ell_{2}(0)=p, \ell_{2}(1)\in S(0,2M).
						\end{split}
					\end{equation*}
				Moreover, by joining $\ell_{1},\ell_{2}$ and smoothing the corner at $p$, we obtain a curve $\ell$, along which the function $f$ is increasing and passing through $p$.\\
				
				On the other hand, in the region $\{r\leq 1\}$, the function
					\begin{equation*}
						f=x^{3}-x+(x+1)y^{2}+(x-1)z^{2}.
					\end{equation*}
				Detail computation shows that for every $p\neq (\pm 3^{-1/2},0,0)$, we find a smooth curve that passes through $p$, with endpoints lying within $\{r>1+\epsilon\}$ and $f$, which is increasing along this curve.\\
			
				Let $(x_{1}',\cdots, x_{n}')$ be the coordinate in $\R^{n}$ and let
					\begin{equation*}
						\bs{y'}=(x_{2}',\cdots, x_{n-r}'), \bs{z'}=(x_{n-r+1}',\cdots,x_{n}'),
					\end{equation*}
				then $\tilde{f}(x_{1}',\bs{y}',\bs{z}'):=f(x_{1}',|\bs{y'}|,|\bs{z'}|)$ defines a function with exactly $2$ critical points
					\begin{equation*}
						p_{\pm}=(\pm 3^{-1/2},\bs{0},\bs{0})
					\end{equation*}
					of index respectively $r$ and $r+1$. We choose Morse coordinates on small neighborhoods $B(p_{\pm},3\delta)$ around $p_{\pm}$ such that
						\begin{equation*}
							\begin{split}
								&\tilde{f}= \tilde{f}(p_{-})-\sum_{j=1}^{r+1} a^{-}_{j} w_{j}^{2}+\sum_{l=r+2}^{n}a^{-}_{l}w_{l}^{2}, \text{ near } p_{-},\\
								&\tilde{f}=\tilde{f}(p_{+})-\sum_{j=1}^{r}a_{j}^{+}w_{j}^{2}+\sum_{l=r+1}^{n}a^{+}_{l}w_{l}^{2}, \text{ near } p_{+},
							\end{split}
						\end{equation*}
					where $a^{\pm}_{j}, a^{\pm}_{l}>0$ and
						\begin{equation*}
							\sum_{j=1}^{r+1}a^{-}_{j}=\sum_{l=r+2}^{n}a^{-}_{l}, \quad\sum_{j=1}^{r} a^{+}_{j}=\sum_{l=r+1}^{n}a_{l}^{+}.
						\end{equation*}
					Then $\tilde{f}$ is harmonic with respect to the standard Euclidean metric $\sum \ud w_{i}^{2}$ on $V_{\pm}$. We can find $n-2$-forms $\xi_{\pm}$ on $B(p_{\pm},3\delta)$ such that
						\begin{equation*}
							\ud \xi_{\pm}=\star \ud \tilde{f}, |\xi_{\pm}|=O(|w|^{2}).
						\end{equation*}
					Let $\gamma_{1}$ be a cutoff function such that $\gamma_{1}(t)=1$ when $|t|\leq 2$ and $\gamma_{1}(t)=0$ when $|t|>3$.
						Let $r_{w}^{2}=w_{1}^{2}+\cdots+w_{n}^{2}$, on the region $r_{w}<\delta$ we construct
						\begin{equation*}
							\eta_{\pm}'=\gamma_{1}\big(r_{w}/\delta\big) \xi_{\pm}
						\end{equation*}
					
					Let $p\in M^{n}$ and $U$ be sufficiently small neighborhoods of $p$, on which the metric is close to the Euclidean metric. Without loss of generality, we view $U$ as a small ball of radius $3\delta'$ in $\R^{n}=(x_{1},\cdots,x_{n})$ and let $\alpha=\ud x_{1}$. Then let $L:\R^{n}\to \R^{n}, (x_{1},\cdots,x_{n})\mapsto (Mx_{1}/\delta',\cdots,Mx_{n}/\delta')$. We perturb $\alpha$ in a ball of radius $\delta'$ into
						\begin{equation*}
							\beta=\ud \bigg(\delta' M^{-1}\tilde{f}\circ L\bigg),
						\end{equation*}
						and define $\eta_{\pm}=(\delta'/M)^{n-2}L^{*}(\eta_{\pm}')$. We also choose a form $\xi_{0}$ in the small ball such that
							\begin{equation*}
								\ud \xi_{0}=\star_{g}\ud x_{1},
							\end{equation*}
							and $|\xi_{0}|=O(|x|)$. We define a closed $n-1$ form
							\begin{equation*}
								\tilde{\zeta}=
                                \begin{cases}
                                    \ud \bigg(\gamma_{1}(|x|/\delta')\xi_{0}+\eta_{+}+\eta_{-}\bigg) & |x|\leq 2\delta'\\
                                    \star_{g} \alpha & |x|>2\delta'
                                \end{cases}
							\end{equation*}
						By construction, $\beta\wedge \tilde{\zeta}$ on $\{|x|>2\delta'\}$ and in a small punctured ball around each non-degenerate critical point of $\beta$.\\
						
						On the other hand, for any point
							\begin{equation*}
								p\in B(0,2\delta')\setminus\big(L^{-1}(B(p_{\pm},2\delta))\big),
							\end{equation*}
							one finds a smooth curve $\tilde{\ell}_{p}$, such that
							\begin{equation*}
								p\in \tilde{\ell}_{p}, \quad\p\tilde{\ell}_{p} \subset \p B(0,3\delta'),
							\end{equation*}
						along which $\IP{\tilde{\ell}_{p}'}{\beta}>0$. We can require that $\tilde{\ell}_{p}$ is the integral curve for the vector field $\frac{\p}{\p x_{1}}$ in $B(0,3\delta')\setminus B(0,(2+\epsilon)\delta')$. Moreover, since $\alpha$ is a harmonic $1$-form with respect to the metric $g$, transitivity implies that every two points on $\p B(0,3\delta')$ we can find a smooth curve $\gamma$ on $M\setminus B(0,3\delta')$, such that $\IP{\gamma'}{\alpha}>0$. Therefore the curve $\tilde{\ell}_{p}$ can be extended smoothly to $\ell_{p}\simeq S^{1}$, with $\IP{\ell_{p}'}{\beta}>0$.

                        For each curve $\ell_{p}$, we define $U_{p}$ to be a tubular neighborhood that does not intersect $L^{-1}(B(p_{\pm},\delta))$ and a map
							\begin{equation*}
								\pi_{p}:U_{p}\simeq \ell_{p}\times B^{n-1}(0,1)\to B^{n-1}(0,1).
							\end{equation*}
						Let $\xi_{p}$ be a non-negative $n-1$-form on $B^{n-1}(0,1)$, such that $\xi_{p}$ is equal to the Euclidean volume form on $B^{n-1}(0,\frac{1}{3})$ and vanishes outside $B^{n-1}(0,\frac{2}{3})$. It is easy to see that $\beta\wedge\pi_{p}^{*}\xi_{p}>0$ is in a slightly smaller tubular neighborhood. Suppose $\{U_{p}\}$ is a finite open cover of $\overline{B(0,2\delta')}\setminus L^{-1}B(p_{\pm},2\delta)$, we pick a sufficiently large constant $N$ such that
							\begin{equation*}
								\beta\wedge \bigg(\tilde{\zeta}+N\sum_{p}\pi_{p}^{*}\xi_{p}\bigg)>0,
							\end{equation*}	
						away from the non-degenerate critical points. By linear algebra, we can find a Riemannian metric $g'$ on $M$ such that $g'$ is equal to the Euclidean metric in Morse coordinates near two non-degenerate critical points and $\tilde{\zeta}+N\sum_{p}\pi_{p}^{*}\xi_{p}=\star_{g'}\beta$. To conclude, $\beta$ is a harmonic $1$-form with respect to $g'$ and agrees with $\alpha$ outside of $B(0,2\delta')$ and we obtain the desired harmonic $1$-form and Riemannian metric.
			\end{proof}

	\subsection{Resolving regular zeros of Joyce-Karigiannis construction}\label{S_5_2}%%%%%%%%
		In this subsection, we will discuss examples of torsion-free $G_{2}$-manifolds that can be used to generalize the Joyce-Karigiannis construction of torsion-free $G_{2}$-manifold \cite{Joyce21}. We begin with some basic definitions.\\
		
		Let $M^{7}$ be an oriented $7$-dimensional manifold. A $3$-form $\varphi\in \Omega^{3}$ is called a $G_{2}$-structure if it defines a metric
				\begin{equation*}
					g_{\varphi}(v,w):=\frac{(u\llcorner \varphi)\wedge (v\llcorner \varphi)\wedge \varphi}{V},
				\end{equation*}
		where $V\in \Omega^{7}$ is a nowhere vanishing section, normalized so that $|\varphi|_{g_{\varphi}}^{2}=7$. A $G_{2}$-structure $\varphi$ is called torsion-free if
				\begin{equation*}
					\ud \varphi=\ud \star_{g_{\varphi}}\varphi=0.
				\end{equation*}
		This condition implies $\varphi$ is parallel with respect to $g_{\varphi}$. Therefore, the holonomy group $\mathrm{Hol}_{g}(p)\subset G_{2}$ is the stabilizer of $\varphi\big|_{p}\in \Lambda^{3}T_{p}^{*}M$.\\

		In the following, we will take $\varphi$ to be a torsion-free $G_{2}$ structure. A $3$-dimensional submanifold $X^{3}\subset M^{7}$ is called associative if it is calibrated by $\varphi$. In other words, $\varphi\big|_{X^{3}}=\mathrm{Vol}_{g_{\varphi}}\big|_{X^{3}}$.
		
		\begin{example}\label{S_5_Ex_G2}
			Let $(Y^{6},\omega, \Omega)$ be a Calabi-Yau $3$-fold. Then
				\begin{equation*}
					\varphi=\omega\wedge \ud \theta+\RRe{\Omega}
				\end{equation*}
			defines a torsion-free $G_{2}$-structure on $Y^{6}\times S^{1}$, where $\theta$ is the coordinate in $S^{1}$. If $X^{3}$ is a special Lagrangian in $Y^{6}$, calibrated by $\RRe{\Omega}$, then $X^{3}\times \{\theta_{0}\}$ is an associative submanifold in $Y^{6}\times S^{1}$.
		\end{example}

		Suppose $\iota : M^{7}\to M^{7}$ is an involution that preserves $\varphi$, and the fixed point set $X^{3}$ is an associative submanifold. Suppose further that there is a nowhere vanishing harmonic $1$-form $\alpha$ on $X^{3}$ with respect to $g_{X^{3}}$. Then, we have the following gluing theorem.
			\begin{theorem}\label{S_6_Thm_JK}[Theorem 1.1 \cite{Joyce21}]
				Let $(M^{7},\varphi,\tau,X^{3})$ be described as above. Then, there exists a compact $7$-manifold $\widetilde{M}^{7}$ that defines a resolution of singularities $P: \widetilde{M}^{7}\to M^{7}/\langle \iota \rangle$ along its singular locus $X^{3}\subset M^{7}/\langle \iota \rangle$ by gluing a bundle $E\to X^{3}$ along $X^{3}$ with fiber the Eguchi-Hanson space. The preimage of $Q:=\pi^{-1}(X^{3})$ is a $5$-dimensional submanifold of $\widetilde{M}$, and $P_{Q}: Q\to X^{3}$ is a fiber bundle with the smooth fiber $S^{2}$.
				
				There exists a one parameter family of torsion-free $G_{2}$-structures $\widetilde{\varphi}_{t}, t\in (0,\epsilon]$ on $\widetilde{M}$, for $\epsilon$ sufficiently small, such that $(\widetilde{\varphi}_{t},g_{\widetilde{\varphi}_{t}})$ converges to $P^{*}(\varphi,g_{\varphi})$ in $C^{0}$ away from $Q$ as $t\to 0$. For each $x\in X$, the preimage $P_{Q}^{-1}(x)\simeq S^{2}$, with metric $g_{\varphi}\big|_{P_{Q}^{-1}(x)}$, approximates a round $2$-sphere with area $\pi t^{2}|\alpha|_{x}$.
			\end{theorem}
			
		Morally speaking, a nowhere vanishing $1$-form together with the $G_{2}$-structure defines a complex structure on the normal bundle of $X^{3}$. For each $x\in X^{3}$, the orbifold quotient acts as $\CC^{2}\to \CC^{2}/\Z_{2}$ in the normal direction. As in the Eguchi-Hanson space, the resolution of singularities $P$ in Theorem \ref{S_6_Thm_JK} is defined by  the blowing up $\CC^{2}/\Z_{2}$ at $0$ fiberwise	. Moreover, the $5$-dimensional submanifold $Q$ is the union of the exceptional divisors for the blow-up.\\
		
		When the harmonic $1$-form $\alpha$ admits only regular zeros, it is expected in Section $8$, (vi) in \cite{Joyce21} that, topologically, near the regular zeros, the resolved manifold looks like a cone over $\mathbb{CP}^{3}$.  To resolve these singularities analytically, we need to construct a one parameter family of singular torsion-free $G_{2}$-structures $\varphi_{s},s\in (0,1]$ on $\R^{+}\times \mathbb{CP}^{3}$, such that near the cone singularity, $\varphi_{s}$ is asymptotic to the $G_{2}$-cone over $\mathbb{CP}^{3}$ as established by Bryant and Salamon \cite{Bryant89}, and the tangent cone of $g_{\varphi_{s}}$ at infinity is $\R^{3}\times (\CC^{2}/\Z_{2})$. This might give rise to a family of compact $G_{2}$-manifolds with isolated conical singularities. Moreover, it is expected that for each $s$, there is a coassociative fibration $\pi_{s}$ from $(\R^{+}\times \mathbb{CP}^{3},\varphi_{s})$ to $\R^{3}$, with generic fibers being the Eguchi-Hanson space. Furthermore, when $|\pi_{s}|$ is sufficiently large, the metric $g_{\varphi_{s}}$ on the $S^{2}$, as stated in Theorem \ref{S_6_Thm_JK}, is approximately a round $2$-sphere with area $\pi|\alpha_{s}|$. Here, $\alpha_{s}=(1+s)x_{3}\ud x_{3}-x_{1}\ud x_{1}-sx_{2}\ud x_{2}$ is a harmonic $1$-form on $\R^{3}$.\\
		
		This conjectural torsion-free $G_{2}$-structure can be thought of as a $G_{2}$-analogue of the singular warped quasi-asymptotically conical Calabi-Yau metric on the affine cone $V_{0}\subset \CC^{4}:z_{1}^{2}+\cdots+z_{4}^{2}=0$, which was constructed by Colon and Rochon \cite{Conlon23}. However, due to the lack of Yau's theorem and Donaldson-Sun theory on $G_{2}$-manifolds, the most feasible approach to constructing such a family of torsion-free $G_{2}$-structures is to use symmetric reduction via a cohomogeneity-$4$ $SO(3)$ action described as follows. Consider $\mathbb{CP}^{3}=S(\Lambda^{+}S^{4})$ as the twistor space of the $4$-sphere and let $S^{4}\subset \R^{5}=\R^{3}\oplus \R^{2}$. Identify $SO(3)$ as $SO(3)\times Id_{\R^{2}}$ in $SO(5)$, we define a $SO(3)$-action on $S^{4}$. Then $SO(3)$ acts on $\mathbb{CP}^{3}$ by lifting its action on $S^{4}$. It is possible to use the "multi-moment map," described in \cite{Kar21}, to further reduce the torsion-free conditions to a system of PDEs in $3$-variables. However, finding a solution is still a difficult task.\\
		 
		On the other hand, our gluing Theorem for non-degenerate $\Z_{2}$-harmonic $1$-forms can transform the above problem into a slightly different gluing problem, where there is a model space to resolve the singularities. Let $(\widetilde{\alpha}, \Sigma, L)$ be a nondegenerate $\Z_{2}$-harmonic $1$-form, admitting only branching zeros. One way to obtain such a $\Z_{2}$-harmonic $1$-form is to use the gluing theorem for nondegenerate $\Z_{2}$-harmonic $1$-forms. Suppose $\alpha$ is a harmonic $1$-form with only regular zeros. Let $\alpha_{\epsilon}$ be the family of nondegenerate $\Z_{2}$-harmonic $1$-forms constructed from $\alpha$ in \ref{S_3_Prop_Construction}. When $\epsilon$ is sufficiently small, $\alpha_{\epsilon}$ will only admit branching zeros.\\
		
		We now speculate on how the gluing theorem in \cite{Joyce21} can be generalized to resolve the orbifold $M^{7}/\langle \iota \rangle$ when the associative submanifold $X^{3}$ admits a nondegenerate $\Z_{2}$-harmonic $1$-form with only branching zeroes. As suggested in Section 6.6 of \cite{Joyce21}, the gluing construction for $G_{2}$-structure still works if we replace the harmonic $1$-form with a harmonic section in $T^{*}X\otimes Z$. Here, $Z$ is a principal $\Z_{2}$-bundle over $X$. Locally, multiplied by $-1$, the approximate $G_{2}$-structures are unchanged. Therefore, away from the branching set, the gluing ansatz proceeds exactly the same as in Section 6.6 of \cite{Joyce21}. The key difference is the construction of the approximate solution near the branching set $\Sigma$. The normal bundle of $\Sigma$ in $M$ can be decomposed as
			\begin{equation}\label{S_5_Eqn_Decomp}
				N_{M/\Sigma}=N_{X/\Sigma}\oplus N_{M/X}.
			\end{equation}
		The $G_{2}$-structure and the tangent vector of $\Sigma$ define a complex structure $J$ on $N_{M/\Sigma}$. Moreover, $N_{X/\Sigma}$ and $N_{M/X}$ are trivial complex subbundles, since $X$ is associative and hence $J$ preserves the composition. Again, we may use Fermi coordinates to identify a small neighborhood of the zero section in $N_{M/\Sigma}$ as a tubular neighborhood of $\Sigma$ in $M$. It should be noted here that the complex structure on $N_{M/X}$ induced by $J$ is different from the complex structure induced by the $\Z_{2}$-harmonic $1$-form in the overlap region.\\
		
		In the $G_{2}$-orbifold $M/\langle \iota \rangle$, the neighborhood of $\Sigma$ is modeled on
			\begin{equation*}
				\Sigma\times \CC\times (\CC^{2}/\Z_{2}).
			\end{equation*}
		Following the papers \cite{Sze20,Yan24,Chiu24} , there is a unique complete Calabi-Yau K\"ahler form on $\CC^{3}$ up to scaling, whose tangent cone of the induced metric at infinity is $\CC\times (\CC^{2}/\Z_{2})$. We choose a particular scaling of the K\"ahler form, which we denote as $\omega_{\CC^{3}}$, such that
			\begin{equation*}
				\frac{\omega_{\CC^{3}}^{3}}{6}=\Omega\wedge \overline{\Omega}.
			\end{equation*} 
		Here, $\Omega=\frac{3}{4}\ud z_{1}\wedge \ud z_{2}\wedge \ud z_{3}$ is the holomorphic volume form.\\
		
		We briefly describe the asymptotic behavior of $\omega_{\CC^{3}}$ at infinity. Now let
			\begin{equation*}
				\pi: \CC^{3}\to \CC, (z_{1},z_{2},z_{3})\mapsto z:=z_{1}^{2}+z_{2}^{2}+z_{3}^{2},
			\end{equation*}
			be the Lefschetz fibration. The generic fibers $\pi^{-1}(z), z\neq 0$ are biholomorphic to the quadric
				\begin{equation*}
					\pi^{-1}(1): z_{1}^{2}+z_{2}^{2}+z_{3}^{2}=1,
				\end{equation*}
			which is diffeomorphic to $T^{*}S^{2}$, while $\pi^{-1}(0)$ is the only singular fiber, which is biholomorphic to $\CC^{2}/\Z_{2}$. We define
			\begin{equation*}
				\phi:=c\big(|z|^{2}+\sqrt{H+|z|}\big), \quad H:=|z_{1}|^{2}+|z_{2}|^{2}+|z_{3}|^{2},
			\end{equation*}
		where $c$ is a normalized constant. The function $\phi$ defines a K\"ahler potential outside a large compact set in $\CC^{3}$ and
			\begin{equation*}
				\nabla_{g_{\omega_{\CC^{3}}}}^{k}\big(\omega_{\CC^{3}}-\sqrt{-1}\p \pb \phi\big)=O(\rho^{-3+\epsilon'-k}), \quad \rho^{2}=|z|^{2}+H^{1/2}, k\in \N
			\end{equation*}
		for any small positive constant $\epsilon'$ when $\rho$ is sufficiently large. The behavior is that asymptotically near infinity, $g_{\omega_{\CC^{3}}}$ looks like the semi-Ricci-flat metric on $\CC^{3}$, namely that restricted to the fibers it approximates the Stenzel metrics on the fibers $\pi^{-1}(z),z\in \CC$, and in the horizontal direction, it is dominated by the pullback of the Euclidean metric induced by $c\sqrt{-1}\ud z\wedge \ud \bar{z}$. Moreover, a rescaling of the metric $\omega_{\CC^{3}}$ amounts to a rescaling of the metric in the fiber direction up to a diffeomorphism of $\CC^{3}$. For a more detailed discussion, see Section 1 in \cite{Yan24}.\\
		
		The nondegenerate $\Z_{2}$-harmonic $1$-form is related in this context as follows. For each $z\neq 0$, we denote $S^{2}_{z}$ as the zero section of $\pi^{-1}(z)\simeq T^{*}S^{2}$. We integrate the $3$-form $\RRe{\Omega}$ along $S^{2}_{z}$ and obtain a nondegenerate $\Z_{2}$-harmonic $1$-form on $\CC$
			\begin{equation*}
				\lambda=\ud \RRe{z^{\frac{3}{2}}}=\int_{S^{2}_{z}}\RRe{\Omega}.
			\end{equation*}
		Note that the Lie group $SO(3)$ preserves the K\"ahler form as well as the fibers. We conclude that $\omega_{\CC^{3}}$ induces round metrics on each $S^{2}_{z}$. Suppose now the $\Z_{2}$-harmonic $1$-form
			\begin{equation*}
				\alpha\sim \ud \RRe {B(\theta)\zeta^{\frac{3}{2}}}
			\end{equation*}
		near the branching set $\Sigma$. We now construct a family of $G_{2}$-structures with small torsion from $\alpha$. We define 
			\begin{equation*}
				\varphi_{\alpha,t}:=\big(t^{2}B(\theta)\big)^{2/3}\omega_{\CC^{3}}\wedge\ud \theta+t^{2}\RRe{B(\theta)\Omega}
			\end{equation*}
		as an approximate torsion free $G_{2}$-structure near $\Sigma$ in $M^{7}$. In particular, the associated metric is
			\begin{equation*}
				g_{\alpha,t}=\big(t^{2}B(\theta)\big)^{2/3}g_{\omega_{\CC^{3}}}\oplus \ud \theta^{2},
			\end{equation*}
		and the Hodge dual of $\varphi_{\alpha,t}$ with respect to this metric is
			\begin{equation*}
				\psi_{\alpha,t}:=\frac{1}{2}\big(t^{2}B(\theta)\big)^{4/3}\omega_{\CC^{3}}+\IIm{t^{2}B(\theta)\Omega}\wedge \ud  \theta.
			\end{equation*}

		Follow Joyce's perturbation theorem (Chapter 11 in \cite{Joyce07}) of $G_{2}$-structure with small torsion. We find a closed $3$-form and a closed $4$-form that are close to $\varphi_{\alpha,t}$ and $\psi_{\alpha,t}$, respectively. Let
			\begin{equation*}
				\begin{split}
					&\widetilde{\omega}:=\sqrt{-1}\p \pb \big(B(\theta)^{2/3}\phi\big),\\
					&\widetilde{\Omega}:=\ud \bigg(B(\theta)\big(z_{1}\ud z_{2}\wedge \ud z_{3}-z_{2}\ud z_{1}\wedge \ud z_{3}+z_{3}\ud z_{1}\wedge \ud z_{2}\big)\bigg).
				\end{split}
			\end{equation*}
		We define the $3$-form and $4$-form to be
			\begin{equation*}
				\begin{split}
					\widetilde{\varphi_{\alpha,t}}:=t^{4/3}\widetilde{\omega}\wedge \ud \theta+\RRe{\widetilde{\Omega}},\\
					\widetilde{\psi_{\alpha,t}}:=\frac{1}{2}t^{8/3}\widetilde{\omega}+\IIm{t^{2}\widetilde{\Omega}}\wedge \ud \theta.
				\end{split}
			\end{equation*}
		The differences between those two $3$-forms and $4$-forms are bounded by
			\begin{equation*}
				\big|\varphi_{\alpha,t}-\widetilde{\varphi_{\alpha,t}}\big|_{g_{\alpha,t}}, \big|\psi_{\alpha,t}-\widetilde{\psi_{\alpha,t}}\big|_{g_{\alpha,t}}=O(t^{2/3}),
			\end{equation*}
		which tend to zero as $t\to 0$.\\
		
		The above discussion suggests that there is a numerous supply of compact torsion-free $G_{2}$ manifolds arising from the resolution of $G_{2}$ orbifolds, which generalizes Joyce and Karigiannis's construction.

\section*{Appendix A}\label{Appendix}
	In this section, we will give a proof of the abstract version of the Hamilton-Nash-Moser-Zehnder implicit function theorem. The idea of the proof follows \cite{Raymond89,Donaldson17A} with some modifications, a significant difference is that we only have an approximate inverse in the tangent spaces. Moreover, our conditions are slightly stronger in order to establish uniqueness and obtain an estimate for the solution.
		\begin{theorem}\label{S_App_Thm_HM}
			Let $X$ and $Y$ be graded Fr{\'e}chet spaces with a family of algebraic smoothing operators $\{S_{\theta}\}, \theta>0$. Let $U$ be an open set of $\bs{0}$ in $X$. Suppose there exists a smooth tame map
				\begin{equation*}
					F:U\to Y.
				\end{equation*}
			 Assume that there exists an even integer $m>0$, a real number $\delta>0$, and a sequence of positive numbers $C_{1},C_{2},C_{3},C_{k},k\geq m,$ such that the following holds: if $\|x\|_{3m}\leq \delta$, then
			\begin{equation*}
				\begin{split}
					&\|F(x)\|_{k}\leq C_{k}(1+\|x\|_{k+m}), \forall k\geq m\\
					&\|DF(X_{1})\|_{k}\leq C_{1}\|X_{1}\|_{k+m/2}, \forall 0\leq k\leq 2m\\
					&\|D^{2}F(X_{1},X_{2})\|_{k}\leq C_{2}\|X_{1}\|_{k+m/2}\|X_{2}\|_{k+m/2}, \forall 0\leq k\leq 2m
				\end{split}
			\end{equation*}
			Moreover, suppose further that there exists an operator $V(x,F): X\to Y$ such that
				\begin{equation*}
					\begin{split}
						& V\circ DF=Id_{X}+Q_{1}(x,F);\\
						& DF\circ V=Id_{Y}+Q_{2}(x,F);\\
						&\|Q_{1}(x,F)(X_{1})\|_{2m}\leq C_{3}\|F\|_{3m}\|X_{1}\|_{3m};\\
						&\|Q_{2}(x,F)(Y_{1})\|_{2m}\leq C_{3}\|F\|_{3m}\|Y_{1}\|_{3m};\\
						&\|V(x,F)(Y)\|_{k}\leq C_{4}\|Y\|_{k+m/2},\forall 0\leq k\leq m\\
						&\|V(x,F)(Y)\|_{k}\leq C_{k}(\|Y\|_{k+m}+\|x\|_{k+m}\|Y\|_{2m}), \forall k\geq m
					\end{split}
				\end{equation*}
		Then, if $\|F(0)\|_{2m}$ is sufficiently small, with a bound that depends only on $\delta$ and a finite number of $C_{i}$, there exists a unique $x\in X$ with $\|x\|_{3m}<\delta$ such that $F(x)=0$. Moreover, we have the estimate for the solution
			\begin{equation*}
				\|x\|_{m}\leq C\|F(0)\|_{2m},
			\end{equation*}
			for a constant $C$.
		\end{theorem}
		\begin{remark}
			The reason for requiring the degree of the maps to be $m/2$ in the case $0 \leq k \leq m$ is to allow sufficient flexibility for establishing the uniqueness and obtaining estimates for the solution $x$. This condition is easy to verify in practice by choosing $m$ sufficiently large.
		\end{remark}
		
		Our strategy is to modify the Newton iteration using a smoothing operator. We will consider the following algorithm.
	 	\begin{enumerate}
	 		\item Pick a sequence $\theta_{j}\to \infty$, such that $\theta_{0}\geq 2$ is sufficiently large and $\theta_{j+1}=\theta_{j}^{5/4}$. In particular, this construction implies
	 			\begin{equation*}
	 				\theta_{0}^{-1}\geq\sum_{j=0}^{\infty}\theta_{j}^{-3}.
	 			\end{equation*}
	 		\item Let $x_{0}=0$, and define $v_{j},x_{j+1}$ inductively
	 			\begin{equation*}
	 				v_{j}=-V(x_{j},F(x_{j}))(F(x_{j})), \quad x_{j+1}=x_{j}+S_{\theta_{j}}v_{j}.
	 			\end{equation*}
	 			In other words, $x_{j}=\sum_{l<j}S_{\theta_{l}}v_{l}$.
	 	\end{enumerate}
	 Then the solution $x$ is found as the limit of the sequence $\{x_{j}\}_{j\geq 0}$.
	 
	\begin{lemma}\label{S_App_Lemma_Unique}[Uniqueness]
		When $\delta$ is sufficiently small, the solution $F(x)=0$ in $\|x\|_{3m}\leq \delta$ is unique.
	\end{lemma}
	\begin{proof}
		Suppose $x,x'$ are two different solutions. Since $F(x)=F(x')=0$, we have
		\begin{equation*}
			V(F(x),x)DF(x)=Id
		\end{equation*}
		and
		\begin{equation*}
			x-x'=V(F(x),x)\bigg(\int_{0}^{1}D^{2}F(tx+(1-t)x')\big(x-x',x-x'\big)\ud t\bigg).
		\end{equation*}
    Then the tame estimate gives
		\begin{equation*}
			\|x-x'\|_{0}\leq C\|x-x'\|_{m}^{2}.
		\end{equation*}
		Then the interpolation formula implies
		\begin{equation*}
			\|x-x'\|_{0}\leq C'\|x-x'\|_{2m}\|x-x'\|_{0}.
		\end{equation*}
	Choosing $\delta < \frac{1}{2}C'$ implies $x-x'=0$.
	\end{proof}
	
	\begin{lemma}\label{S_App_Lemma_Estimate}[Estimate of the solution]
		Let $\delta$ and $x$ be as defined above, then there exists a constant $C$ such that
			\begin{equation*}
				\|x\|_{m}\leq C\|F(0)\|_{2m}
			\end{equation*}
	\end{lemma}
	\begin{proof}
		Using the approximate inverse, we have
			\begin{equation*}
				x=V(F(x),x)\big(F(x)-F(0)\big)+V(F(x),x)\bigg(\int_{0}^{1}D^{2}F(tx)\big(x,x\big)\ud t\bigg).
			\end{equation*}
		Therefore,
			\begin{equation*}
				\|x\|_{m}\leq C''\big(\|F(0)\|_{2m}+\|x\|_{2m}^{2}\big)\leq C'(\|F(0)\|_{2m}+\|x\|_{m}\|x\|_{3m}).
			\end{equation*}
		Choosing $\delta\leq \frac{1}{2C'}$, we have $\|x\|_{m}\leq C\|F(0)\|_{2m}$.
	\end{proof}
	 
	 \begin{lemma}\label{S_App_Lemma_1}
	 	With the assumption in Theorem \ref{S_App_Thm_HM}, for a sufficiently large $\theta_{0}$, if $\|F(0)\|_{2m}\leq \theta_{0}^{-4}$, there exist constants $D_{k}$ and $M$, such that for all $j$
	 		\begin{enumerate}
	 			\item $I_{(j)}$: $\|x_{j}\|_{3m}< \delta$ and $\|F(x_{j})\|_{2m}\leq \theta_{j}^{-4}$;
	 			\item $II_{(j)}$: $\|v_{j}\|_{3m+3}\leq M\theta_{j}^{-3}$;
	 			\item $III_{(j)}$: $\forall k\geq m$, we have $(1+\|x_{j+1}\|_{k+2m})\leq D_{k}\theta_{j}^{2m}(1+\|x_{j}\|_{k+2m})$.
	 		\end{enumerate}
	 \end{lemma}
	 \begin{proof}
	 The hypothesis $I_{(0)}$ and $III_{(0)}$ hold automatically by construction. We first show that $II_{(j)}$ can be deduced from $I_{(j)}$ and $III_{(j-1)}$
	 
	 \subsubsection*{Proof of $II_{(j)}$}
	 We use the tame estimate for the approximate inverse
	 	\begin{equation}\label{S_App_Eqn_v_1}
	 		\|v_{j}\|_{k}\leq C_{k}(\|F(x_{j})\|_{k+m}+\|x_{j}\|_{k+m}\|F(x_{j})\|_{2m}), k\geq m.
	 	\end{equation}
		For $k=m$ and using $I_{(j)}$, we get
		\begin{equation}\label{S_App_Eqn_v_2}
			\|v_{j}\|_{m}\leq C_{m}(1+\delta)\theta_{j}^{-4}.
		\end{equation}
		Then $II_{(j)}$ will be obtained by interpolating \ref{S_App_Eqn_v_2} and an estimate
		\begin{equation*}
			\|v_{j}\|_{K}\leq M_{1}\theta_{j}^{N},
		\end{equation*}
		for a fixed large constant $K$. To make a choice, we apply the tame estimate for $F(x_{j})$ to \ref{S_App_Eqn_v_1}
		\begin{equation}\label{S_App_Eqn_v_3}
			\|v_{j}\|_{k}\leq C_{k}(C_{k+m}+1)(1+\|x_{j}\|_{k+2m}).
		\end{equation}
		
		Since $III_{j'}$ holds for $\forall j'<j$. We now have
		\begin{equation}\label{S_App_Eqn_x_1}
			(1+\|x_{j}\|_{K+2m})\leq (D_{K}\theta_{0}^{-1})^{j}\theta_{j}^{8m+2}
		\end{equation}
		The interpolation formula gives
		\begin{equation*}
			\|v_{j}\|_{3m+3}\leq M_{2}(K,m)\|v_{j}\|_{m}^{\frac{K-3m-3}{K-m}}\|v_{j}\|_{K}^{\frac{2m+3}{K-m}}.
		\end{equation*}
		Choose $K=16m^{2}+37m+12$ and $\theta_{0}$ large enough such that $\theta_{0}\geq D_{K}$, then we get
		\begin{equation*}
			\|v_{j}\|_{3m+3}\leq M\theta_{j}^{-3},
		\end{equation*}
		where $M$ depends only on $K,m,\delta$.
		\begin{remark}
			We will see later that the choice of $D_{k}$ can be independent of the choice of $\theta_{0}$.
		\end{remark}
	
	\subsubsection*{Proof of $III_{(j)}$}
		This follows from the construction and tame estimates. Indeed, since $x_{j+1}=x_{j}+S_{\theta_{j}}v_{j}$,
		\begin{equation*}
			\|x_{j+1}\|_{k+2m}\leq \|x_{j}\|_{k+2m}+\|S_{\theta_{j}}v_{j}\|_{k+2m}\leq \|x_{j}\|_{k+2m}+C_{k,k+2m}\theta_{j}^{2m}\|v_{j}\|_{k}.
		\end{equation*}
		Here, $C_{k,k+2m}=e^{2m}$. Equation \ref{S_App_Eqn_v_3} and the induction hypothesis imply
		\begin{equation*}
			\begin{split}
				\theta_{j}^{2m}\|v_{j}\|_{k}\leq C_{k}(C_{k+m}+1)\theta_{j}^{2m}(1+\|x_{j}\|_{k+2m}).
			\end{split}
		\end{equation*}
		Now we can choose $D_{k}=1+C_{k,k+2m}C_{k}(C_{k+m}+1)$.
		
	\subsubsection*{Proof of $I_{(j+1)}$}
		The main difference between our case and that of X. Saint Raymond lies in the estimate of $F(x_{j})$. First, since $x_{j+1}=\sum_{l\leq j}S_{\theta_{l}}v_{l}$, then $\forall t\in [0,1]$
			\begin{equation*}
				\|x_{j}+tS_{\theta_{j}}v_{j}\|_{3m}\leq M\sum_{l\leq j}\theta_{j}^{-3}.
			\end{equation*}
		From the construction $\sum_{l\leq j}\theta_{j}^{-3}\leq \theta_{0}^{-1}$, we again choose $\theta_{0}$ large enough so that $M\theta_{0}^{-1}\leq \delta$. When $t=0$, this gives $\|x_{j+1}\|_{3m}\leq \delta$.\\
		
		We now give an estimate for $F(x_{j+1})$. Taylors' expansion yields to
			\begin{equation*}
				F(x_{j+1})=F(x_{j})+DF(x_{j})(S_{\theta_{j}}v_{j})+\int_{0}^{1}(1-t)D^{2}F(x_{j}+tS_{\theta_{j}}v_{j})(S_{\theta_{j}}v_{j},S_{\theta_{j}}v_{j})\ud t.
			\end{equation*}
		Since we have $v_{j}=-V(x_{j},F(x_{j}))(F(x_{j}))$, we have
			\begin{equation*}
				F(x_{j})=DF(x_{j})(-v_{j})-Q_{2}(x_{j},F(x_{j}))(F(x_{j})).
			\end{equation*}
		Write $F(x_{j})=F_{1}+F_{1}+F_{3}$ with
			\begin{equation*}
				\begin{split}
					&F_{1}=DF(x_{j})\big(S_{\theta_{j}}v_{j}-v_{j}\big),\\
					&F_{2}=\int_{0}^{1}(1-t)D^{2}F(x_{j}+tS_{\theta_{j}}v_{j})(S_{\theta_{j}}v_{j},S_{\theta_{j}}v_{j})\ud t,\\
					&F_{3}=-Q_{2}(x_{j},F(x_{j}))(F(x_{j})).
				\end{split}
			\end{equation*}
		We now have
			\begin{equation*}
				\begin{split}
					&\|F_{1}\|_{2m}\leq C_{1}\|S_{\theta_{j}}v_{j}-v_{j}\|_{3m}\leq C_{1}C_{3m,3m+3}\theta_{j}^{-3}\|v_{j}\|_{3m+3}\leq C_{1}Me^{3}\theta_{j}^{-6},\\
					&\|F_{2}\|_{2m}\leq C_{2}\|S_{\theta_{j}}v_{j}\|_{3m}^{2}\leq C_{2}\|v_{j}\|_{3m}^{2}\leq C_{2}M^{2}\theta_{j}^{-6}.
				\end{split}
			\end{equation*}
		Moreover, using the estimate for the quadratic error, we get
			\begin{equation*}
				\|F_{3}\|_{2m}\leq C_{3}\|F(x_{j})\|_{3m}^{2}.
			\end{equation*}
		We now bound $\|F(x_{j})\|_{3m}$ using the smoothing operator. Choose $\bar{\theta}_{j}=\theta_{j}^{1/m}$, choose $k=K+m=16m^{2}+38m+12$, we can interpolate
			\begin{equation*}
				\begin{split}
					\|F(x_{j})\|_{3m}&\leq \|S_{\bar{\theta}_{j}}F(x_{j})\|_{3m}+\|(1-S_{\bar{\theta}_{j}})F(x_{j})\|_{3m}\\
					&\leq C_{2m,3m}\bar{\theta}_{j}^{m}\|F(x_{j})\|_{2m}+C_{3m,K+m}\bar{\theta}_{j}^{2m-K}\|F(x_{j})\|_{K+m}\\
					&\leq e^{m}\theta_{j}^{-3}+C_{3m,K+m}C_{K+m}\bar{\theta}_{j}^{2m-K}\big(1+\|x_{j}\|_{K+2m}\big)\\
					&\leq e^{m}\theta_{j}^{-3}+C_{3m,K+m}C_{K+m}\theta_{j}^{-8m-34}.
				\end{split}
			\end{equation*}
		Gathering everything, we conclude that there exists a constant $M_{3}$ such that
			\begin{equation*}
				\|F(x_{j+1})\|_{2m}\leq M_{3}\theta_{j}^{-6}\leq (M_{3}\theta_{j}^{-5/9})\theta_{j+1}^{-3}.
			\end{equation*}
		When $\theta_{0}$ sufficiently large, $\|F(x_{j+1})\|_{2m}\leq \theta_{j+1}^{-3}$ and we complete the proof of Lemma \ref{S_App_Lemma_1}.
	\end{proof}
	
	With the above estimate, we can now give a bound for $\|v_{j}\|_{k}$ for every $k\geq m$.
	\begin{lemma}
		There exist constants $M_{k}, k\geq 2m$, such that $v_{k}$ satisfies
		\begin{equation*}
			\|v_{j}\|_{k}\leq M_{k}\theta_{j}^{-3}
		\end{equation*}
	\end{lemma}
	\begin{proof}
		We will again use the interpolation formula. Equation \ref{S_App_Eqn_v_2} and $III_{(j)}$ imply
			\begin{equation*}
				(1+\|x_{j+1}\|_{k+2m})\theta_{j+1}^{-8m-2}\leq (D_{k}\theta_{j}^{-1/2})(1+\|x_{j}\|_{k+2m})\theta_{j}^{-8m-2}.
			\end{equation*}
		Since $\theta_{j}\geq D_{k}$ for sufficiently large $k$, the sequence $(1+\|x_{j}\|_{k+2m})\theta_{j}^{-8m-2}$ is then bounded with respect to $j$. Then we obtain an estimate
			\begin{equation*}
				\|v_{j}\|_{k+2m}\leq D_{k}'\theta_{j}^{8m+2}. 
			\end{equation*}
		Let $\bar{\theta}_{j}=\theta_{j}^{\frac{1}{k-m}}$ and $N=(8m+5)(k-m)$, we have
			\begin{equation*}
				\begin{split}
					\|v_{j}\|_{k}&\leq \|S_{\bar{\theta}_{j}} v_{j}\|_{k}+\|(1-S_{\bar{\theta}_{j}})v_{j}\|_{k}\\
					&\leq C_{m,k}\theta_{j}\|v_{j}\|_{m}+C_{k,k+N}\bar{\theta}_{j}^{-N}\|v_{j}\|_{k+N}\\
					&\leq M_{k}\theta_{j}^{-3}.
				\end{split}
			\end{equation*}
	\end{proof}
		
	To complete the proof, we show that $\|S_{\theta_{j}}v_{j}\|_{k}\leq M_{k}\theta_{j}^{-3}$ exponentially decays. As a result, the sequence
		\begin{equation*}
			x_{j}=\sum_{l<j}S_{\theta_{j}}v_{j}
		\end{equation*}
		converges in each norm $\|\cdot\|_{k}, k\geq m$. We denote $x$ as the limit. In particular, we have
		\begin{equation*}
			\begin{split}
				\|F(x)\|_{2m}&=\|F(x_{j})\|_{2m}+\|\int_{0}^{1}DF((1-t)x_{j}+tx)(x-x_{j})\|_{2m}\\
				&\leq \|F(x_{j})\|_{2m}+C_{1}\|x-x_{j}\|_{3m}\xrightarrow{j\to \infty} 0.
			\end{split}
		\end{equation*}
	To conclude, $x$ is the unique solution to $F(x)=0$ in $\|x\|_{3m}\leq \delta$ with a bound
		\begin{equation*}
			\|x\|_{m}\leq C\|F(0)\|_{2m}.
		\end{equation*}

\section*{Appendix B}\label{AppendixB}
    In this section, we will sketch a proof of Proposition \ref{S_3_Prop_Construction}.
        \begin{proposition}
            For any positive numbers $h_{1},\cdots, h_{n-1}$, there exists a non-degenerate $\Z_{2}$-harmonic function $f_{\bs{h}}$ on $\R^{n}$, whose branching set is a codimension-$2$ ellipsoid 
				\begin{equation*}
					E_{\bs{h}}: \sum_{i=1}^{n-1}\frac{x_{i}^{2}}{h_{i}^{2}}=1, x_{n}=0,
				\end{equation*}
				and such that $\ud f_{\bs{h}}\neq 0$ outside $E_{\bs{h}}$, and that at a large distance, we can pick a single-valued branch of $f_{\bs{h}}$, on which
				\begin{equation*}
					f_{\bs{h}}=a_{0}-\sum_{i=1}^{n}a_{i}x_{i}^{2}+O'(|\bs{x}|^{2-n}).
				\end{equation*}
				Here, let $S(y)=\prod_{i=1}^{n-1}(y+h_{i}^{2})$ and $a_{i}$ be constants given by
				\begin{equation*}
					\begin{split}
						&a_{i}=\frac{\prod_{j=1}^{n-1}h_{j}}{2}\int_{0}^{\infty}\frac{\ud u}{(u^{2}+h_{i}^{2})\sqrt{S(u^{2})}}, 1\leq i\leq n-1;\\
						&a_{n}=-\frac{\prod_{j=1}^{n-1}h_{j}}{2}\int_{0}^{\infty}\frac{S'(u^{2})\ud u}{S(u^{2})^{3/2}};\\
						&a_{0}=\frac{\prod_{j=1}^{n-1}h_{j}}{2}\int_{0}^{\infty}\frac{\ud u}{\sqrt{S(u^{2})}}.
					\end{split}
				\end{equation*}
				Moreover, the map $(h_{1},\cdots, h_{n-1})\mapsto (a_{1},\cdots,a_{n-1})$ from $(\R_{+})^{n-1}$ to $(\R_{+})^{n-1}$ is bijective.
        \end{proposition}

    	We sketch the construction using Lawlor's necks. We will begin with a more general scheme of constructing $\Z_{2}$-harmonic $1$-forms from a certain family of special Lagrangians in $\CC^{n}$. Regard $T^{*}\R^{n}=(x_{k},y_{k}\frac{\p}{\p x_{k}})$ as $\CC^{n}$ such that $z_{k}=x_{k}+\sqrt{-1}y_{k}$ is the complex coordinate, and let $(\CC^{n},g,\omega,\Omega)$ be the flat Calabi-Yau structure, where $g$ is the Euclidean metric on $\CC^{n}$ and
			\begin{equation*}
				\omega=\sum_{k=1}^{n}\frac{\sqrt{-1}}{2}\ud z_{k} \wedge \ud \overline{z_{k}},\Omega=\ud z_{1} \wedge \cdots \wedge \ud z_{n}.
			\end{equation*}
		Moreover, the symplectic form $\omega=\ud \lambda$ is exact, with the Liouville form $\lambda=-\sum_{k=1}^{n}y_{k}\ud x_{k}$. Let $S$ be a Lagrangian in $\CC^{n}$, then the restriction of $-\lambda$ on $S$ defines a closed $1$-form on $S$, and $S$ is locally a graph from $\R^{n}$ to $T^{*}\R^{n}$ representing this form. We call $S$ an \textbf{exact} Lagrangian if $\lambda|_{C}$ is exact.\\
		
		Given a family of angles $\phi_{i}\in (0,\pi)$ parameterized by $t$, such that
			\begin{equation*}
				\begin{split}
					& \phi_{1}+\cdots+\phi_{n}=m\pi, \quad m=1,\cdots, n-1,\\
					& \phi_{k}= 2t^{-1}a_{k}+o(t^{-1}), \quad 1\leq k\leq n-m,\\
					& \phi_{l}=\pi-2t^{-1}a_{l}+o(t^{-1}),\quad n-m+1\leq l\leq n,
				\end{split}
			\end{equation*}
		and define a pair of special Lagrangian planes
			\begin{equation*}
				\Pi_{\bs{\phi}}^{\pm}=(e^{\pm\II \phi_{k}/2}x_{k}, e^{\II(\pi\pm\phi_{l})/2}x_{l}),  1\leq k\leq n-m<l\leq n.
			\end{equation*}
		If there exists a family of special Lagrangians $S_{t}$ such that
			\begin{enumerate}
				\item $S_{t}$ meets $x$-plane in a fixed compact codimension-$2$ submanifold $\Sigma$.
				\item  Each $S_{t}$ is invariant under the involution $(x,y)\mapsto (x,-y)$. In other words, $S_{t}$ can be viewed as the graph of a $2$-valued $1$-form $\alpha_{t}$, which branches along $\Sigma$.
				\item The $1$-form $t\alpha_{t}$ converges to $\alpha$ in $C^{\infty}_{loc}(\R^{n}\setminus \Sigma)$, with $|\alpha|,|\nabla \alpha|\in L^{2}_{loc}$.
				\item $S_{t}$ is asymptotic to the $\cup \Pi_{\bs{\phi}}^{\pm}$ at infinity,
			\end{enumerate}
		then the limit $\alpha$ is a $\Z_{2}$-harmonic $1$-form. On a single-valued branch on $\R^{n}\setminus B$
			\begin{equation*}
				\alpha=\ud \bigg(\sum_{k=1}^{n-m}a_{k}(x^{k})^{2}-\sum_{l=n-m+1}^{n}a_{l}(x^{l})^{2}+O'(|x|^{2-n})\bigg).
			\end{equation*}
		Moreover, let $r$ be the distance to $\Sigma$, if one can show $r^{-\frac{1}{2}}|\alpha|>0$ near the branching set, then $\alpha$ is non-degenerate.\\

		The only known examples of such a family of special Lagrangians are Lawlor's necks, which correspond to the case when $m=1$ or $m=n-1$. Following the above discussion, we can write down the parametrization of the graph of the $\Z_{2}$-harmonic $1$-form $\ud f_{\bs{h}}$ from Lawlor's necks. Let $Pr_{X}$, $Pr_{Y}$ be the projection to the $x$-plane and $y$-plane respectively, then the graph is 
			\begin{equation}\label{S_3_Eqn_Graph}
				\begin{split}
					& Pr_{X}=(w_{k}\sqrt{h_{k}^{2}+s^{2}},-w_{n}s),\\
					& Pr_{Y}=(w_{k}\beta_{k}(s)\sqrt{h_{k}^{2}+s^{2}},w_{n}(1-s\beta_{n}(s))),(w_{k},w_{n})\in S^{n-1}, s\in \R.
				\end{split}
			\end{equation}
			Here, let $C(y^{2})=\prod_{j=1}^{n-1}(1+h_{j}^{-2}y^{2})$, then
			\begin{equation*}
				\begin{split}
					& \beta_{k}(s)=\int_{0}^{s}\frac{\ud y}{(h_{k}^{2}+y^{2})\sqrt{C(y^{2})}},\\
					& \beta_{n}(s)=-\int_{0}^{s}\frac{(C(y^{2})-1)\ud y}{y^{2}(C(y^{2})+\sqrt{C(y^{2})})}.
				\end{split}
			\end{equation*}
\nocite{*}

\bibliographystyle{alpha}

\end{document}